\documentclass[english]{article}
\usepackage[applemac]{inputenc}
\usepackage{amsmath,amsfonts,amsthm}
\usepackage{graphicx}
\usepackage{tikz,tkz-tab}
\usepackage{hyperref,diagbox}
\usepackage{color}

\usepackage[normalem]{ ulem }   
\usepackage{soul}   
\newcommand{\comment}[1]{}   
\def\diag{{\mathrm {\tt diag}}}   
\def\cN{{\cal N}}   

\usepackage[overload]{empheq} 
\def\trace{{\mathrm{\tt trace}}} 

\newtheorem{thm}{Theorem}
\newtheorem{lem}{Lemma}
\newtheorem{prop}{Proposition}
\newtheorem{rem}{Remark}
\newcommand{\tcm}{\textcolor{black}}

\newcolumntype{N}{>{\centering\arraybackslash}m{.5in}}
\graphicspath{{./fig/}}

\begin{document}
\title{Implementation of Control Strategies for Sterile Insect Techniques}
\author{Pierre-Alexandre Bliman$^{1}$\footnote{corresponding author: pierre-alexandre.bliman@inria.fr}, Daiver Cardona-Salgado$^2$, Yves Dumont$^{3,4,5}$, Olga Vasilieva$^6$\\
$^1$ \small Sorbonne Universit\'e, Universit\'e Paris-Diderot SPC, \color{black} Inria, CNRS\\ \small Laboratoire Jacques-Louis Lions, \'equipe Mamba, \color{black}  Paris, France\\
$^2$ \small Universidad Aut\'{o}noma de Occidente, Cali, Colombia \\
$^3$ \small CIRAD, Umr AMAP, Pretoria, South Africa \\
$^4$ \small AMAP, Univ Montpellier, CIRAD, CNRS, INRA, IRD, Montpellier, France, \\
$^5$ \small University of Pretoria, Department of Mathematics and Applied Mathematics, South Africa\\ 
$^6$ \small Universidad del Valle, Cali, Colombia
}
\maketitle
\begin{abstract}
In this paper, we propose a sex-structured entomological model that serves as a basis for design of control strategies relying on releases of sterile male mosquitoes (\textit{Aedes spp}) and aiming at \tcm{elimination} of the wild vector population in some target locality. We \tcm{consider different types of releases (constant and periodic impulsive), providing necessary} conditions to reach \tcm{elimination}. 
\tcm{However,} the main part of the paper is focused on the study of the periodic impulsive control in 
\tcm{different} situations. When the size of wild mosquito population cannot be assessed in real time, we propose the so-called \textit{open-loop} control strategy that relies on periodic impulsive releases of sterile males with constant release size. Under this control mode, global convergence towards the mosquito-free equilibrium is proved on the grounds of sufficient condition that relates the size and frequency of releases. 
If periodic assessments (either synchronized with releases or more sparse) of the wild population size are available in real time, we propose the so-called \textit{closed-loop} control strategy, which is adjustable in accordance with reliable estimations of the wild population sizes. 
Under this control mode, global convergence to the mosquito-free equilibrium is proved on the grounds of another sufficient condition that relates not only the size and frequency of periodic releases but also the frequency of sparse measurements taken on wild populations. Finally, we propose a mixed control strategy that combines open-loop and closed-loop strategies. This control mode renders the best result, in terms of overall time needed to reach \tcm{elimination} and the number of releases to be effectively carried out during the whole release campaign, while requiring for a reasonable amount of released sterile insects.
\end{abstract}


\tableofcontents

\section{Introduction}
Since decades, the control of vector-borne diseases has been a major issue in Southern countries. It recently became a major issue in Northern countries too. Indeed, the rapid expansion of air travel networks connecting regions of endemic vector-borne diseases to Northern countries, the rapid invasion and establishment of mosquitoes population, like \textit{Aedes albopictus}, in Northern hemisphere have amplified the risk of Zika, Dengue, or Chikungunya epidemics\footnote{See, for instance, the most recent distribution map of \textit{Aedes albopictus} provided by ECDC (European Centre for Disease Prevention and Control, {\scriptsize \url{https://ecdc.europa.eu/en/publications-data/aedes-albopictus-current-known-distribution-june-2018}})}.

For decades, chemical control was the main tool to control or eradicate mosquitoes. Taken into account resistance development and the impact of insecticides on the biodiversity, other alternatives have been developed, such as biological control tools, like the Sterile Insect Technique (SIT).

Sterile Insect Technique (SIT) is a promising control method that has been first studied by E. Knipling and collaborators and first experimented successfully in the early 1950's by nearly eradicating screw-worm population in Florida. Since then, SIT has been applied on different pest and disease vectors (see \cite{SIT} for an overall presentation of SIT and its applications).

The classical SIT relies on massive releases of males sterilized by ionizing radiations. However, another technique, called the {\itshape Wolbachia} technique, is under consideration because {\itshape Wolbachia} \cite{W} is a symbiotic bacterium that infects many Arthropods, including some mosquito species in nature. These bacteria have many particular properties, including one that is very useful for vector control: the cytoplasmic incompatibility (CI) property \cite{Bourtzis2008,Sinkin2004}. CI can be used for two different control strategies:
\begin{itemize}
\item
{\itshape Incompatible Insect Technique (IIT):} males infected with CI-inducing {\itshape Wolbachia} produce altered sperms that cannot successfully fertilize uninfected eggs. This can result in a progressive reduction of the target population. Thus, IIT can be seen as equivalent to classical SIT.
\item
{\itshape Population Replacement (PR):} in this case, males and females, infected with CI-inducing {\itshape Wolbachia}, are released in a susceptible (uninfected) population, such that {\itshape Wolbachia}-infected females will produce more offspring than uninfected females. Because {\itshape Wolbachia} is maternally inherited, this will result in a population replacement by {\itshape Wolbachia}-infected mosquitoes (such replacements or invasions have been observed in natural population, see \cite{rasgon2003wolbachia} for the example of Californian {\itshape Culex pipiens}). Recent studies have  {\color{black} shown} that PR may be very interesting with \textit{Aedes aegypti}, shortening their lifespan (see for instance \cite{Schraiber2012}), or more interesting, cutting down their competence for dengue virus transmission \cite{Moreira2009}. However, it is also acknowledged that {\itshape Wolbachia} infection can have fitness costs, so that the introgression of {\itshape Wolbachia} into the field can fail \cite{Schraiber2012}.
\end{itemize}
Based on these biological properties, classical SIT and IIT (see \cite{dufourd2011,dufourd2013,dumont2012,Li2017,Li2015,Strugarek2018b} and references therein) or population replacement (see \cite{Campo2017,Campo2018,Farkas2017,Farkas2010,Fen.Solving,HugBri.Modeling,nadin2017,Schraiber2012,Strugarek2018} and references therein) have been modeled and studied theoretically in a large number of  {\color{black} papers, in} order to derive results to explain the success or not of these strategies using discrete, continuous or hybrid modeling approaches, temporal and spatio-temporal models. More recently, the theory of monotone dynamical systems \cite{Smith1995} has been applied efficiently to study SIT \cite{anguelov2012,Strugarek2018b} or population replacement \cite{Bliman2017,koiller2014} systems.

In this paper, we derive and study a dynamical system to model the release and elimination process for SIT/IIT. We analyze and compare \tcm{constant continuous}/periodic impulsive releases and derive {\color{black} conditions relating the sizes and frequency of the releases that are sufficient to ensure successful elimination}. Such conditions  {\color{black} enable} the design of SIT-control strategies with constant  {\color{black} or} variable number of sterile males to be released  {\color{black} that} drive the wild population of mosquitoes towards elimination. \tcm{Among all the previous strategies, we are also able to derive the (best) strategy that needs to release the least amount of sterile males to reach elimination. This can be of utmost importance for field applications.}

The outline of the paper is as follows. In Section 2, we first develop and briefly study a simple entomological model that describes  {\color{black} the} natural evolution of mosquitoes. Then, in Section 3, we introduce a \tcm{constant continuous} SIT-control and estimate the size of constant releases that ensures global \tcm{elimination} of wild mosquitoes in the target locality. In Section 4, periodic impulsive SIT-control  {\color{black} with constant impulse amplitude is considered,} and a sufficient condition relating the size and frequency of periodic releases is derived to ensure global convergence towards the mosquito-free equilibrium. This condition enables the design of open-loop (or feedforward) strategies that ensure mosquito elimination in finite time and without assessing the size of wild mosquito population. Alternatively, Section 5 is focused on the design of closed-loop (or feedback) SIT-control strategies,  {\color{black} which are achievable when} periodic estimations (either synchronized with releases or more sparse) of the wild population size are available in real time. 
 {\color{black} In such situation, the release amplitude is computed on the basis of these measurements.}
 Thorough analysis of the feedback SIT-control implementation mode leads to another sufficient condition to reach mosquito elimination. This condition relates not only the size and frequency of periodic releases but also the frequency of sparse measurements.
{\color{black} Finally, in Section 6 we propose a mixed control strategy for periodic impulsive SIT-control.
The latter is essentially based on the use of the smallest of the release values proposed by the previous open-loop and closed-loop strategies.
It turns out that this control mode renders the best result from multiple perspectives: in terms of overall time needed to reach \tcm{elimination} and of peak-value of the input control, but also in terms of total amount of released sterile insects and of number of releases to be effectively carried out during a whole SIT-control campaign}.
 The paper ends with numerical simulations highlighting the key features and outcomes of periodic impulsive SIT-control strategies (Section 7) followed by discussion and conclusions.

\paragraph{Notations}
For any $z\in\mathbb R$, define the decomposition $z=|z|_-+|z|_+$ in negative and positive parts, fulfilling:
\begin{equation}
\label{eq453}
|z|_- := \begin{cases}
z &\text{ if } z \leq 0\\
0 &\text{otherwise}
\end{cases}\qquad\text{ and }\qquad
|z|_+: = \begin{cases}
z &\text{ if } z \geq 0\\
0 &\text{otherwise}
\end{cases}
\end{equation}

\section{A sex-structured entomological model}

We consider the following 2-dimensional system to model the dynamics of mosquito populations.
It involves two state variables, the number of males $M$ and the number of females $F$.

\begin{subequations}
\label{one}
\begin{align}[left = \empheqlbrace\,]
\label{one-a}
\dot M& = r\rho F e^{-\beta(M+F)}-\mu_M M,\\
\label{one-b}
\dot F& = (1-r)\rho F e^{-\beta(M+F)}-\mu_F F.
\end{align}
\end{subequations}

All the parameters are positive, and listed in Table \ref{table0}. The model assumes that all females are equally able to mate.
It includes direct and/or indirect competition effect at different stages (larvae, pupae, adults), through the parameter $\beta$.
The latter may be seen as the ratio, $\frac{\sigma}{K},$ between $\sigma$, a quantity characterizing the transition between larvae and adults under density dependence and larval competition, and a carrying capacity $K$, typically proportional to the breeding sites capacity.
The primary sex ratio in offspring is denoted $r$, and $\rho$ represents the mean number of eggs that a single female can deposit in average per day.
Last, $\mu_M$ and $\mu_F$ represent, respectively, the mean death rate of male and female adult mosquitoes.
As a rule, it is observed that the male mortality is larger, and we assume throughout the paper that:
\begin{equation}
\label{eq600}
\mu_M \geq \mu_F.
\end{equation}

\begin{table}[h]
\centering
\[
\begin{array}{|r|l|l|}
\hline
\textbf{Parameter} & \textbf{Description} & \textbf{Unit} \\ \hline
r & \text{Primary sex ratio} & -- \\ \hline
\rho & \text{Mean number of eggs deposited per female per day} & \text{day}^{-1}  \\
\mu_M, \mu_F & \text{Mean death rates for male \& female per day}& \text{day}^{-1} \\
\beta & \text{Characteristic of the competition effect per individual} & --   \\
\hline
\end{array}
\]
\caption{Parameters of the sex-structured entomological model \eqref{one}}
\label{table0}
\end{table}

Existence and uniqueness of the solutions of the Cauchy problem for dynamical system \eqref{one} follow from standard theorems, ensuring continuous differentiability of the latter. System \eqref{one} is dissipative: there exists a bounded positively invariant set $\mathcal{D}$ with the property that, for any bounded set in $E \subset \mathbb{R}_+^2$, there exists $t^*=t(\mathcal{D},E)$ such that $\big( M(0),F(0)\big)\in E$ implies $\big( M(t),F(t) \big) \in \mathcal{D}$ for all $t>t^*$. The set $\mathcal{D}$ is called an absorbing set.
In our case, it may be taken, e.g., as:
\[
\mathcal{D}=\{(M,F)\ :\ 0\leq M\leq C,0\leq F\leq C\}
\]
for some $C>0$.
\begin{rem}
Population models of the form $\dot N=B(N)N-\mu N$ for several birth rate functions, including $B(N)=e^{-\beta N}$, have been studied in {\rm\cite{Cooke99}}.
Maturation delay can also be included {\rm\cite{Cooke99}}.
\end{rem}

Obviously $E^*_0=(0,0)$ is a trivial equilibrium of system \eqref{one}, called the {\em mosquito-free equilibrium}.
Being the state to which one desires to drag the system by adequate releases of sterile insects, it will play a central role in the sequel.
Denote for future use
\begin{equation}
\label{eqq2}
\cN_F:=\dfrac{(1-r)\rho}{\mu_F },\qquad \cN_M:=\dfrac{r\rho}{\mu_M }.
\end{equation}
These positive constants represent {\em basic offspring numbers} related to the wild female and male populations, respectively.
The first of them governs the number of equilibria, as stated by the following result, whose proof presents no difficulty and is left to the reader.

\begin{thm}[Equilibria of the entomological model]
\label{thh}
\mbox{}

\begin{itemize}
\item If $\cN_F<1$, then system \eqref{one} possesses $E^*_0$ as unique equilibrium.
\item If $\cN_F>1$, then system \eqref{one} also possesses a unique {\em positive} equilibrium $E^*:=(M^*,F^*)$, namely
\[
F^*=\dfrac{\cN_F}{\cN_F+\cN_M}\dfrac{1}{\beta} \ln \cN_F, \qquad
M^*=\dfrac{\cN_M}{\cN_F+\cN_M}\dfrac{1}{\beta} \ln \cN_F.
\]
\end{itemize}

\end{thm}

Notice that the total population at the nonzero equilibrium is given by $M^*+F^*=\dfrac{1}{\beta}\ln\cN_F$.
It depends upon the basic offspring number and the competition parameter $\beta$.
As an example, mechanical control through reduction of the breeding sites induces increase of $\beta$ and consequently decrease of the population at equilibrium.
Analogously, altering biological parameters modifies the basic offspring number, and therefore the size of the population.

The stability of the equilibria is addressed by the following result.

\begin{thm}[Stability properties of the entomological model]
\label{th0}
\mbox{}

\begin{itemize}
\item If $\cN_F<1$, then the (unique) equilibrium $E^*_0$ is Globally Asymptotically Stable (GAS).
\item If $\cN_F>1$, then $E^*_0$ is unstable, and $E^*$ is GAS in $\mathcal{D}\setminus{\{(M,0), M\in \mathbb{R}_+\}}$.
\end{itemize}
\end{thm}

Figure \ref{phase_portrait-Model1} shows the convergence of all trajectories to the positive equilibrium in a case where $\cN_F>1$ (the pertinent case for the applications we have in mind).

\begin{figure}[h]
\centering
 \includegraphics[scale=0.4]{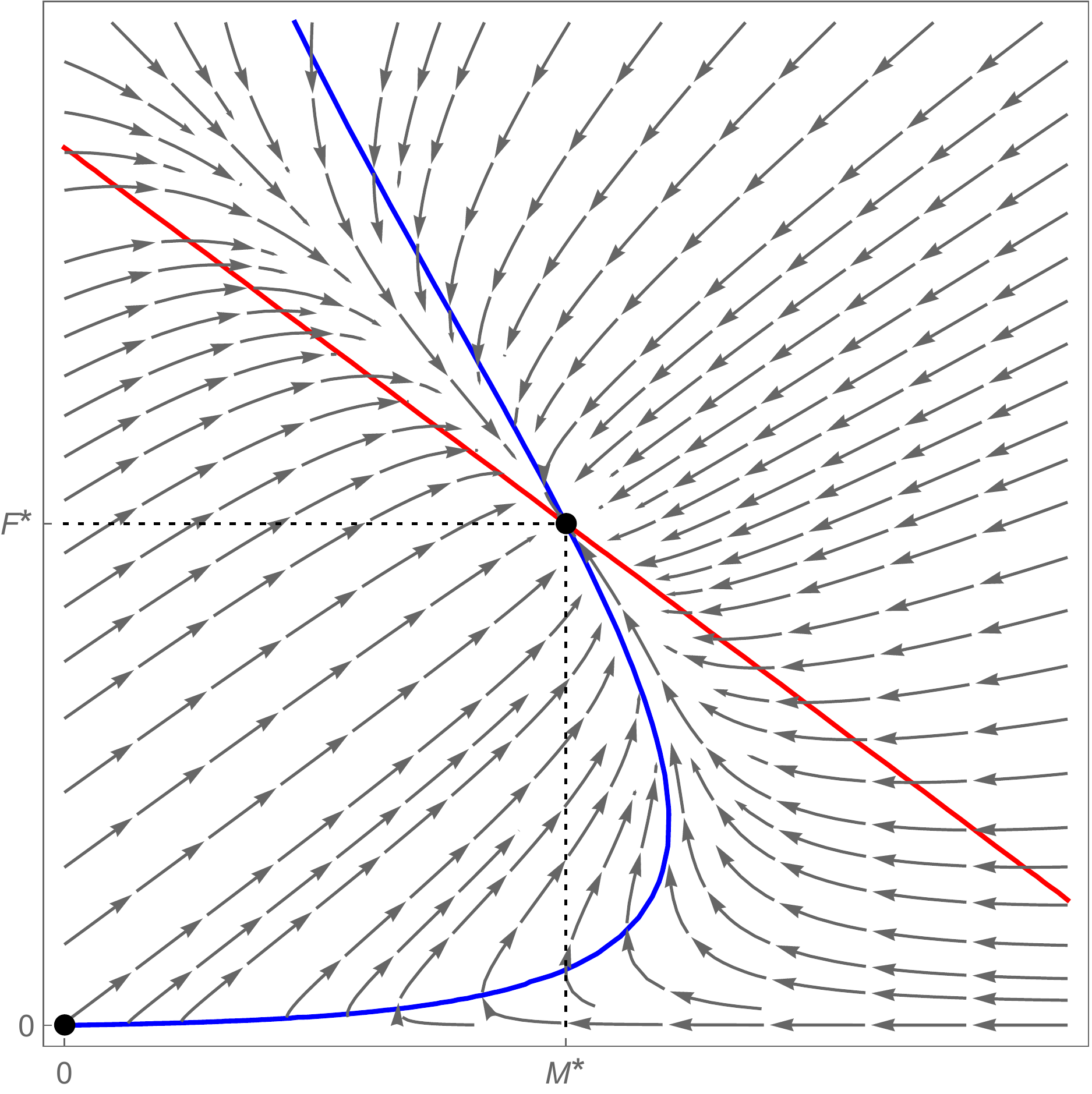}
 \caption{Phase portrait of model \eqref{one} when $\cN_F>1$.  The positive equilibrium appears at the intersection of the two curves on which $\dot F$ (in {\color{red} red}) and $\dot M$ (in {\color{blue} blue}) vanish.}
 \label{phase_portrait-Model1}
\end{figure}

\begin{proof}[Proof of Theorem \rm\ref{th0}]
\mbox{}

$\bullet$
Assume first $\cN_F<1$. Rewriting equation \eqref{one-b} as follows:
$$
\dot F=\left((1-r)\rho e^{-\beta(M+F)}-\mu_F \right) F \leq \big( (1-r)\rho -\mu_F \big) F
$$
one deduces that $\dot F<-\varepsilon F$ for some positive $\varepsilon$. The state variable $F$ being nonnegative, it then converges to $0$.
Using now equation \eqref{one-a}, we deduce that $M$ converges to $0$ too, and the GAS of $E^*_0$ follows.

$\bullet$
Assume now that $\cN_F>1$. Let us compute the Jacobian matrix related to entomological system \eqref{one}, page \pageref{one}:
\[
J(M,F)=\left(\begin{array}{cc}
- \beta r\rho Fe^{-\beta(M+F)}-\mu_{M} & r \rho (1-\beta F) e^{-\beta(M+F)}\\
& \\
-\beta (1-r)\rho Fe^{-\beta(M+F)} & (1-r)\rho (1-\beta F) e^{-\beta(M+F)} -\mu_{F}
\end{array}\right)
\]
{\color{black} so that}
\[ J(E_0^*)=\left(\begin{array}{cc}
-\mu_{M} & r\rho\\
0 & (1-r)\rho-\mu_{F}
\end{array}\right),
\]
from which we deduce that $E_0^*$ is locally asymptotically stable (LAS) if $\cN_F<1$.

{\color{black} For the} positive equilibrium $E^*$, using the fact that $e^{-\beta(M^*+F^*)}=\dfrac{1}{\cN_F}$, we have:
\[
J \big( E^{*} \big)=\left(\begin{array}{cc}
-\dfrac{\beta r \rho}{\cN_F}  F^{*} - \mu_{M} & \dfrac{r\rho}{\cN_F} \big(1-\beta F^{*}\big)\\
 & \\
-\dfrac{\beta (1-r)\rho }{\cN_F} F^{*} & -\dfrac{\beta (1-r)\rho }{\cN_F} F^{*}
\end{array}\right).
\]
Obviously $\trace\big\{J(E^{*})\big\}<0$ and
\[
\det J \big( E^{*} \big)=\dfrac{\beta}{\cN_F}(1-r)\rho F^{*} \left(\mu_{M}+\dfrac{r\rho}{\cN_F} \right)>0
\]
{\color{black} so that} $E^{*}$ is LAS when $\cN_F >1$.

Using Dulac criterion \cite{Perko}, we {\color{black} now} show that system \eqref{one} has no closed orbits wholly contained in the attracting set $\cal D$.
Indeed, setting
\[
\psi_1(F):=\dfrac{1}{F},\qquad f_1(M,F):=r\rho Fe^{-\beta(M+F)}-\mu_M M,\qquad g_1(M,F):=(1-r)\rho Fe^{-\beta(M+F)}-\mu_F F,
\]
let us study the sign of the function
$$
D_1(M,F):=\dfrac{\partial }{\partial M}\Big( \psi_1(F)f_1(M,F)\Big)+\dfrac{\partial}{\partial F} \Big(\psi(F)g_1(M,F)\Big).
$$
We have
$$
\dfrac{\partial }{\partial M}\Big(\psi_1(F)f_1(M,F)\Big)=-\beta r\rho e^{-\beta(M+F)}-\dfrac{\mu_M}F,
\qquad
\dfrac{\partial}{\partial F} \Big(\psi_1(F)g_1(M,F)\Big)=-\beta(1-r)\rho e^{-\beta(M+F)},
$$
and thus
$$
D_1(M,F)=-\beta\rho e^{-\beta(M+F)}-\dfrac{\mu_M}F<0
$$
for all $(M,F)\in \mathcal{D}$ such that $F>0$.
Therefore, Dulac criterion \cite{Perko} applies, demonstrating that system \eqref{one} possesses no nonconstant periodic solutions.
Thus, using the fact that $E^*$ is LAS, by the Poincar\'e-Bendixson theorem, all trajectories in $\mathcal{D}\setminus\{(M,0)\ :\ M\geq 0\}$ converge  towards $E^*$.

Convergence towards $E^*_0$ clearly occurs in absence of females, i.e.\ when $F(0)=0$. (Notice that for this reason, the point $E^*_0$ cannot be repulsive.) Consider on the contrary a trajectory such that $F(0)>0$. We will show that convergence to $E^*_0$ is impossible, so convergence towards $E^*$ occurs. First of all, one deduces from \eqref{one} and the continuity of $F$ that
\[
M(t) = e^{-\mu_Mt}M(0) + r\rho \int \limits_0^t e^{-\mu_M(t-s)}F(s)e^{-\beta(M(s)+F(s))}\ ds >0
\]
for any $t>0$.
The ratio $\dfrac{F}{M}$ is therefore {\color{black} well} defined and remains positive along this trajectory. It is moreover continuously differentiable, and
\[
\frac{d}{dt}\left(
\frac{F}{M}
\right)
= \frac{F}{M}
\left(
\mu_M-\mu_F +\rho e^{-\beta(M+F)}
\left(
1-r-r\frac{F}{M}
\right)
\right)
> (\mu_M-\mu_F) \frac{F}{M}
\qquad \text{ if }\:\: \frac{F}{M}\leq \frac{1-r}{r}.
\]
From \eqref{eq600}, it is deduced immediately that there exists for this trajectory a real number $T\geq 0$, such that
\[
\forall \: t\geq T,\qquad \frac{F}{M}> \frac{1-r}{r}.
\]

Then it holds for any $t\geq T$ that
\[
\dot F = \left(
(1-r)\rho e^{-\beta(M+F)}-\mu_F
\right) F
\geq \left(
(1-r)\rho e^{-\frac{\beta}{1-r}F}-\mu_F
\right) F
\]
and thus
\begin{equation}
\label{eq801}
\liminf_{t\to +\infty} F \geq \frac{1-r}{\beta} \ln\cN_F >0.
\end{equation}

As the compact set $\mathcal{D}$ is absorbing, the trajectory is bounded.
We deduce from this and \eqref{eq801} the existence of certain $\delta>0$ and $T'\geq T$ (whose precise values depend upon the considered trajectory), such that
\[
\forall \: t\geq T',\qquad Fe^{-\beta F}\geq\delta >0.
\]

Now, we have for any $t\geq T'$
\[
\dot M \geq r\rho\delta e^{-\beta M}- \mu_M M,
\]
which is strictly positive in a neighborhood of $M=0$.
The trajectory under study therefore stays at a positive distance from the point $E^*_0$, and, being convergent, has to converge to the other equilibrium, namely $E^*$.
This shows that any trajectory departing with $F(0)>0$ converges towards $E^*$, and finally concludes the proof of Theorem \ref{th0}.
\end{proof}

\section{Elimination with constant releases of sterile insects}

We now extend system \eqref{one}, in order to incorporate {\em continuous, constant} releases driven by an equation for $M_S$, the number of {\em sterile} males:
\begin{subequations}
\label{system_SIT}
\begin{align}[left = \empheqlbrace\,]
\label{system_SIT-a}
\dot M &=r\rho\dfrac{FM}{M+\gamma M_S}e^{-\beta(M+F)}-\mu_M M,\\
\label{system_SIT-b}
\dot F &=(1-r)\rho\dfrac{FM}{M+\gamma M_S}e^{-\beta(M+F)}-\mu_F F, \\
\label{system_SIT-c}
\dot M_S &=\Lambda-\mu_S M_S.
\end{align}
\end{subequations}

The positive constants $\mu_S$ and $\gamma$ represent, respectively, the mortality rate of sterile insects, and their relative reproductive efficiency or fitness (compared to the wild males), which is usually smaller than 1.
The {\color{black} nonnegative} quantity $\Lambda$ is the number of sterile insects released per time unit.
It is taken constant over time in the present section. The other parameters are the same as for model \eqref{one}, see Table \ref{table0}.

The mortality of the sterile males is usually larger than that of wild males, so in complement to \eqref{eq600}, we also have:
\begin{equation}
\label{mus}
\mu_S \geq \mu_M.
\end{equation}

Assuming $t$ large enough, we may suppose $M_S(t)$ at its equilibrium value $M_S^*:=\dfrac{\Lambda}{\mu_S}$, and the previous system then reduces to
\begin{subequations}
\label{eq:3}
\begin{align}[left = \empheqlbrace\,]
\label{eq:3-a}
\dot M &=r\rho\dfrac{FM}{M+\gamma M_S^*}e^{-\beta(M+F)}-\mu_M M,\\
\label{eq:3-b}
\dot F &=(1-r)\rho\dfrac{FM}{M+\gamma M_S^*}e^{-\beta(M+F)}-\mu_F F .
\end{align}
\end{subequations}
System \eqref{eq:3} is dissipative too, with all trajectories converging towards the same set $\mathcal{D}$ introduced in the previous section.
It admits the same mosquito-free equilibrium $E^*_0$.

We are interested here in the issues of existence and stability of positive equilibria.
Driven by the application in view, we  assume that the mosquito population is viable (that is $\cN_F>1$, see Theorem \ref{thh}), and focus on conditions sufficient for its elimination.

\subsection{Existence of positive equilibria}

The mosquito-free equilibrium $E^*_0$ is always an equilibrium of system \eqref{eq:3}.
The following result is concerned with possible supplementary equilibria.

\begin{thm}[Existence of positive equilibria for the SIT entomological model with constant releases]
\label{tth1}
Assume $\cN_F>1$.
Then
\color{black}
\begin{itemize}
\item
there exists  $\Lambda^{crit}>0$ such that system \eqref{system_SIT} admits two positive distinct  equilibria if $0 < \Lambda < \Lambda^{crit}$, one positive equilibrium if $\Lambda = \Lambda^{crit}$, and no positive equilibrium if $\Lambda > \Lambda^{crit}$;
\item
the value of $\Lambda^{crit}$ is uniquely determined by the formula
\begin{equation}
\label{eeq5}
\Lambda^{crit} := 2 \: \frac{\mu_S}{\beta\gamma} \: \frac{\phi^{crit}(\cN_F)}{1+\frac{\cN_F}{\cN_M}},
\end{equation}
where $\phi^{crit}:=\phi^{crit}(\cN_F)$ is the unique positive solution to the equation
\begin{equation}
\label{eeq1}
1+ \phi\left(
1+ \sqrt{1+\frac{2}{\phi}}
\right)
= \cN_F \exp\left(
-\frac{2}{1+ \sqrt{1+\dfrac{2}{\phi}}}
\right).
\end{equation}
\end{itemize}
\end{thm}

Theorem \ref{tth1} provides a characterization of the constant release rate above which no positive equilibrium may appear.
We prove in the next section (Section \ref{se32}) that in such a situation, convergence towards the mosquito-free equilibrium $E^*_0$ occurs, that ensures elimination of the wild population.

\begin{proof}[Proof of Theorem \rm\ref{tth1}]
\mbox{}

\noindent
Clearly, nullity of $M$ at equilibrium is equivalent to nullity of $F$.
In order to find possible nonzero equilibria, let $(M^*,F^*)$ with $ M^*>0, F^*>0$ be one of them.
The populations at equilibrium have to fulfill:
\begin{equation*}
r\rho\dfrac{F^*}{M^*+\gamma M_S^*}e^{-\beta(M^*+F^*)}=\mu_M ,\qquad
(1-r)\rho\dfrac{M^*}{M^*+\gamma M_S^*}e^{-\beta(M^*+F^*)}=\mu_F .
\end{equation*}
In particular, we have, for $\cN_F, \cN_M$ defined in \eqref{eqq2},
\begin{equation}
\label{eq:SIT_equil}
\dfrac{M^*}{M^*+\gamma M_S^*}e^{-\beta(M^*+F^*)}=\dfrac{1}{\cN_F},\qquad
\dfrac{F^*}{M^*+\gamma M_S^*}e^{-\beta(M^*+F^*)}=\dfrac{1}{\cN_M},
\end{equation}
which imply the relation:
\[
\text{\ensuremath{\dfrac{F^*}{M^*}}=\ensuremath{\dfrac{\cN_F}{\cN_M}}.}
\]
Injecting this value in the first equation of \eqref{eq:SIT_equil}, the number of males $M^*$ at equilibrium has to verify the equation
\[
\dfrac{M^*}{M^*+\gamma M_S^*}e^{-\beta\left(1+\frac{\cN_F}{\cN_M}\right)M^*}=\dfrac{1}{\cN_F}
\]
or again
\begin{equation}
\label{eqq3}
1+\dfrac{\gamma M_S^*}{M^*}=\cN_Fe^{-\beta\left(1+\frac{\cN_F}{\cN_M}\right)M^*}.
\end{equation}

\noindent
The study of equation \eqref{eqq3} is done through the following result, whose proof is given in Appendix.
\begin{lem}
\label{le0}
Let $a,b,c$ be positive constants, with $b>1$.
Then the equation
\begin{equation}
\label{eeq0}
1+ \phi\left(
1+ \sqrt{1+\frac{2}{\phi}}
\right)
= b \exp\left(
-\frac{2}{1+ \sqrt{1+\dfrac{2}{\phi}}}
\right)
\end{equation}
admits a unique positive root, denoted $\phi^{crit}$.
Moreover, the equation
\begin{equation}
\label{eeq2}
f(x):= 1+\frac{a}{x} - be^{-cx}=0
\end{equation}
admits two positive distinct roots if $0<ac<2\phi^{crit}$; one positive root if $ac=2\phi^{crit}$; no positive root otherwise.
\end{lem}

Using Lemma \ref{le0} with
\[
a:=\gamma M_S^*=\gamma \dfrac{\Lambda}{\mu_S},\qquad {\color{black} b:=\cN_F >1}, \qquad c:= \beta\left(
1+\frac{\cN_F}{\cN_M}
\right),
\]
one deduces that equation \eqref{eqq3} admits exactly one positive root when the root of \eqref{eeq1} is equal to $\phi^{crit}
=\frac{ac}{2}
=\frac{1}{2}\beta\gamma \left(
1+\frac{\cN_F}{\cN_M}
\right)\dfrac{\Lambda^{crit}}{\mu_S}$,
which implies \eqref{eeq5} and thus achieves the proof of Theorem \ref{tth1}.
\end{proof}

\subsection{Asymptotic stability of the equilibria}
\label{se32}
Assume $\cN_F>1$.
We first study the asymptotic stability of the mosquito-free equilibrium $E^*_0$ in the case where it is the unique equilibrium, that is when $\Lambda>\Lambda^{crit}$.

\begin{thm}[Stability of the mosquito-free equilibrium of the SIT entomological model with constant releases]
\label{thh2}
If system \eqref{system_SIT} admits no positive equilibrium (that is, if $\Lambda>\Lambda^{crit}$), then the mosquito-free equilibrium $E^*_0$ is globally exponentially stable.
\end{thm}
\begin{proof}[Proof of Theorem {\rm\ref{thh2}}]
The Jacobian matrix of the reduced system \eqref{eq:3} is equal to
\[ J(M,F)=
\begin{pmatrix}
\dfrac{r \rho F}{M+\gamma M_S^*}e^{-\beta(M+F)}\left(1-\beta M - \dfrac{M}{M+\gamma M_S^*} \right) - \mu_M  & \dfrac{r \rho M}{M+\gamma M_S^*} e^{-\beta(M+F)}(1-\beta F)\\
& \\
\dfrac{(1-r)\rho F}{M+\gamma M_S^*}e^{-\beta(M+F)}\left(1-\beta M-\dfrac{M}{M+\gamma M_S^*}\right) & \dfrac{(1-r)\rho M}{M+\gamma M_S^*} e^{-\beta(M+F)} \left( 1-\beta F \right) - \mu_F
\end{pmatrix}.
\]
Its value at the mosquito-free equilibrium $E^*_0$ is just $\diag\{-\mu_M;-\mu_F\}$, {\color{black} which} guarantees local asymptotic stability at this point.

We use again Dulac criterion to show that system \eqref{eq:3} has no closed orbits wholly contained in the set $\mathcal{D}$.
We set
\begin{gather*}
\psi_2(M,F):=\dfrac{M+\gamma M_S^*}{MF},\\
f_2(M,F):=r\rho \dfrac{FM}{M+\gamma M_S^*}e^{-\beta(M+F)}-\mu_M M,\qquad
g_2(M,F):=(1-r)\rho \dfrac{FM}{M+\gamma M_S^*}e^{-\beta(M+F)}-\mu_F F,
\end{gather*}
and then study the sign of the function
$$
D_2(M,F):=\dfrac{\partial }{\partial M} \Big( \psi_2(M,F)f_2(M,F) \Big) + \dfrac{\partial}{\partial F} \Big( \psi_2(M,F)g_2(M,F)\Big).
$$
As
$$
\dfrac{\partial }{\partial M} \Big(\psi_2 (M,F) f_2(M,F) \Big) = - \beta r \rho e^{-\beta(M+F)} - \frac{\mu_M}{F},\qquad
\dfrac{\partial}{\partial F} \Big( \psi_2 (M,F) g_2(M,F) \Big) = - \beta (1-r) \rho e^{-\beta(M+F)},
$$
one has
$$
D_2 (M,F) = - \beta \rho e^{-\beta(M+F)} - \frac{\mu_M}{F} <0,
$$
for all $(M,F) \in \mathcal{D}$ such that $F>0$. Thus, by the Poincar\'e-Bendixson theorem, since $E^*_0$ is the only asymptotically stable equilibrium, all trajectories in $\mathcal{D}$ approach the equilibrium $E^*_0$. This concludes the proof of Theorem \ref{thh2}.
\end{proof}


On the other hand,  when $\Lambda < \Lambda^{crit}$ is not large enough and system \eqref{eq:3} admits two distinct positive equilibria $ E^*_1 <E^*_2$, one may show by studying the spectrum of the Jacobian matrices that  $E^*_0=(0,0)$ and $E^*_2=(M_2^*,F_2^*)$ are locally asymptotically stable. It is likely that this case presents bistability and that $E^*_1=(M_1^*,F_1^*)$ is unstable, with the basin of attraction of $E^*_0$ containing the interval $[0, E_1^*) := \{(M,F)\in \mathbb{R}_+^2\ :\ 0\leq M < M^*_1,\ 0\leq F < F^*_1\}$, and the basin of attraction of $E_2^*$ containing the interval $(E_2^*, \mathbf{\infty}):= \{(M,F)\in \mathbb{R}_+^2\ :\  M > M^*_2,\ F > F^*_2\}$. This is at least what is suggested by the vector field illustrating this situation presented in Figure \ref{bistability}. It is worth noting that when $\Lambda \to \Lambda^{crit}$ from below, we have $ E^*_1 \to E^*_2$ and {\color{black} the} two positive equilibria merge.

\begin{figure}[h]
\centering
 \includegraphics[scale=0.4]{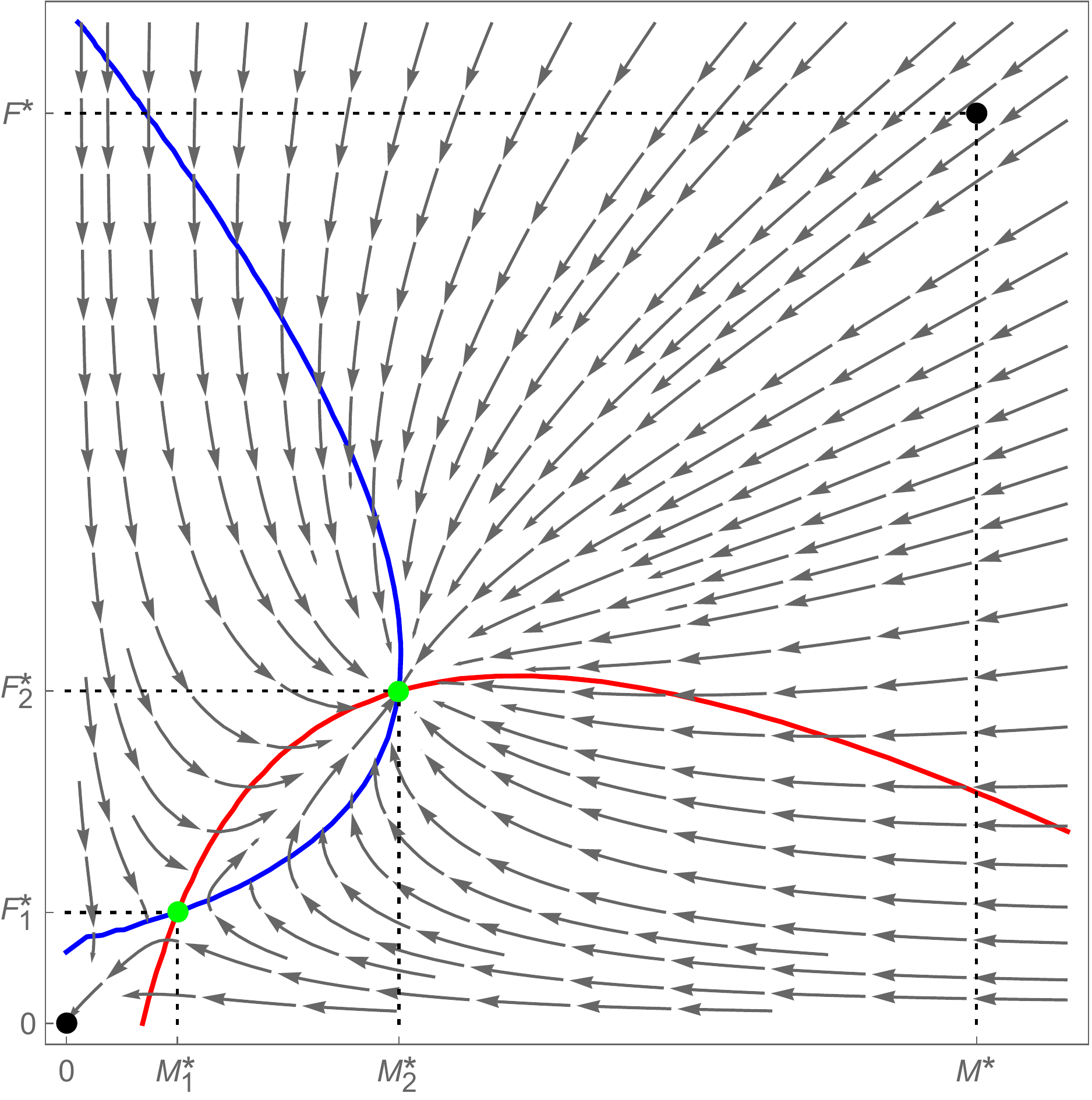}
 \caption{Trajectories of system \eqref{eq:3} related to each equilibria: bi-stable case.
 The two positive equilibria $E^*_1, E^*_2$ ({\color{green} green} points) are located at the intersection of the two curves where $\dot F$ (in {\color{red} red}) and $\dot M$ (in {\color{blue} blue}) vanish.
 The isolated black points denote the initial equilibria $E^*_0=(0,0), E^*=(M^*,F^*)$ of the system \eqref{one}, present when no control is applied.}
 \label{bistability}
\end{figure}

\section{Elimination with periodic impulsive releases of sterile insects}
\label{se4}

We now consider {\em periodic impulsive} releases $\Lambda(t)$, modeled by the following variant of system \eqref{system_SIT}:
\begin{subequations}
\label{system_SIT_imp}
\begin{align}[left = \empheqlbrace\,]
\label{system_SIT_imp-a}
\dot M &=r\rho\dfrac{FM}{M+\gamma M_S}e^{-\beta(M+F)}-\mu_M M,\\
\label{system_SIT_imp-b}
\dot F &=(1-r)\rho\dfrac{FM}{M+\gamma M_S}e^{-\beta(M+F)}-\mu_F F,\\
\label{system_SIT_imp-c}
\dot M_S &=-\mu_SM_S \: \text{ for any } \: t\in \bigcup_{n\in\mathbb{N}} \big( n\tau,(n+1)\tau \big),\\
\label{system_SIT_imp-d}
M_S(n\tau^{+}) &=\tau \Lambda_n+M_S(n\tau),\quad n=1,2,3, \ldots .
\end{align}
\end{subequations}
We choose in this section $\Lambda_n$ {\em constant}, and drop consequently the subindex $n$.
For such release schedule, it is clear that the function $M_S$ converges when $t\to +\infty$ towards the periodic solution
\begin{equation}
\label{eq601}
M_S^{\text{per}}(t)=\frac{\tau \Lambda e^{-\mu_S \left( t-\lfloor\frac{t}{\tau}\rfloor\tau \right)}}{1-e^{-\mu_S \tau}}.
\end{equation}
\comment{
In particular we have
\begin{equation}
\label{eqq11}
\dfrac{\tau \Lambda e^{-\mu_S\tau}}{1-e^{-\mu_S\tau}}=\underline{M}_s^{\text{per}}\leq M_S^{\text{per}}(t)\leq\overline{M}_s^{\text{per}}=\dfrac{\tau \Lambda}{1-e^{-\mu_S\tau}}
\end{equation}
}
We therefore introduce now the following {\em periodic system}:
\begin{subequations}
\label{eq:per}
\begin{align}[left = \empheqlbrace\,]
\label{eq:per-a}
\dot M &=r\rho\dfrac{FM}{\color{black} M+\gamma M_S^{\text{per}}(t)}e^{-\beta(M+F)}-\mu_M M,\\
\label{eq:per-b}
\dot F &=(1-r)\rho\dfrac{FM}{\color{black} M+\gamma M_S^{\text{per}}(t)}e^{-\beta(M+F)}-\mu_F F.
\end{align}
\end{subequations}
Existence and uniqueness of continuously differentiable solutions of system \eqref{eq:per} on the interval $[0,+\infty)$ may be shown by standard arguments, as well as the forward invariance of the positive orthant.
Notice that the mosquito-free equilibrium $E^*_0$ previously introduced is still an equilibrium of \eqref{eq:per}.
We are interested here in studying the conditions under which $E^*_0$ is globally asymptotically stable.
For future use, we note that the {\em mean value} of $1/M_S^{\text{per}}$ corresponding to \eqref{eq601} verifies:
\begin{equation}
\label{eqq12}
\left\langle
\dfrac{1}{M_S^{\text{per}}}
\right\rangle
:= \frac{1}{\tau} \int_{0}^{\tau}\dfrac{1}{M_S^{\text{per}}(t)}dt
=\frac{1-e^{-\mu_S\tau}}{\tau^2 \Lambda}\int_{0}^{\tau}e^{\mu_St}dt
=\frac{2\big( \cosh\left(\mu_S\tau\right)-1\big)}{\mu_S \tau^2 \Lambda}.
\end{equation}

\begin{thm}[Sufficient condition for elimination by periodic impulses]
\label{th3}
For any given $\tau>0$, assume that $\Lambda$ is chosen such that
\begin{equation}
\Lambda \geq \Lambda^{crit}_{per}= {\color{black} \frac{\cosh\left(\mu_S\tau\right)-1}{\mu_S \tau^2}} \: \frac{1}{e\beta\gamma}
\min\left\{ 2\cN_M, 2\cN_F, \max\{r,1-r\} \max\left\{ \frac{\cN_M}{r},\frac{\cN_F}{1-r} \right\} \right\}.
\label{Lambda_cond}
\end{equation}
Then every solution of system \eqref{eq:per} converges globally exponentially to the mosquito-free equilibrium $E^*_0$.
\end{thm}

The previous result provides a simple sufficient condition for stabilization of the mosquito-free equilibrium, through an adequate choice of the amplitude of the releases, $\Lambda$, for given period $\tau$.
\begin{rem}
{\color{black}
When $r=1-r$ and $\cN_F > \cN_M$ (which is the case of the application we are interested in), the expression of $\Lambda_{per}^{crit}$ simplifies as follows:
$$
\Lambda_{per}^{crit}=\frac{2 \big(\cosh\left(\mu_S\tau\right)-1\big)}{\mu_S\tau^2} \: \frac{\cN_F}{e\beta\gamma}.
$$
The function $\tau \mapsto \dfrac{2 \cosh\left(\mu_S\tau\right)-1}{\mu_S\tau^2}$ is increasing and tends towards $\mu_S$ when $\tau\to 0$.
Making $\tau \to 0_+$, we derive the following sufficient condition for stabilization:}
$$
\Lambda_{per}^{crit} \geq \frac{ \mu_S \cN_F}{e\beta\gamma},
$$
to be compared to $\Lambda^{crit} = 2\dfrac{\mu_S}{\beta\gamma} \frac{\phi^{crit}(\cN_F)}{1+\frac{\cN_F}{\cN_M}}$ (see Theorem {\rm\ref{tth1}}).

\comment{
\color{red} It seems plausible to compare the corresponding value to $\Lambda^{crit}$, and to have for the limit:
\[
\frac{\mu_S}{e\beta\gamma}
\min\left\{
2\cN_M, 2\cN_F, \max\{r,1-r\} \max\left\{
\frac{\cN_M}{r},\frac{\cN_F}{1-r}
\right\}
\right\}
\geq \Lambda^{crit}
= 2\frac{\mu_S}{\beta\gamma}
\frac{\phi^{crit}(\cN_F)}{1+\frac{\cN_F}{\cN_M}}.
\]
where $\Lambda^{crit}$ is the critical release rate in the case of constant releases (see Theorem \ref{tth1}).

}
\end{rem}

\begin{proof}[Proof of Theorem \rm\ref{th3}]
First rewrite \eqref{eq:per} as
\begin{subequations}
\label{ea1}
\begin{gather}
\label{ea1a}
\dot M = \left(
r \rho \dfrac{F}{M+\gamma M_S^{\text{per}} }e^{-\beta (M+F)}-\mu_{M}
\right) M, \\
\label{ea1b}
\dot F = \left(
(1-r)\rho\dfrac{M}{M+\gamma M_S^{\text{per}} }e^{-\beta (M+F)}-\mu_{F}
\right) F,
\end{gather}
\end{subequations}
in order to emphasize the factorization of $M$ and $F$.
\vspace{2mm}

\noindent $\bullet$ 1.\
Notice that, for any $M,F\geq 0$ and any $t\geq 0$,
\begin{equation}
\label{ea2}
\dfrac{M}{M+\gamma M_S^{\text{per}} }e^{-\beta (M+F)}
\leq \dfrac{M}{M+\gamma M_S^{\text{per}} }e^{-\beta M}
\leq \dfrac{\alpha}{M+\gamma M_S^{\text{per}} }
\leq \dfrac{\alpha}{\gamma M_S^{\text{per}} },\qquad
\end{equation}
where we write for simplicity
\begin{equation}
\label{ea15}
\alpha := \max \big\{ xe^{-\beta x}\ :\ x\geq 0 \big\} = \frac{1}{e\beta}.
\end{equation}
One then deduces from \eqref{ea1b} that, for any $n \in \mathbb{N}$,
\[
F \big( (n+1)\tau \big) \leq e^{\left(
(1-r)\rho \dfrac{\alpha}{\gamma}
\left\langle
\dfrac{1}{M_S^{\text{per}} }
\right\rangle - \mu_F
\right)\tau} F(n\tau).
\]
Therefore, the sequence $\big\{ F(n\tau) \big\}_{n \in \mathbb N}$ decreases towards $0$, provided that
\[
(1-r)\rho \dfrac{\alpha}{\gamma}
\left\langle
\dfrac{1}{M_S^{\text{per}} }
\right\rangle < \mu_F,
\]
that is,
\begin{equation}
\label{ea3}
\left\langle
\dfrac{1}{M_S^{\text{per}} }
\right\rangle <  \dfrac{\gamma}{\alpha} \frac{\mu_F}{(1-r)\rho} =  e\beta\gamma \frac{1}{{\cal N}_F}.
\end{equation}
This is sufficient to ensure that $F$ converges towards $0$, and this induces the same behavior for $M$: condition \eqref{ea3} implies that $E^*_0$ is GAS.
\vspace{2mm}

\noindent $\bullet$ 2.\
The same argument may be conducted from \eqref{ea1a} rather than \eqref{ea1b}, leading to:
\begin{equation}
\label{ea7}
\dfrac{F}{M+\gamma M_S^{\text{per}} }e^{-\beta (M+F)}
\leq \dfrac{F}{M+\gamma M_S^{\text{per}} }e^{-\beta F}
\leq \dfrac{\alpha}{M+\gamma M_S^{\text{per}} }
\leq \dfrac{\alpha}{\gamma M_S^{\text{per}} }
\end{equation}
Global asymptotic stability is thereby guaranteed if
\begin{equation}
\label{ea10}
\left\langle
\dfrac{1}{M_S^{\text{per}} }
\right\rangle <  \frac{\gamma}{\alpha} \frac{\mu_M}{r\rho} =  e\beta\gamma \frac{1}{{\cal N}_M}.
\end{equation}
\vspace{2mm}

\noindent $\bullet$ 3.\
Define the positive definite function
\begin{equation}
\label{V}
\mathcal{V}(M,F) := \frac{1}{2}(M^2+F^2)
\end{equation}
and write its derivative along the trajectories of \eqref{eq:per} as
\begin{equation}
\label{ea11}
\dot{\mathcal{V}}
= M\dot M+F\dot F
= -\mu_MM^2-\mu_FF^2
+ \rho\dfrac{FM(rM+(1-r)F)}{M+\gamma M_S^{\text{per}} }e^{-\beta (M+F)}.
\end{equation}
On the one hand, we have
\[
-\mu_MM^2-\mu_FF^2 \leq -\min\{ \mu_M,\mu_F\} (M^2+F^2) = -2 \min\{ \mu_M,\mu_F\} \mathcal{V}.
\]
On the other hand,
\begin{eqnarray*}
\dfrac{FM(r M+(1-r)F)}{M+\gamma M_S^{\text{per}} }e^{-\beta (M+F)}
& \leq &
\max\{r,1-r\}
\dfrac{FM(M+F)}{M+\gamma M_S^{\text{per}} }e^{-\beta (M+F)}\\
& \leq &
\max\{r,1-r\} \alpha
\dfrac{FM}{M+\gamma M_S^{\text{per}} }\\
& \leq &
\max\{r,1-r\} \alpha
\dfrac{1}{M+\gamma M_S^{\text{per}} } \mathcal{V}\\
& \leq &
\max\{r,1-r\} \alpha
\dfrac{1}{\gamma M_S^{\text{per}} } \mathcal{V}.
\end{eqnarray*}
Coming back to \eqref{ea11}, we deduce that
\[
\dot{\mathcal{V}} \leq \left(\max\{r,1-r\} \alpha \dfrac{1}{\gamma M_S^{\text{per}} }-2 \min\{ \mu_M,\mu_F\}\right)\mathcal{V}.
\]
One may conclude that $E^*_0$ is GAS provided that
\[
\max\{r,1-r\} \rho \frac{\alpha}{\gamma}
\left\langle \frac{1}{M_S^{\text{per}} } \right\rangle < 2 \min\{ \mu_M,\mu_F\},
\]
that is,
\begin{equation}
\label{ea5}
\left\langle \frac{1}{M_S^{\text{per}} } \right\rangle < 2 \frac{\gamma}{\alpha} \frac{\min\{ \mu_M,\mu_F\}}{\max\{r,1-r\}\rho}
= 2 e\beta\gamma \frac{1}{\max\{r,1-r\}}
\min\left\{
\frac{r}{{\cal N_M}},\frac{1-r}{{\cal N}_F}
\right\}.
\end{equation}
\vspace{2mm}

\noindent $\bullet$ 4.\
Finally, putting together the sufficient conditions in \eqref{ea3}, \eqref{ea10} and \eqref{ea5} yields the following sufficient condition for global asymptotic stability of $E^*_0$:
\begin{eqnarray*}
\left\langle \frac{1}{M_S^{\text{per}} } \right\rangle
&  < &
e\beta\gamma \max\left\{
\frac{1}{{\cal N}_M}, \frac{1}{{\cal N}_F},
\frac{2}{\max\{r,1-r\}}
\min\left\{
\frac{r}{{\cal N}_M},\frac{1-r}{{\cal N}_F}
\right\}
\right\}.
\end{eqnarray*}

Expressing the mean value as a function of $\Lambda$ with the help of \eqref{eqq12}, one establishes that $E^*_0$ is GAS if
\begin{eqnarray*}
\Lambda
& > &
\frac{2}{e\beta\gamma} \:
\frac{\cosh(\mu_S\tau)-1}{\mu_S\tau^2} \:
\frac{1}{\max \left\{
\frac{1}{{\cal N}_M}, \frac{1}{{\cal N}_F},
\frac{2}{\max\{r,1-r\}}
\min\left\{
\frac{r}{{\cal N}_M},\frac{1-r}{{\cal N}_F}
\right\}
\right\}}\\
& = &
\frac{2}{e\beta\gamma} \:
\frac{\cosh(\mu_S\tau)-1}{\mu_S\tau^2} \:
\min\left\{
{\cal N}_M, {\cal N}_F,
\frac{\max\{r,1-r\}}{2
\min\left\{
\frac{r}{{\cal N}_M},\frac{1-r}{{\cal N}_F}
\right\}}
\right\}\\
& = &
\frac{2}{e\beta\gamma} \:
\frac{\cosh(\mu_S\tau)-1}{\mu_S\tau^2} \:
\min\left\{
{\cal N}_M, {\cal N}_F,
\frac{\max\{r,1-r\}}{2}
\max\left\{
\frac{{\cal N}_M}{r},\frac{{\cal N}_F}{1-r}
\right\}
\right\},
\end{eqnarray*}
which is exactly the formula \eqref{Lambda_cond}.
This concludes the proof of Theorem \ref{th3}.
\end{proof}
\begin{rem}
A rough {\color{black} upper bound estimate for $\Lambda^{crit}_{per}$} can be obtained using the result from the \tcm{constant} continuous release case: if $\Lambda$ is chosen such that $\Lambda>\Lambda^{crit} := 2\dfrac{\mu_S}{\beta\gamma}
\dfrac{\phi^{crit}(\cN_F)}{1+\frac{\cN_F}{\cN_M}}$, then $E^*_0$ is GAS. Thus, using a comparison principle, a sufficient condition to ensure global asymptotic stability of $E^*_0$ is to choose
\[
\underline{M}_{S}^{per}\geq\dfrac{\Lambda^{crit}}{\mu_S},
\]
where $\underline{M}_{S}^{per}=\min \limits_{t\in[0,\tau]} M_{S}^{per}(t)=\tau \Lambda \dfrac{e^{-\mu_S \tau}}{1-e^{-\mu_S \tau}}$. Thus, we derive that, for a given $\tau$, if
\begin{equation}
\ensuremath{\Lambda\geq\Lambda^{crit}\dfrac{e^{\mu_{S}\tau}-1}{\mu_S\tau},}
\label{lambda}
\end{equation}
then $E^*_0$ is GAS. When $\tau\rightarrow 0^+$, we recover the result for the \tcm{constant} continuous release (cf.\ Theorem {\rm\ref{tth1}}).
\end{rem}

\section{Elimination by feedback control}
\label{se5}

We now assume that measurements are available, providing real time estimates of the number of wild males and females $M(t), F(t)$, at least for any $t = n \tau, n \in \mathbb N$. One thus has the possibility to choose the number $\tau\Lambda_n$ of mosquitoes released at time $n\tau$ in view of this information: this is a {\em closed-loop control} option.
We study in the sequel this strategy.


\subsection{Principle of the method}

The principle of the stabilization method {\color{black} that we introduce now} is based on two steps.
The first one (Section \ref{se1}) consists in solving the stabilization problem under the hypothesis that one can directly actuate on $M_S$.
The second one (Section \ref{se2}) consists in showing how to realize, through adequate choice of $\Lambda_n$, the prescribed behavior of $M_S$ defined in Step 1.
The formal statement and proof are provided later, in Section \ref{se3}.

\subsubsection{Step 1 -- Setting directly the sterile population level}
\label{se1}
We first suppose to be capable of directly controlling the quantity $M_S$. We will rely on the following key property.
\begin{prop}
\label{prop2}
Let $k$ be a real number such that
\begin{equation}
\label{eq1}
0 < k < \frac{1}{\cN_F}.
\end{equation}
Assume that
\begin{equation}
\label{eq5}
\gamma M_S(t) \geq \left(
\frac{1}{k}-1
\right) M(t),\qquad t\geq 0.
\end{equation}
Then every solution of \eqref{system_SIT-a}-\eqref{system_SIT-b} converges exponentially to $E_0^*$.
\end{prop}

The idea behind formula \eqref{eq1} is quite natural: it suffice to impose a fixed upper bound $k$ on the ratio $\dfrac{M}{M+\gamma M^*_S}$ in order to make the `apparent' basic offspring number {\color{black} $k\cN_F$ smaller than 1}, and consequently to render inviable the wild population.
Notice that this condition corresponds exactly to the stability of the system linearized around the origin.
It {\color{black} may be} excessively demanding for {\em large} population sizes, as it ignores the effects of competition modeled by the exponential term.
{\color{black} We shall come back to this point in Section \ref{se6} and introduce saturation.}

\begin{proof}[Proof of Proposition \rm\ref{prop2}]
From equations \eqref{system_SIT-a} and \eqref{system_SIT-b} we have
\begin{subequations}
\label{eq6}
\begin{gather}
\label{eq6a}
\dot M
= r\rho \dfrac{FM}{M+\gamma M_S}e^{-\beta(M+F)}-\mu_M M \leq r\rho\dfrac{FM}{M+\gamma M_S}-\mu_M M \leq
 -\mu_M M + r\rho k F, \\
\label{eq6b}
\dot F
= (1-r)\rho\dfrac{FM}{M+\gamma M_S}e^{-\beta(M+F)}-\mu_F F \leq ((1-r)\rho k-\mu_F )F.
\end{gather}
\end{subequations}
The linear  autonomous system
\begin{equation}
\label{eq7}
\begin{pmatrix}
\dot M' \\ \dot F'
\end{pmatrix}
= \begin{pmatrix}
-\mu_M  & r\rho k\\
0 & -\mu_F  + (1-r)\rho k
\end{pmatrix}
\begin{pmatrix}
M' \\ F'
\end{pmatrix}
\end{equation}
is {\em monotone} (it involves a Metzler matrix) and may thus serve as a comparison system for  the evolution of \eqref{system_SIT-a}-\eqref{system_SIT-b}. Thus, it is  deduced that
\[
0 \leq M(t) \leq M'(t),\qquad 0 \leq F(t) \leq F'(t),\qquad t\geq 0,
\]
where $(M',F')$ is the solution of \eqref{eq7} generated by the same initial values as the underlying solution $(M,F)$ of \eqref{system_SIT-a}-\eqref{system_SIT-b}.

On the other hand, system \eqref{eq7} is asymptotically stable when \eqref{eq1} holds. In other words, $M'(t)$ and $F'(t)$ converge to $E_0^*$ asymptotically. In consequence, $M(t)$ and $F(t)$ also converge to $E_0^*$ asymptotically when \eqref{eq1} is in force.
This achieves the proof of Proposition \ref{prop2}.
\end{proof}

\subsubsection{Step 2 -- Shaping an impulsive control compliant with Step 1}
\label{se2}

We now want to ensure that condition \eqref{eq5} is fulfilled, through an adequate choice of the impulse amplitude $\Lambda_n$.
In virtue of \eqref{system_SIT_imp-c}-\eqref{system_SIT_imp-d}, the value of $M_S$ on the interval $\big( n\tau,(n+1)\tau \big]$ is given by
\begin{equation}
\label{eq11}
M_S (t) = M_S(n\tau^+) e^{-\mu_S(t-n\tau)} = \big( \Lambda_n\tau + M_S(n\tau) \big)e^{-\mu_S(t-n\tau)},
\end{equation}
and we would like to choose $\Lambda_n$ in such a way that \eqref{eq5} stays in force.
However, instead of computing the (nonlinear) evolution of $M(t)$ on the interval $\big( n\tau,(n+1)\tau \big)$, we will impose, rather than \eqref{eq5}, the stronger condition
\begin{equation}
\label{eq9}
\gamma M_S(t) \geq \left(
\frac{1}{k}-1
\right) M'(t),\qquad t\geq 0
\end{equation}
where $M'(t)$ refers to the \emph{super-solution} of $M(t)$ introduced in the proof of Proposition \ref{prop2}.
(Notice that the conservatism introduced 
in this step remains reasonable when the original nonlinear system evolves in region where $\beta(M+F) \ll 1$.)
Thereby, we can solve \eqref{eq7} explicitly on $\big( n\tau,(n+1)\tau \big]$ using the following result.

\begin{lem}
\label{le1}
The solution of system \eqref{eq7} on $\big( n\tau,(n+1)\tau \big]$ with initial values $\big( M'(n\tau),F'(n\tau) \big) = \big( M(n\tau),F(n\tau) \big)$ is given by
\begin{equation}
\label{eq10}
\begin{pmatrix}
M'(t) \\ F'(t)
\end{pmatrix}
=
\begin{pmatrix}
e^{-\mu_M(t-n\tau)} & \dfrac{r\rho k}{\mu_M-\mu_F+(1-r)\rho k}
\left(
e^{-(\mu_F-(1-r)\rho k)(t-n\tau)} - e^{-\mu_M(t-n\tau)}
\right)\\
0 & e^{-(\mu_F-(1-r)\rho k)(t-n\tau)}
\end{pmatrix}
\begin{pmatrix}
M(n\tau) \\ F(n\tau)
\end{pmatrix}
\end{equation}


\end{lem}

The proof of Lemma \ref{le1} presents no difficulty and is left to the reader.

All the components of the matrix in \eqref{eq10} are nonnegative provided that $\mu_F$, $\mu_M$, $\rho$ and $k$ are chosen such that $\mu_F-\mu_M-(1-r)\rho k\leq 0$. It is worthwhile to recall that $\mu_F \leq \mu_M$ (see \eqref{eq600}, page \pageref{eq600}); therefore, the former condition is verified for any positive $\rho$ and $k$.

\comment{
\begin{proof}
Clearly $F'$ given in \eqref{eq10} verifies
\[
\dot F' = - (\mu_F-(1-r)\rho k)F'
\]
as prescribed by \eqref{eq7}; while deriving the expression in the first line of \eqref{eq10} yields:
\begin{eqnarray*}
\dot M'
& = &
-\mu_M \left(
M'(t)  - \frac{r\rho k}{\mu_M-\mu_F+(1-r)\rho k} e^{-(\mu_F-(1-r)\rho k)(t-n\tau)} F(n\tau)
\right)\\
& &
- \frac{r\rho k(\mu_F-(1-r)\rho k)}{\mu_M-\mu_F+(1-r)\rho k}
e^{-(\mu_F-(1-r)\rho k)(t-n\tau)} F(n\tau)\\
& = &
-\mu_M M'(t) + \left(
 \frac{r\rho k(\mu_M-\mu_F+(1-r)\rho k)}{\mu_M-\mu_F+(1-r)\rho k}
\right)
e^{-(\mu_F-(1-r)\rho k)(t-n\tau)} F(n\tau)\\
& = &
-\mu_M M'(t) + r\rho k F(n\tau) e^{-(\mu_F-(1-r)\rho k)(t-n\tau)}\\
& = &
- \mu_M M'(t) + r\rho k F'(t)
\end{eqnarray*}
which also obeys \eqref{eq7}.
On the other hand, the functions in \eqref{eq10} fulfill $M'(n\tau) = M(n\tau)$ and $F'(n\tau) = F(n\tau)$: Lemma \ref{le1} is proved.
\end{proof}
}

We now come back to the control synthesis.
Using \eqref{eq11} and \eqref{eq10}, condition \eqref{eq9} is fulfilled provided that, on any interval $\big( n\tau, (n+1)\tau \big]$,
\begin{multline}
\label{eq24}
\gamma \big( \Lambda_n\tau + M_S(n\tau) \big)e^{-\mu_S(t-n\tau)}
= \gamma M_S(t)
\geq \left(
\frac{1}{k}-1
\right) M'(t) \\
= \frac{1-k}{k} \left(
e^{-\mu_M(t-n\tau)} M(n\tau)
+ \frac{r\rho k}{\mu_M-\mu_F+(1-r)\rho k}
\left(
e^{-(\mu_F-(1-r)\rho k)(t-n\tau)} - e^{-\mu_M(t-n\tau)}
\right) F(n\tau)
\right).
\end{multline}
This condition is equivalent to
\begin{equation}
\label{eq800}
\Lambda_n\tau \geq
\frac{1-k}{\gamma k}e^{(\mu_S-\mu_M)s} \left(
 M(n\tau) + \frac{r\rho k}{\mu_M-\mu_F+(1-r)\rho k} \left(
e^{(\mu_M-\mu_F+(1-r)\rho k)s}-1\right) F(n\tau)
\right) - M_S(n\tau)
\end{equation}
for any $s \in [0,\tau]$. In virtue of the relationships \eqref{eq600} and \eqref{mus}, the right-hand side of previou inequality \eqref{eq800} is {\em increasing} in $s$. Therefore, condition \eqref{eq800} has to be checked only for $s=\tau$.


\subsection{Stabilization result}
\label{se3}

\subsubsection{Synchronized measurements and releases}

We now state and prove the stabilization result suggested by the previous considerations.

\begin{thm}[Sufficient condition for stabilization by impulsive feedback control]
\label{th1}
Assume that {\em nonnegative} releases $\Lambda_n$, $n\in\mathbb N$, are used in accordance with the following constraint:
\begin{subequations}
\label{eq12}
\begin{multline}
\label{eq12a}
\Lambda_n \geq -\frac{1}{\tau} M_S(n\tau)\\
+ \frac{1}{\gamma\tau}
\left(
\frac{1-k}{k}e^{(\mu_S-\mu_M)\tau}M(n\tau)
+ \frac{r\rho (1-k)}{\mu_M-\mu_F+(1-r)\rho k} \left(
e^{(\mu_S-\mu_F+(1-r)\rho k)\tau}-e^{(\mu_S-\mu_M)\tau}
\right) F(n\tau)
\right)
\end{multline}
for a given constant
\begin{equation}
\label{eq12b}
k\in \left(
0,\frac{\mu_F}{(1-r)\rho}
\right).
\end{equation}
Then every solution of system \eqref{system_SIT_imp} converges exponentially towards $E_0^*$, with a convergence rate bounded from below by a value independent of the initial condition.

If moreover
\begin{equation}
\label{eq12c}
\Lambda_n \leq
\frac{1}{\gamma\tau}
\left(
\frac{1-k}{k}e^{(\mu_S-\mu_M)\tau}M(n\tau)
+ \frac{r\rho (1-k)}{\mu_M-\mu_F+(1-r)\rho k} \left(
e^{(\mu_S-\mu_F+(1-r)\rho k)\tau}-e^{(\mu_S-\mu_M)\tau}
\right) F(n\tau)
\right)
\end{equation}
then the series of impulses $\sum \limits_{n=0}^{+\infty} \Lambda_n$ converges.
\end{subequations}
\end{thm}

Implementing of the previous control law necessitates the measurement of $M(n\tau), F(n\tau)$ (or their upper estimates), and of $M_S(t)$ (or its lower estimate). A possibility to have \eqref{eq12a} fulfilled, is to ignore the population of sterile males already present at time $n\tau$ and to take simply the {\em linear} control law
\begin{equation*}
\Lambda_n
:=
\frac{1-k}{\gamma k\tau}
\begin{pmatrix} e^{(\mu_S-\mu_M)\tau} & \dfrac{r\rho k}{\mu_M-\mu_F+(1-r)\rho k} \left(
e^{(\mu_S-\mu_F+(1-r)\rho k)\tau}-e^{(\mu_S-\mu_M)\tau}
\right)
\end{pmatrix}
\begin{pmatrix}
M(n\tau) \\ F(n\tau)
\end{pmatrix}.
\end{equation*}
Notice that this expression corresponds to the value in the right-hand side of \eqref{eq12c}.

\comment{
An interesting issue is how to choose the gain $k$.
In this respect one has to check whether or not it is possible to take $k$ as a function $k_n$ of $n$.
In such an eventuality, it would be tempting to choose this value in order to minimize the right-hand side of \eqref{eq12a}, for given values of $M(n\tau), F(n\tau)$, or the upper estimate of the sum \eqref{eq23} of the series $\sum\Lambda_n$, proportional to the total number of sterile mosquitoes introduced.

The birth rates per capita are smaller when $M+F$ is larger, due to the exponential term $e^{-\beta(M+F)}$.
Therefore a smaller control effort is indeed needed to drag the state closer to the origin in a point where $M+F$ is not zero.
This property could be exploited for the purpose of gaining convergence with smaller input values.
A practical way to achieve this, could be to use a {\em lower} bound of $M(n\tau), F(n\tau)$ on each interval $(n\tau, (n+1)\tau]$, allowing to upper bound the exponential term in a tighter way.
}

\begin{proof}[Proof of Theorem \rm\ref{th1}]
When $\big( M(n\tau), F(n\tau) \big) = (0,0)$, an impulsion $\Lambda_n$ has no effect on the evolution of $(M,F)$: the origin is an equilibrium point of system \eqref{system_SIT_imp}.
We now consider the case $\big( M(n\tau), F(n\tau) \big) \neq (0,0)$.
\vspace{2mm}

\noindent $\bullet$ 1.\
Assume first that \eqref{eq12a} is fulfilled with a {\em strict} inequality.
By construction, one has:
\begin{equation}
\label{eq13}
\forall \: t\in \big( n\tau, (n+1)\tau \big],\qquad \gamma M_S(t)> \frac{1-k}{k} M'(t)
\end{equation}
where $(M',F')$ stands for solution of \eqref{eq7} departing from $\big( M(n\tau), F(n\tau) \big)$ at time $n\tau$.

We will {\color{black} first establish that this implies:}
\begin{equation}
\label{eq18}
\forall \: t\in \big[ n\tau, (n+1)\tau \big],\qquad M(t) \leq M'(t), \quad  F(t) \leq F'(t).
\end{equation}
For this, let $t_0$ be any element of $\big[ n\tau,(n+1)\tau \big)$ such that $M(t_0) \leq M'(t_0)$, $F(t_0)\leq F'(t_0)$ {\em with at least one equality}.
Let us show the existence of $t_1$ such that $t_0 < t_1 <(n+1)\tau$ and
\begin{equation}
\label{eq19}
\forall \: t\in (t_0,t_1),\qquad M(t)< M'(t),\ F(t)<F'(t).
\end{equation}
Indeed, due to \eqref{eq13} and by definition of $t_0$, one has
\begin{equation*}
\gamma M_S(t_0)  > \frac{1-k}{k} M'(t_0) \geq \frac{1-k}{k} M(t_0),
\end{equation*}
where we write by convention $M_S(t_0):=M_S(n\tau^+)$ when $t_0=n\tau$.
By continuity of the functions $M(t)$ and $M_S(t)$ on the open interval $\big( n\tau, (n+1)\tau \big)$, there thus exists $t_1$ such that $t_0 < t_1 <(n+1)\tau$ and
\begin{equation*}
\forall \:  t\in (t_0,t_1),\qquad \gamma M_S(t) > \frac{1-k}{k} M(t).
\end{equation*}
In such conditions, it can be shown as in Proposition \ref{prop2} that $\big( M'(t),F'(t) \big) \geq \big( M(t),F(t) \big)$ for any $t\in (t_0,t_1)$, and even that $\big( M'(t),F'(t) \big)> \big( M(t),F(t)\big)$, because the functions defining the right-hand sides of \eqref{system_SIT_imp-a} and \eqref{system_SIT_imp-b} take on strictly smaller values than those defining the right-hand sides of \eqref{eq7}.
Therefore, for any $t_0\in \big\{ n\tau^+ \big\} \cup \big( n\tau,(n+1)\tau \big)$, there exists $t_1 > t_0$ such that \eqref{eq19} holds.

From \eqref{eq19} and the fact that $\big( M(n\tau), F(n\tau) \big) = \big( M'(n\tau), F'(n\tau) \big)$, {\color{black} one deduces that \eqref{eq19} is true for $t_1=(n+1)\tau$, and therefore that \eqref{eq18} is true.}
Finally, putting together \eqref{eq13} and \eqref{eq18} yields the following key property:
\begin{equation}
\label{eq22}
\forall \: t\in \big( n\tau, (n+1)\tau \big],\qquad \gamma M_S(t) > \frac{1-k}{k} M(t).
\end{equation}
\vspace{2mm}

\noindent $\bullet$ 2.\
Assume now that \eqref{eq12a} is fulfilled (with the original non-strict inequality).
Considering values of $\Lambda_n$ converging from above towards the quantity in the right-hand side of this inequality and relying on the continuity of the flow with respect to $\Lambda_n$, yields instead of \eqref{eq22} the non-strict inequality:
\begin{equation}
\label{eq22b}
\forall t\in (n\tau, (n+1)\tau],\qquad \gamma M_S(t) \geq \frac{1-k}{k} M(t).
\end{equation}
\vspace{2mm}

\noindent $\bullet$ 3.\
Consider now the positive {\color{black} \em semidefinite} function
\begin{equation}
\label{eq16a}
V(M,F) := F.
\end{equation}
In view of \eqref{eq22b}, we have that for any $t\in \big( n\tau, (n+1)\tau \big]$ it holds that
\[
\frac{M(t)}{M(t)+\gamma M_S(t)}e^{-\beta(M(t)+F(t))} \leq \frac{M(t)}{M(t)+\gamma M_S(t)} <k.
\]
Therefore,
\[
\dot F=(1-r)\rho\dfrac{FM}{M+\gamma M_S}e^{-\beta(M+F)}-\mu_F F
\leq \big( (1-r)\rho k-\mu_F \big) F.
\]
Due to \eqref{eq12b}, there exists $\varepsilon>0$ such that
\begin{equation*}
\label{eq16b}
\mu_F - (1-r)\rho k > \varepsilon
\end{equation*}
and then $\dot F \leq -\varepsilon F$.
This property ensures that $F(t)$ {\em decreases} with time, and converges exponentially towards {\color{black} 0.
It is then deduced from \eqref{system_SIT_imp-a} that $M(t)$ also converges exponentially towards 0: overall, $(M(t),F(t))$ converges towards $E_0^*$.}
\vspace{2mm}

\noindent $\bullet$ 4.\
Last, choose now $\Lambda_n$ fulfilling \eqref{eq12a} and \eqref{eq12c}.
{\color{black} From the property of exponential stability previously demonstrated, there exist $C,\varepsilon>0$} such that $M(t)<Ce^{-\varepsilon t}$ and $F(t)<Ce^{-\varepsilon t}$ for any $t\geq 0$.
{\color{black} We then deduce that}
\begin{eqnarray*}
\Lambda_n
& \leq &
\frac{1}{\gamma\tau}
\left(
\frac{1-k}{k}e^{(\mu_S-\mu_M)\tau}M(n\tau) + \frac{r\rho (1-k)}{\mu_M-\mu_F+(1-r)\rho k} \left(
e^{(\mu_S-\mu_F+(1-r)\rho k)\tau}-e^{(\mu_S-\mu_M)\tau}
\right) F(n\tau)
\right)\\
& \leq &
\color{black}
\frac{C}{\gamma\tau}
\left(
\frac{1-k}{k}e^{(\mu_S-\mu_M)\tau} + \frac{r\rho (1-k)}{\mu_M-\mu_F+(1-r)\rho k} \left(
e^{(\mu_S-\mu_F+(1-r)\rho k)\tau}-e^{(\mu_S-\mu_M)\tau}
\right)
\right)
e^{-n\varepsilon \tau},
\end{eqnarray*}
and one gets by summation
\begin{equation*}
\label{eq23}
\color{black}
\sum_{n=0}^{+\infty} \Lambda_n
\leq \frac{C}{\gamma\tau}
\left(
\frac{1-k}{k}e^{(\mu_S-\mu_M)\tau}+ \frac{r\rho (1-k)}{\mu_M-\mu_F+(1-r)\rho k} \left(
e^{(\mu_S-\mu_F+(1-r)\rho k)\tau}-e^{(\mu_S-\mu_M)\tau}
\right)
\right)
\frac{1}{1-e^{-\varepsilon\tau}}.
\end{equation*}
This shows the convergence of the series and concludes the proof of Theorem \ref{th1}.
\comment{
For any $n\in\mathbb N$,
\begin{eqnarray*}
\Lambda_n
& \leq &
\frac{1}{\gamma\tau}
\left(
\frac{1-k}{k}e^{(\mu_S-\mu_M)\tau}M(n\tau) + \frac{r\rho (1-k)}{\mu_M-\mu_F+(1-r)\rho k} \left(
e^{(\mu_S-\mu_F+(1-r)\rho k)\tau}-1
\right) F(n\tau)
\right)\\
& \leq &
\frac{1}{\gamma\tau}
\max\left\{
\frac{1-k}{\varepsilon k}e^{(\mu_S-\mu_M)\tau}; \frac{r\rho (1-k)}{\mu_M-\mu_F+(1-r)\rho k} \left(
e^{(\mu_S-\mu_F+(1-r)\rho k)\tau}-1
\right)
\right\}
V(M(n\tau),F(n\tau))\\
& \leq &
\frac{1}{\gamma\tau}
\max\left\{
\frac{1-k}{\varepsilon k}e^{(\mu_S-\mu_M)\tau}; \frac{r\rho (1-k)}{\mu_M-\mu_F+(1-r)\rho k} \left(
e^{(\mu_S-\mu_F+(1-r)\rho k)\tau}-1
\right)
\right\}
e^{-n\eta\tau}
V(M(0),F(0))
\end{eqnarray*}
by use of \eqref{eq25}.
The constant $\eta$ being positive, one thus gets by summation
\begin{equation}
\label{eq23}
\sum_{n=0}^{+\infty} \Lambda_n
\leq \frac{1}{\gamma\tau}
\max\left\{
\frac{1-k}{\varepsilon k}e^{(\mu_S-\mu_M)\tau}; \frac{r\rho (1-k)}{\mu_M-\mu_F+(1-r)\rho k} \left(
e^{(\mu_S-\mu_F+(1-r)\rho k)\tau}-1
\right)
\right\}
\frac{1}{1-e^{-\eta\tau}}
(\varepsilon M(0)+F(0))
\end{equation}
}
\end{proof}

\begin{rem}
Since $\Lambda_n$ is chosen nonnegative, inequality \eqref{eq12a} practically becomes
\begin{align*}
\Lambda_n \geq & \: \max \bigg\{ 0, \: -\frac{1}{\tau} M_S(n\tau) \\
 & + \: \dfrac{1}{\gamma\tau}
\left(
\frac{1-k}{k}e^{(\mu_S-\mu_M)\tau}M(n\tau)
+ \frac{r\rho (1-k)}{\mu_M-\mu_F+(1-r)\rho k} \left(
e^{(\mu_S-\mu_F+(1-r)\rho k)\tau}-e^{(\mu_S-\mu_M)\tau}
\right) F(n\tau)\right) \bigg\}.
\end{align*}
It means that the release of sterile males at time $t=n\tau$ is not (really) necessary, if the sterile males population is large enough, i.e.
$$
\gamma M_S(n\tau) >
\frac{1-k}{k}e^{(\mu_S-\mu_M)\tau}M(n\tau)
+ \frac{r\rho (1-k)}{\mu_M-\mu_F+(1-r)\rho k} \left(
e^{(\mu_S-\mu_F+(1-r)\rho k)\tau}-e^{(\mu_S-\mu_M)\tau}
\right) F(n\tau).
$$
Using this result, one may avoid unnecessary releases, thereby reducing the overall cumulative number of released males and the underlying cost of SIT control.
\end{rem}


\subsubsection{Sparse measurements}
{\color{black} The feedback control approach requires to assess 
the size of mosquito population at every $t\in\tau\mathbb N$.
Using Mark-Release-Recapture (MRR) technique} \cite{Gouagna2015}, it is possible to estimate (roughly) the wild population, by completing sparse measurements with a period $p\tau$ for some $p\in\mathbb N^*:=\mathbb{N} \setminus \{0\}$, and computing the sizes of  $(p-1)$ intermediate releases using the last sampled information between two measurements.
{\color{black} However, the previous operation is long and costly, and reducing the frequency of realization of this protocol is legitimate.}

The following result adapts the control laws given in Theorem {\rm\ref{th1}} to sparse measurements.

\begin{thm}[Stabilization by impulsive control with sparse measurements]
\label{th2}
Let $p\in\mathbb N^*$ and set $k$ complying with \eqref{eq12b}.
Assume that, for any $m=0,1,\dots, p-1$, {\em nonnegative} releases $\Lambda_{np+m}$, $n\in\mathbb N$, are accomplished in accordance with the following constraint:
\begin{subequations}
\label{eq450}
\begin{eqnarray}
\label{eq450a}
\lefteqn{\Lambda_{np+m}
\geq -\Lambda_{np+m-1} e^{-\mu_S\tau} - \cdots - \Lambda_{np} e^{-m\mu_S\tau}
-\frac{1}{\tau} M_S(np\tau) e^{-m\mu_S\tau}
} \\
& \!+\! & \frac{e^{\mu_S\tau}}{\gamma\tau}\left(
\frac{1 \!-\! k}{k} e^{-\mu_M(m+1)\tau} M(n\tau)
+ \frac{r\rho (1 \!-\!k)}{\mu_M \!-\! \mu_F+(1 \!-\! r)\rho k}
\left(
e^{-(\mu_F-(1-r)\rho k)(m+1)\tau} \!- \! e^{-\mu_M(m+1)\tau}
\right)\! F(n\tau)
\right). \nonumber
\end{eqnarray}


Then every solution of system \eqref{system_SIT_imp} converges exponentially towards $E_0^*$, with a convergence speed bounded from below by a value independent of the initial condition.

If moreover
\begin{align}
\label{eq450b}
\Lambda_{np+m} & \leq
\frac{e^{\mu_S\tau}}{\gamma\tau} \bigg(
\frac{1-k}{k} e^{-\mu_M(m+1)\tau} M(n\tau) \\
& + \frac{r\rho (1-k)}{\mu_M-\mu_F+(1-r)\rho k}
\left(
e^{-(\mu_F-(1-r)\rho k)(m+1)\tau} - e^{-\mu_M(m+1)\tau}
\right) F(n\tau)
\bigg), \nonumber
\end{align}
\end{subequations}
then the series of impulses $\sum \Lambda_n$ converges.
\end{thm}

Notice that Theorem \ref{th2} represents an {\em extension} of Theorem \ref{th1}, recovered in the case $p=1, m=0$: in this case, \eqref{eq450a} boils down to \eqref{eq12a}.

\begin{proof}[Proof of Theorem \rm\ref{th2}]
The demonstration comes from a slight adaptation of the proof of Theorem \ref{th1}.
Indeed, it suffices to verify that, {\color{black} under the conditions in Theorem \ref{th2}, property \eqref{eq24} holds on the interval $(np\tau, (n+1)p\tau]$, of length $p\tau$.}
Let $m\in\{0,1,\dots, p-1\}$.
One has for any $s\in (0,\tau]$ that
\begin{align*}
M_S \big( s+(np+m)\tau \big)
& = \Big(
\Lambda_{np+m}\tau + M_S \big( (np+m)\tau \big)
\Big)e^{-\mu_Ss}\\
& = \Big(
\Lambda_{np+m}\tau + \Lambda_{np+m-1}\tau e^{-\mu_S\tau}+ \dots + \Lambda_{np}\tau e^{-m\mu_S\tau}
+ M_S(np\tau) e^{-m\mu_S\tau}
\Big) e^{-\mu_S s}.
\end{align*}

Inequality \eqref{eq24} is thus {\color{black} true on $\big( (np+m)\tau, (np+m+1)\tau \big]$ \em if and only if} it is imposed that, for any $m\in\{0,1,\dots, p-1\}$ and any $s\in (0,\tau]$,
\begin{eqnarray}
\lefteqn{\gamma
\Big( \Lambda_{np+m}\tau + \Lambda_{np+m-1}\tau e^{-\mu_S\tau}+ \dots + \Lambda_{np}\tau e^{-m\mu_S\tau}
+ M_S(np\tau) e^{-m\mu_S\tau} \Big) e^{-\mu_Ss} } \\
\! & \!\! \geq \!\!& \frac{1\!-\!k}{k} \left(
e^{-\mu_M (s+m\tau)} M (n\tau)
+ \frac{r\rho k}{\mu_M \!-\!\mu_F+(1\!-\!r)\rho k}
\left(
e^{-(\mu_F-(1-r)\rho k)(s+m\tau)} \!-\! e^{-\mu_M(s+m\tau)}
\right) F(n\tau)
\right), \nonumber
\end{eqnarray}
that is,
\begin{eqnarray}
\label{eq45}
\lefteqn{\Lambda_{np+m}\tau e^{m\mu_S\tau} + \Lambda_{np+m-1}\tau e^{(m-1)\mu_S\tau}+ \dots + \Lambda_{np}\tau
+ M_S(np\tau) } \\
\!\! & \!\! \geq \!\! & \!\! \frac{1 \!-\!k}{\gamma k}
\left(
e^{(\mu_S-\mu_M)(s+m\tau)} M(n\tau)
+ \frac{r\rho k}{\mu_M \!-\! \mu_F+(1 \!-\! r)\rho k}
\left(
e^{(\mu_S-\mu_F+(1-r)\rho k)(s+m\tau)} \!-\! e^{(\mu_S-\mu_M)(s+m\tau)}
\right) F(n\tau)
\right). \nonumber
\end{eqnarray}

In virtue of the relationships \eqref{eq600} and \eqref{mus}, the right-hand side of \eqref{eq45} is an {\em increasing} function of $s$.
Therefore, \eqref{eq45} is more restrictive when taken at $s=\tau$.
This yields \eqref{eq450a} and shows the first part of the result.
The convergence of the series of impulses is demonstrated similarly to Theorem \ref{th1}.
\end{proof}

\begin{rem}
Again, if $\Lambda_n$ is chosen nonnegative and using the fact that the size of releases is constant between two MRR experiments, inequality \eqref{eq450a} {\color{black} leads} to the following practical choice of $\Lambda_{np+m}$:
\begin{multline}
\label{eq450practical}
\Lambda_{np+m} \geq  \max \bigg\{ 0, -\Lambda_{np} e^{-\mu_S\tau} - \dots - \Lambda_{np} e^{-m\mu_S\tau}
 +\frac{e^{\mu_S\tau}}{\gamma\tau} \Big( \frac{1-k}{k} e^{-\mu_M(m+1)\tau} M(n\tau) \\
+ \frac{r\rho (1-k)}{\mu_M-\mu_F+(1-r)\rho k} \left(
e^{-(\mu_F-(1-r)\rho k)(m+1)\tau} - e^{-\mu_M(m+1)\tau}
\right) F(n\tau) \Big) \bigg\}.
\end{multline}
\end{rem}

\section{Mixed control strategies}
\label{se6}
The results obtained in the previous sections for open-loop and closed-loop SIT control allow us to compare several SIT release strategies.
Here, we consider only periodic impulsive control, which is more realistic than continuous control.

The open-loop approach (developed in Section \ref{se4}), is based on the determination of a sufficient size of sterile males to be released, in order to eradicate the wild population. This choice is made according to \eqref{Lambda_cond}.
Under this approach, even though the previous formula is `tight', the same amount of sterile insects is used during the whole release campaign.

On the contrary, the closed-loop control approach (exposed in Section \ref{se5}) is based on estimates of the wild population and thereby it enables fitting the release sizes. As evidenced by \eqref{eq12a}, under this approach the released volume is essentially chosen as {\em proportional} to the measured population. However, this condition is certainly too demanding for large values of $M,F$ (see the comments preceding Lemma \ref{le1}).
Taking advantage of the apparent complementarity of the two approaches, we propose here {\em mixed impulsive control strategies}, combining the two previous modes.
They gather the advantages of both approaches, guaranteeing convergence to the mosquito-free equilibrium with releases that remain bounded (like the periodic impulsive control strategies, Section \ref{se4}) and vanishing with the wild population (like the feedback control strategies, Section \ref{se5}).


\begin{thm}
\label{th4}
Let $p\in\mathbb N^*$.
Assume that, for any $n\in\mathbb N$, $\Lambda_n$ is chosen at least equal to the {\em smallest} of the right-hand side of \eqref{eq450a} and of a positive constant $\bar\Lambda$.
Then every solution of {\color{black} system \eqref{system_SIT_imp}} converges globally exponentially to $E_0^*$, in any of the following situations:

$\bullet$ {\bf Case 1.}
\begin{equation}
\label{eq610}
k\in \left(
0,\frac{\mu_F}{(1-r)\rho}
\right),\quad \mbox{ and }  \quad
\bar\Lambda = 2\frac{\left(\cosh\left(\mu_S\tau\right)-1\right)}{\mu_S\tau^2}
\frac{1}{e\beta\gamma}
\cN_F.
\end{equation}

$\bullet$ {\bf Case 2.}
\begin{multline}
\label{eq611}
k\in \left(
0,2\; \frac{\mu_M}{\rho}\frac{1-r}{r^2}\left(
\sqrt{1 + \frac{\mu_F}{\mu_M} \left(
\frac{r}{1-r}
\right)^2} -1
\right)
\right), \quad \mbox{ and }  \quad \\
\bar\Lambda=\frac{\left(\cosh\left(\mu_S\tau\right)-1\right)}{\mu_S\tau^2}
\frac{1}{e\beta\gamma}
\max\{r,1-r\} \max\left\{
\frac{\cN_M}{r},\frac{\cN_F}{1-r}
\right\}.
\end{multline}
\end{thm}

\comment{
\begin{thm}
Assume that $\Lambda_n$ is chosen at least equal to the {\em smallest} of the two values in \eqref{Lambda_cond} and \eqref{eq12a}.
Then every solution of system \eqref{eq:per} converges globally asymptotically to $0$, provided that $k$ fulfills
\begin{equation}
k\in \left(
0,2\frac{\mu_M}{\rho}\frac{1-r}{r^2}\left(
\sqrt{1 + \frac{\mu_F}{\mu_M} \left(
\frac{r}{1-r}
\right)^2} -1
\right)
\right)
\end{equation}
\end{thm}
}

The interest of the previous result is of course to consider the {\em smallest} of the two values of $\bar\Lambda$: it results in {\color{black} \em saturated} control laws.

The main issue of the proof (presented below) is to overcome the possible occurrence of infinitely many switches between the two modes.
The demonstration is based on the use of {\em common Lyapunov functions}, that decrease along the trajectories of the system, regardless of the mode in use.
{\color{black} Different} Lyapunov functions are required for the two cases.

\begin{rem}
Notice that
\begin{equation}
2\frac{\mu_M}{\rho}\frac{1-r}{r^2}\left(
\sqrt{1 + \frac{\mu_F}{\mu_M} \left(
\frac{r}{1-r}
\right)^2} -1
\right)
< 2\frac{\mu_M}{\rho}\frac{1-r}{r^2} \frac{1}{2} \frac{\mu_F}{\mu_M} \left(
\frac{r}{1-r}
\right)^2
= \frac{\mu_F}{(1-r)\rho},
\end{equation}
so the condition on $k$ contained in \eqref{eq611} is more restrictive than the one in \eqref{eq610}.
\end{rem}

\begin{rem}
The values of $\bar\Lambda$ that appear in \eqref{eq610} and \eqref{eq611} are two of the three that appear in \eqref{Lambda_cond}, corresponding to \eqref{ea3} and \eqref{ea5} in the proof of Theorem {\rm\ref{th3}}, page {\rm \pageref{th3}}.
See the proof for more explanations.
\end{rem}

\begin{proof}[Proof of Theorem \rm\ref{th4}]
We consider here the case where $p=1$, and the case with $p>1$ is treated in a similar way.
\vspace{2mm}

\noindent $\bullet$ 1.\
For the \textbf{Case 1}, consider the positive semidefinite  function $V(M,F):=F$ introduced earlier by \eqref{eq16a}, page \pageref{eq16a}.
For the sake of simplicity, we will denote $V(t):=V \big( M(t),F(t) \big)$. As shown in the proof of Theorem \ref{th3}, item 1, it holds that
\begin{equation}
\label{eq700}
V \big( (n+1)\tau \big) \leq e^{-\varepsilon\tau} V(n\tau)
\end{equation}
for a certain $\varepsilon>0$ (independent of $n$) when $\Lambda_n$ is at least equal to $\bar\Lambda$ given in \eqref{eq610}.
On the other hand, it is shown in the proof of Theorem \ref{th1}, item 3, that $V$ decreases exponentially when $\Lambda_n$ is chosen according to \eqref{eq12a} (which is \eqref{eq450a} in the case $p=1$).
Therefore, regardless of the mode commutations, $V(t)$, and thus $F(t)$, converges exponentially towards zero for every trajectory.
As substantiated in the proof of Theorem \ref{th1}, this is sufficient to deduce  the convergence of $M(t)$ towards zero.
Thereby, Theorem \ref{th4} is proved in the \textbf{Case 1}.
\vspace{2mm}

\noindent $\bullet$ 2.\
For the \textbf{Case 2}, let $\mathcal{V}$ be the positive definite function $\mathcal{V}(M,F):=\frac{1}{2}(M^2+F^2)$ introduced by \eqref{V}, page \pageref{V}.
It was shown in the proof of Theorem \ref{th3}, item 3, that property \eqref{eq700} also holds for some $\varepsilon>0$ when $\Lambda_n$ is chosen according to \eqref{eq12a}.
\comment{When $\Lambda_n$ is chosen at least equal to the value in \eqref{Lambda_cond}, then the function $V=\frac{1}{2}(M^2+F^2)$ decreases from $n\tau$ to $(n+1)\tau$, as shown in the demonstration of Theorem \ref{th3}.}

On the other hand, when $\Lambda_n$ is taken smaller than the value in \eqref{Lambda_cond}, due to Theorem \ref{th1}, one has for all $t\in \big( n\tau, (n+1)\tau) \big]$, see \eqref{eq24}, that
\begin{equation}
\gamma M_S(t) \geq \left(
\frac{1}{k}-1
\right) M(t),
\qquad \text{ that is: }\quad
\frac{M(t)}{M(t)+\gamma M_S(t)} \leq k.
\end{equation}
Therefore, on the same interval, it holds:
\begin{eqnarray*}
\dot{\mathcal{V}}
& = &
M\dot M+F\dot F\\
& = &
\rho \frac{FM(rM+(1-r)F)}{M+\gamma M_S} e^{-\beta(M+F)}
-\mu_M M^2 -\mu_F F^2\\
& \leq &
\rho k F(rM+(1-r)F) e^{-\beta(M+F)}
-\mu_M M^2 -\mu_F F^2\\
& \leq &
\rho k F(rM+(1-r)F) -\mu_M M^2 -\mu_F F^2\\
& = &
-\Big(
\mu_M M^2 - \rho kr MF + (\mu_F-\rho k (1-r)) F^2
\Big).
\end{eqnarray*}
The reduced discriminant of the previous quadratic form is
\begin{equation}
\Delta' =
r^2\rho^2k^2+4\mu_M(1-r)\rho k - 4\mu_M\mu_F,
\end{equation}
which is negative when $k$ is taken according to \eqref{eq611}.
In such case, $\dot{\mathcal{V}}$ is negative definite.
One concludes that $\mathcal{V}$ decreases exponentially to zero, and this ensures the global exponential stability of the mosquito-free equilibrium $E_0^*$.
The result is thus also proved in the \textbf{Case 2}. This achieves the proof of Theorem \ref{th4}.
\end{proof}

\section{Numerical illustrations}


The values of the vital characteristics of the mosquitoes which are used in the simulations are summarized in Table \ref{tab1}.
\begin{table}[h!]
\centering
\begin{tabular}{|c|c|l|}   \hline
\textbf{Parameter} & \textbf{Value}	 & \textbf{Description} \\ \hline
$\rho$		    	&4.55& Number of eggs a  female can deposit \\ \hline
$r$							&0.5& $r:(1-r)$ expresses the primary sex ratio in offspring  \\ \hline
$\sigma$	    	&0.05&	Regulates the larvae development into adults under \\ & & density dependence and larval competition \\ \hline
$K$	          	&140&	Carrying capacity \\ \hline
$\mu_M $	    	&0.04&	Mean mortality rate of wild adult male mosquitoes \\ \hline
$\mu_F $	    	&0.03&	Mean mortality rate of wild adult female mosquitoes \\ \hline
$\mu_S$				&0.04&	Mean mortality rate of sterile adult male mosquitoes \\ \hline 
$\ensuremath{\gamma}$ &1 & Fitness of sterile adult male mosquitoes \\ \hline 
\end{tabular}
\caption{{\em Aedes spp} parameters values}
\label{tab1}
\end{table}

With the above numbers, we have here for the global competition coefficient $\beta= \dfrac{\sigma}{K}=3.57 \times 10^{-4}$, and for the basic offspring numbers $\cN_F\approx75.83$ and $\cN_M\approx 56.87$. At equilibrium, the mosquito population is thus $E^*=(M^*,F^*)$ with $M^*\approx 6,925$, $F^*\approx 5,194$ individuals per hectare.

We now present simulations related to the different impulsive strategies developed in the previous sections.

For open-loop periodic impulsive releases carried out every $7$ (resp.\ $14$) days, we consider the optimal value given in (\ref{Lambda_cond}), page \pageref{Lambda_cond}, to estimate the number of sterile males to release, that is, {\color{black} $7\times 1,573=11,011$ (resp.\ $14\times1,604=22,456$)} sterile males per hectare and per week (resp.\ every two weeks). 

The simulations run as long as $F(t)$ is greater than a threshold value, here $10^{-1}$, {\color{black} below which} we assume that elimination has been reached.

The corresponding simulations are given in Figure \ref{fig:1}. In Table \ref{tab2}, we summarize the cumulative number of sterile males as well as the number of releases needed to reach nearly ``elimination''. While, as expected, the total number of released sterile males is lower for $\tau=7$, there is no {\color{black} gain in terms of treatment duration.}
Thus, taking into account the cost of each release and also the risk of failure during the transport, it seems preferable to consider {\color{black} the lower} number of releases, and thus to choose $\tau=14$.
\begin{figure}[!t]
\begin{tabular}{cc}
 \includegraphics[width=.5\textwidth]{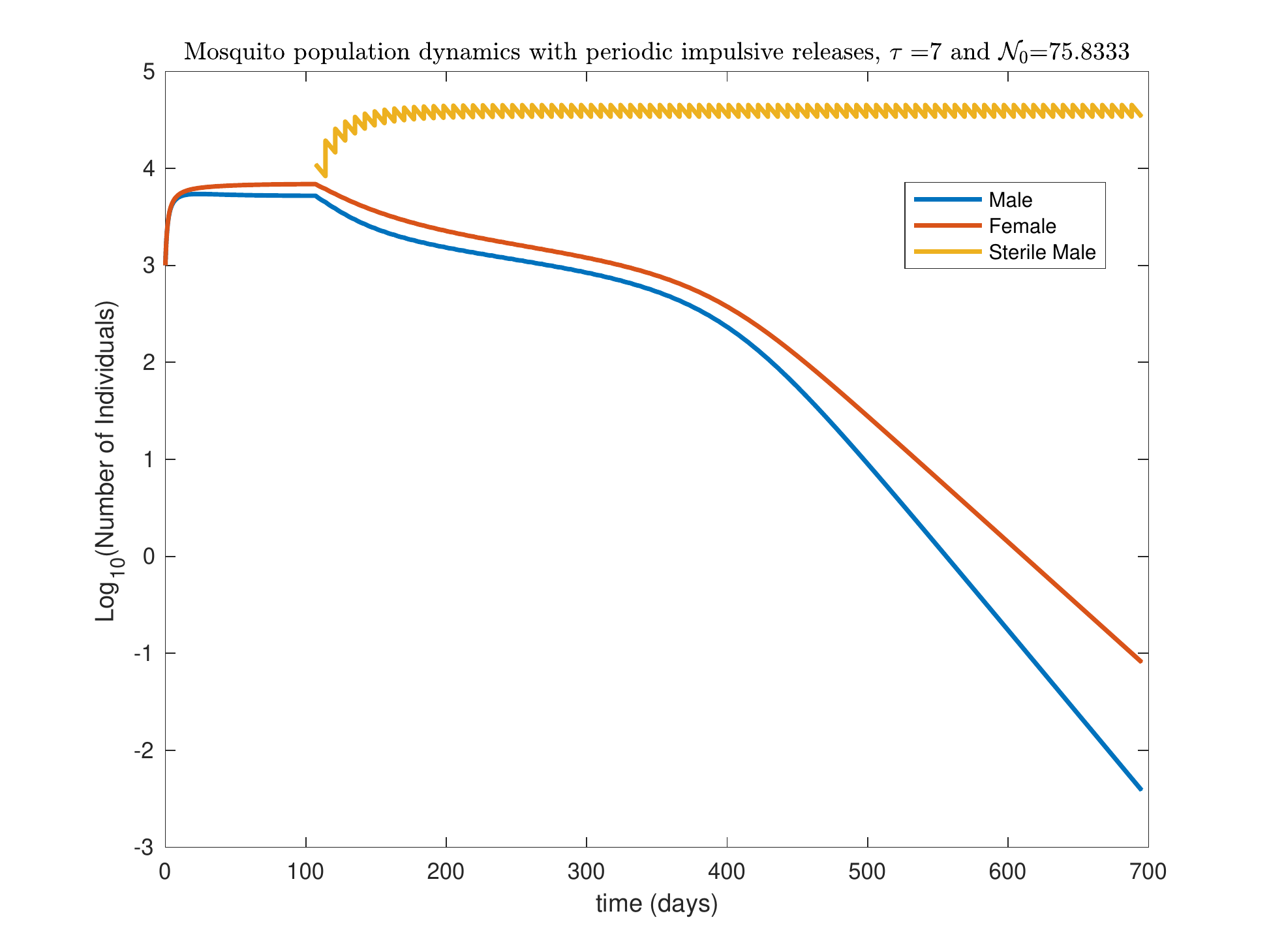} &
 \includegraphics[width=.5\textwidth]{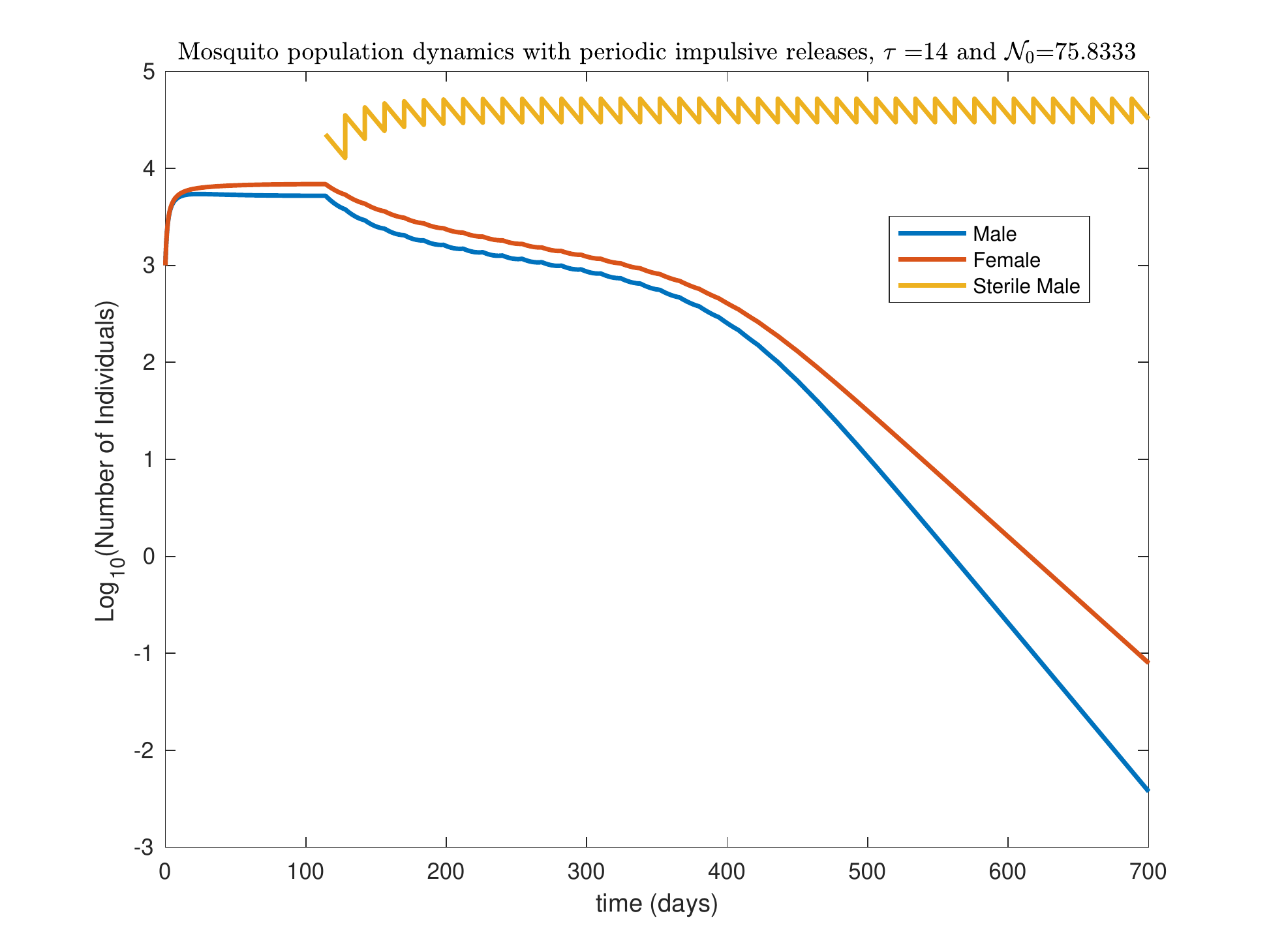} \\
 (a) & (b)
 \end{tabular}
 \caption{Open-loop periodic impulsive  SIT control of system \eqref{system_SIT_imp} with a period of: (a) $7$ days, (b) $14$ days. 
 }
 \label{fig:1}
\end{figure}

\begin{table}[!th]
\centering
\begin{tabular}{|c|c|c|}   \hline
 Period (days) & Cumulative Number of released sterile males & Nb of Weeks to reach elimination\\
 \hline
$\tau=7$  & $ \color{black} \bf 924,627$ & 84\\
\hline
$\tau=14$  & \color{black} $942,869$  & 84 \\
\hline
\end{tabular}
\caption{Cumulative number of released sterile males for each open-loop periodic SIT control treatment. 
}
\label{tab2}
\end{table}

The closed-loop approach can be used to reduce the cumulative number of released sterile insects and the number of effective releases. Further on, we consider several sub-cases.

We first consider measurements of the wild population every $\tau$ days or every $p\tau$ days for prescribed $p$ (here typically $p=4$).
Also, we take several values for the gain: smaller $k$ provides faster convergence towards $E^*_0$ -- at the price of large peak values of $\Lambda_n$ --, while the convergence slows down as  $k$ approaches $1/\cN_F$ with moderate values of $\Lambda_n$.
We will consider for practical applications two particular values of $k$, namely
\begin{equation}
\label{k}
\color{black} k\cN_F=0.2 \quad \text{and} \quad k\cN_F=0.99.
\end{equation}

\comment{
that we are able to estimate the size of the wild population every $tau$ days; the other case is when we consider that the wild population is estimated every $p\tau$ days, with $p \in \mathbb{N}^*$.
However, according to Theorem \ref{th1}, page \pageref{th1}, the choice of $k \in (0,\dfrac{1}{\cN_F})$ will play a fundamental role: close to $0$ means that the convergence towards $E^*_0$ is fast, while it is slower when $k$ is close to $\dfrac{1}{\cN_F}$. Thus, we will consider two (very) different values for $k$, namely $k=0.20 \dfrac{1}{\cN_F}$ and $k=\dfrac{0.99}{\cN_F}$, so that the last value of $k$ is really close to $\dfrac{1}{\cN_F}$.
}

The size $\Lambda_n$ of the $n$-th release is taken equal to the right-hand side of formula \eqref{eq12a} for $p=1$ (of \eqref{eq450a} for $p=4$): if, at the moment of the estimate, the size of the sterile male population is sufficiently large, $\Lambda_n$ may be null or small.

Simulations presented in Figures \ref{fig:2a} (page \pageref{fig:2a}) and \ref{fig:2b} (page \pageref{fig:2b})
clearly show that the choice of $k$ and $p$, as well as the period $\tau$ of the releases play an important role in the convergence of the wild population to $E^*_0$.
Tables \ref{tab21} and \ref{tab3} provide the total cumulative number of released sterile males, the number of weeks of SIT treatment needed to reach elimination, and the number of {\color{black} {\em effective} (that is nonzero) releases.}
For instance, when $(\tau,p)=(14,4)$ and $k=\dfrac{0.2}{\cN_F}$ is relatively small, elimination of wild mosquitoes can be achieved in $56$ weeks, with only $17$ effective releases, as shown in Fig. \ref{fig:4-5}(b), page \pageref{fig:4-5}. However, this option requires to release significant number of sterile insects per hectare (close to $2.9 \times 10^{6}$ {\color{black} for the whole treatment}).

For the larger $k=\dfrac{0.99}{\cN_F}$ and with $(\tau,p)=(7,1)$ (see Figure \ref{fig:2b}(a)), the convergence is {\color{black} slower:} more than $240$ weeks of SIT treatment are required to reach nearly elimination. For $p=4$ (see Figure \ref{fig:2b}(b)), the wild population is close to extinction after $58$ weeks of SIT treatment. However, based on Table \ref{tab3}, it seems that the choice $(\tau, p)=(14,4)$ leads to the best result in terms of timing (62 weeks) and also in terms of cumulative size encompassing 20 effective releases.

The parameter $k$ is of main importance: when $p=4$, while the number of weeks to reach elimination is quite similar for both values of $\tau$, the cumulative number of released sterile males is clearly smaller when $k$ is closer to $1/\cN_F$.
\begin{figure}[!t]
\begin{tabular}{cc}
 \includegraphics[width=.5\textwidth]{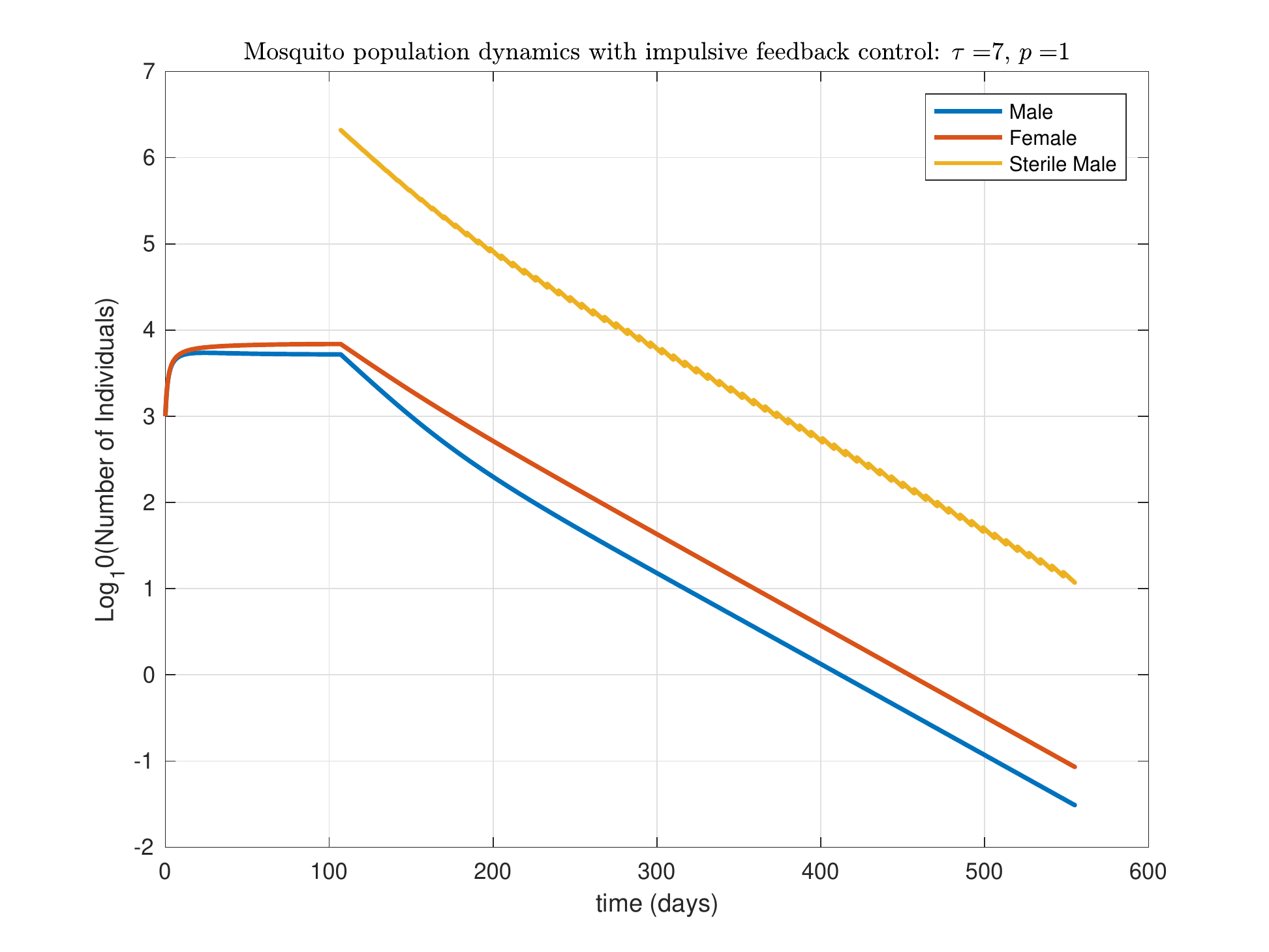} &
 \includegraphics[width=.5\textwidth]{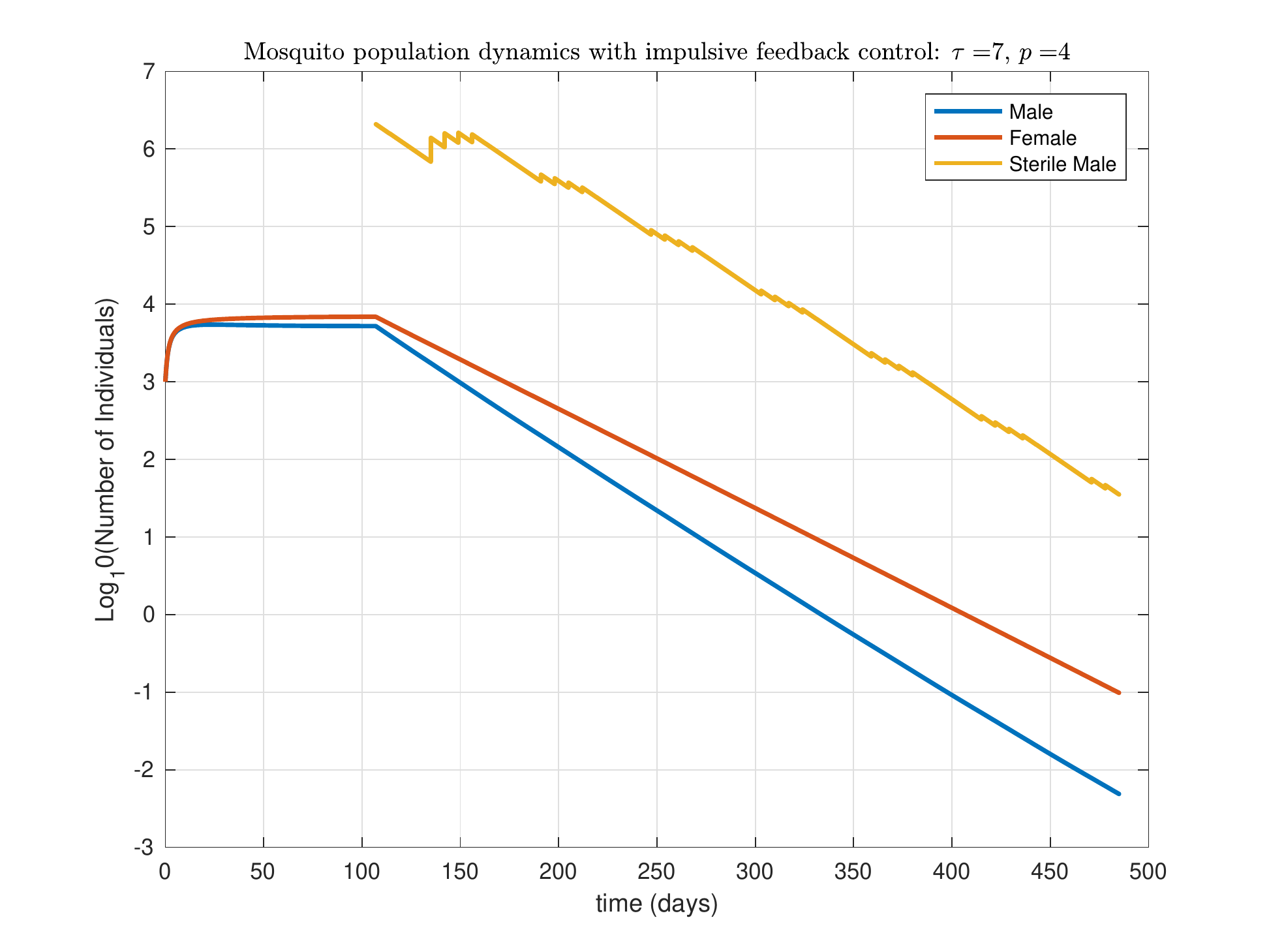} \\
 (a) & (b) \\
  \includegraphics[width=.5\textwidth]{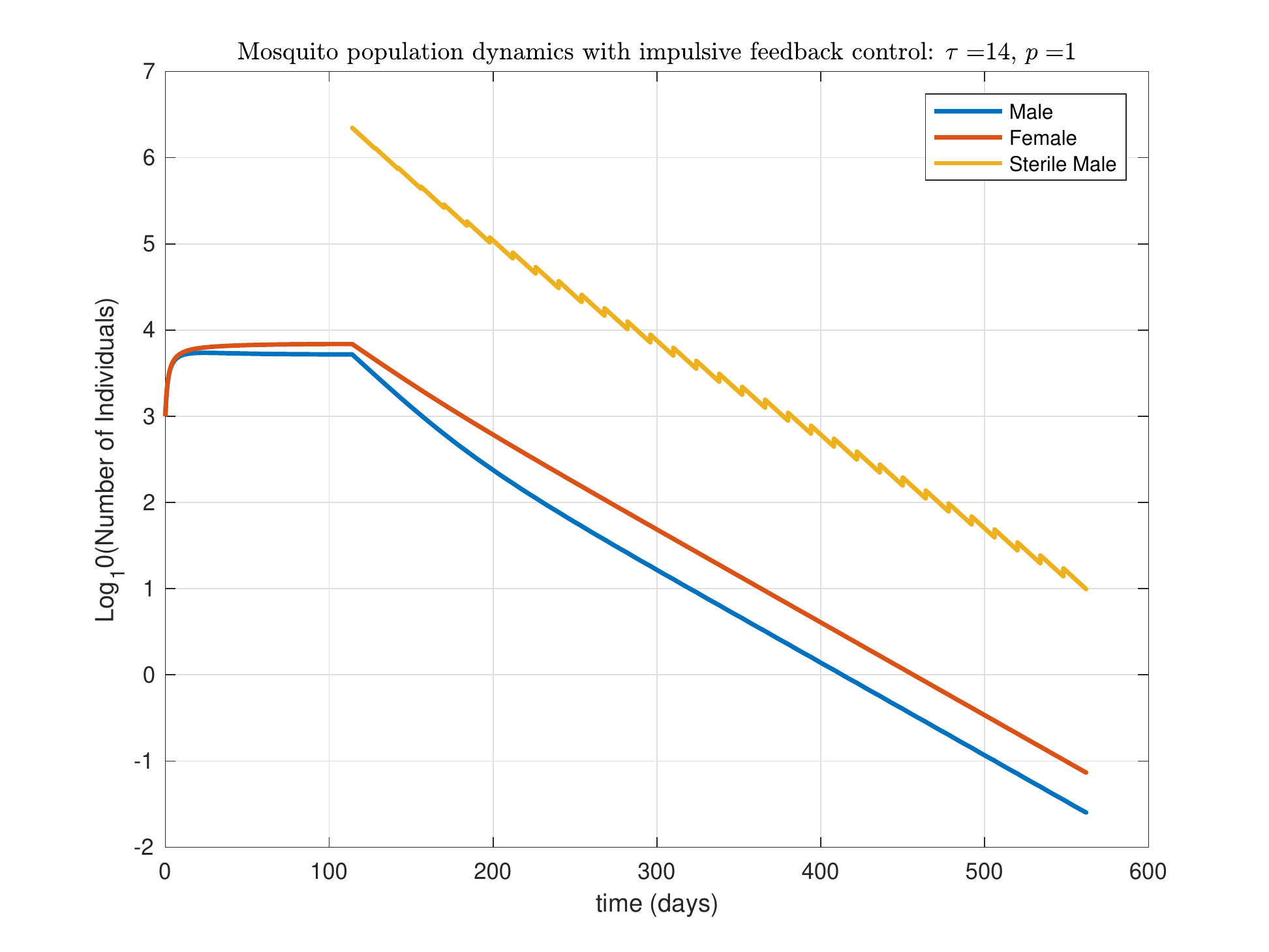} &
 \includegraphics[width=.5\textwidth]{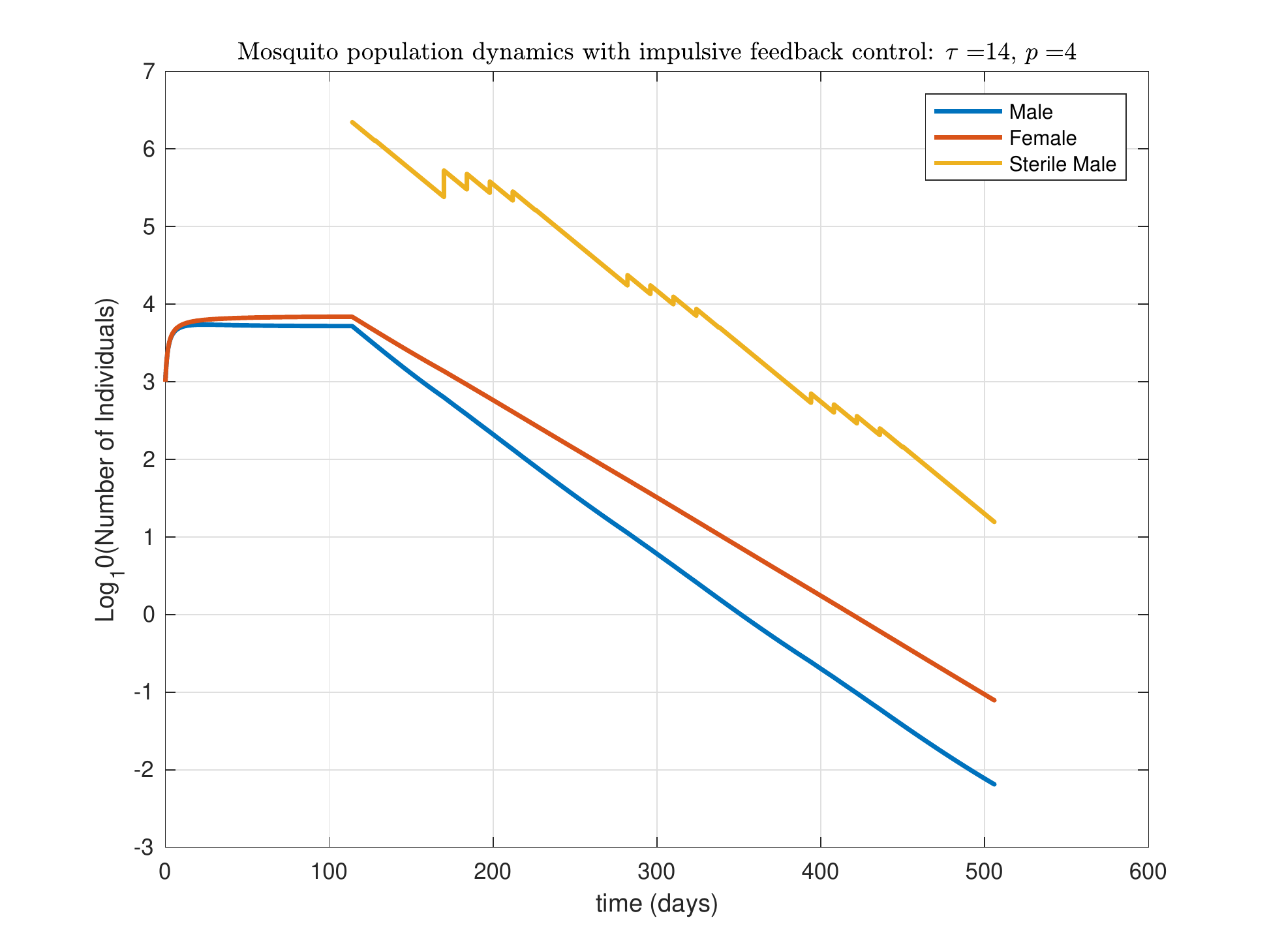} \\
 (c) & (d)
 \end{tabular}
 \caption{Closed-loop periodic impulsive  SIT control of system \eqref{system_SIT_imp} with $k=\dfrac{0.2}{\cN_F}$: (a) $7$ days, $p=1$; (b) $7$ days, $p=4$; (c) $14$ days, $p=1$; (d) $14$ days, $p=4$. See Table \ref{tab21}, page \pageref{tab21}}
 \label{fig:2a}
\end{figure}
\begin{table}[!ht]
\centering
\begin{tabular}{|c|c|c|>{\centering}p{1.5cm}|>{\centering}p{1.5cm}|c|}   \hline
 & \multicolumn{2}{|c|}{Cumulative Nb of}  & \multicolumn{2}{c|}{Nb of weeks needed} &  \\
  & \multicolumn{2}{|c|}{released sterile males}  & \multicolumn{2}{c|}{to reach elimination} & Nb of \tcm{nonzero} releases \\
   \hline
\backslashbox{Period}{p}  & $1$ & $4$ & $1$ & $4$ & $4$\\
   \hline
$\tau=7$ & $2,251,052$ & $4,363,430$ & $64$ & $54$ & $34$\\
\hline
$\tau=14$ & $2,390,676$ & $2,896,835$ & 64 & 56 & 17\\
\hline
\end{tabular}
\caption{Cumulative number of released sterile males and number of releases for each closed-loop periodic SIT control treatment when $k=\dfrac{1}{5} \cN_F$. See Figure \ref{fig:2a}, page \pageref{fig:2a}.}
\label{tab21}
\end{table}

\begin{figure}[!t]
\begin{tabular}{cc}
 \includegraphics[width=.5\textwidth]{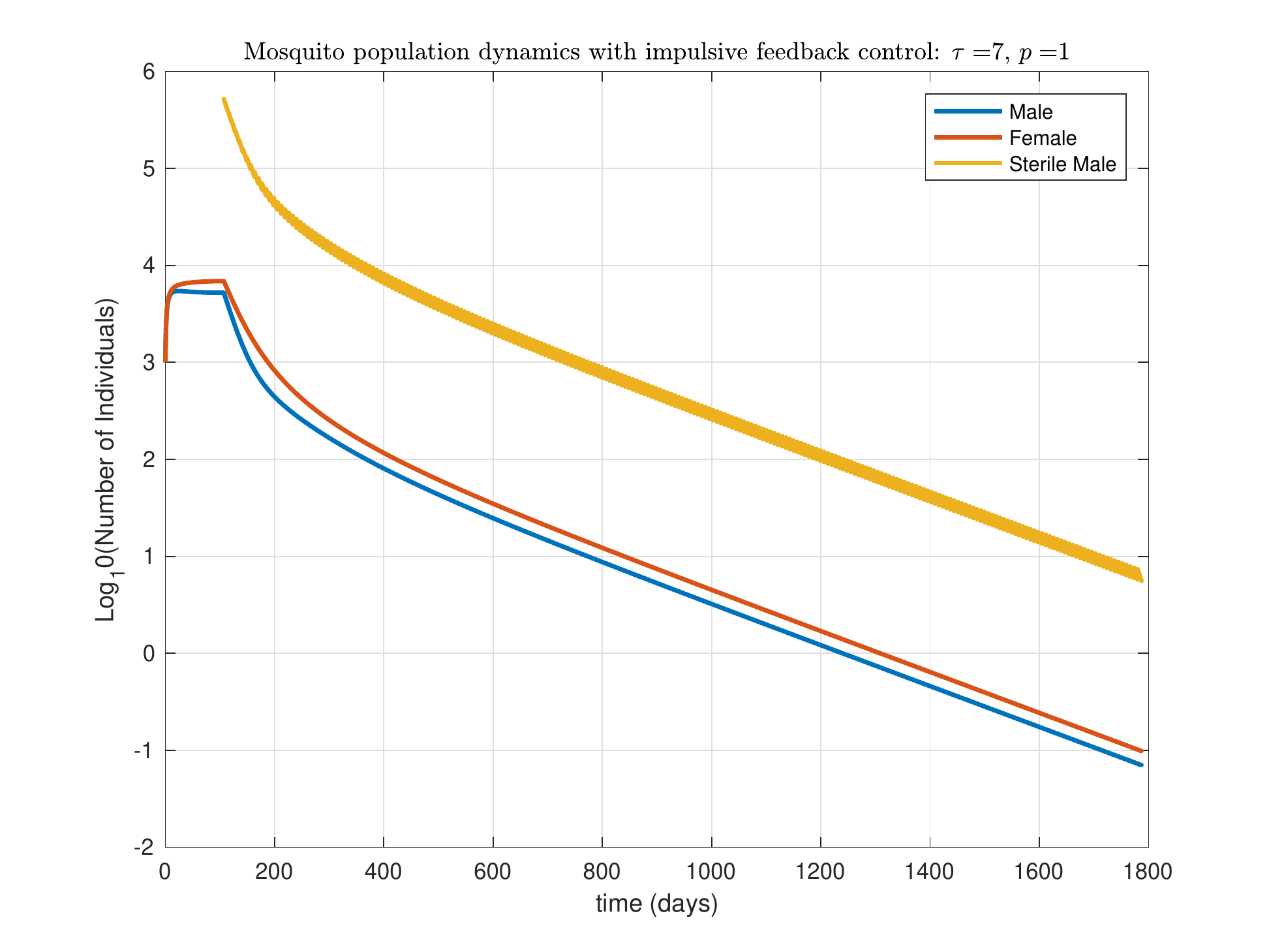} &  \includegraphics[width=.5\textwidth]{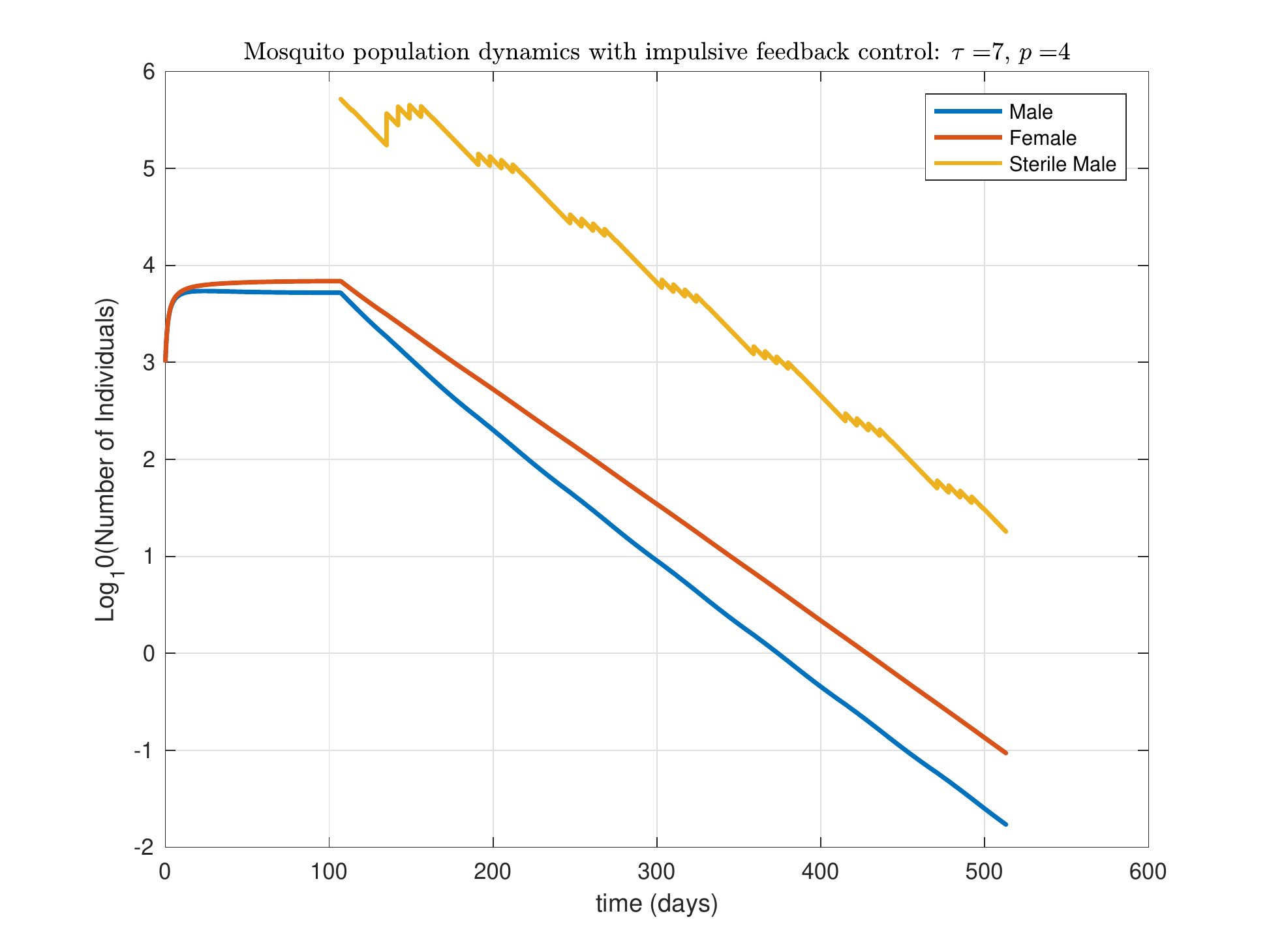} \\
 (a) & (b)  \\
  \includegraphics[width=.5\textwidth]{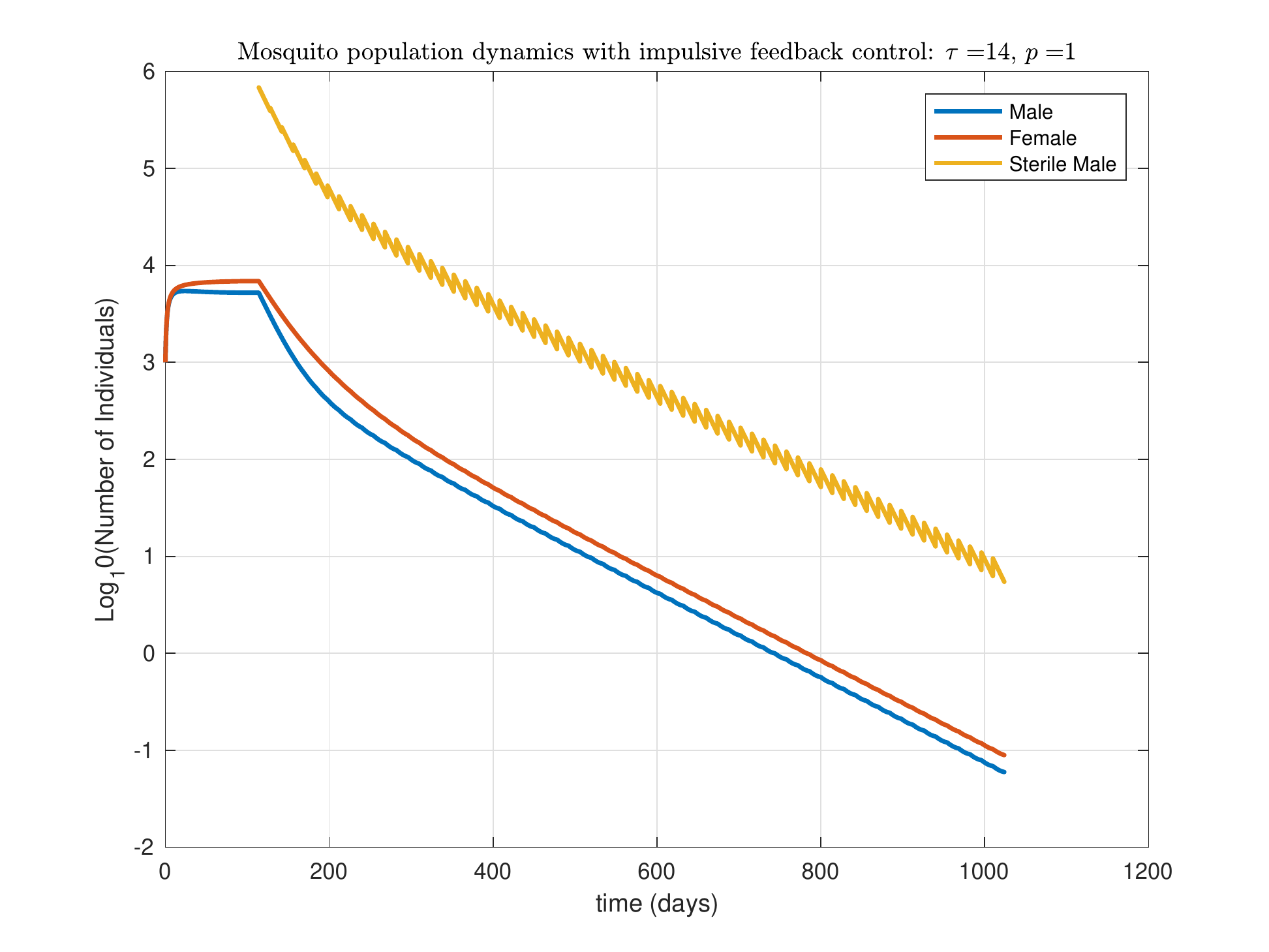} &  \includegraphics[width=.5\textwidth]{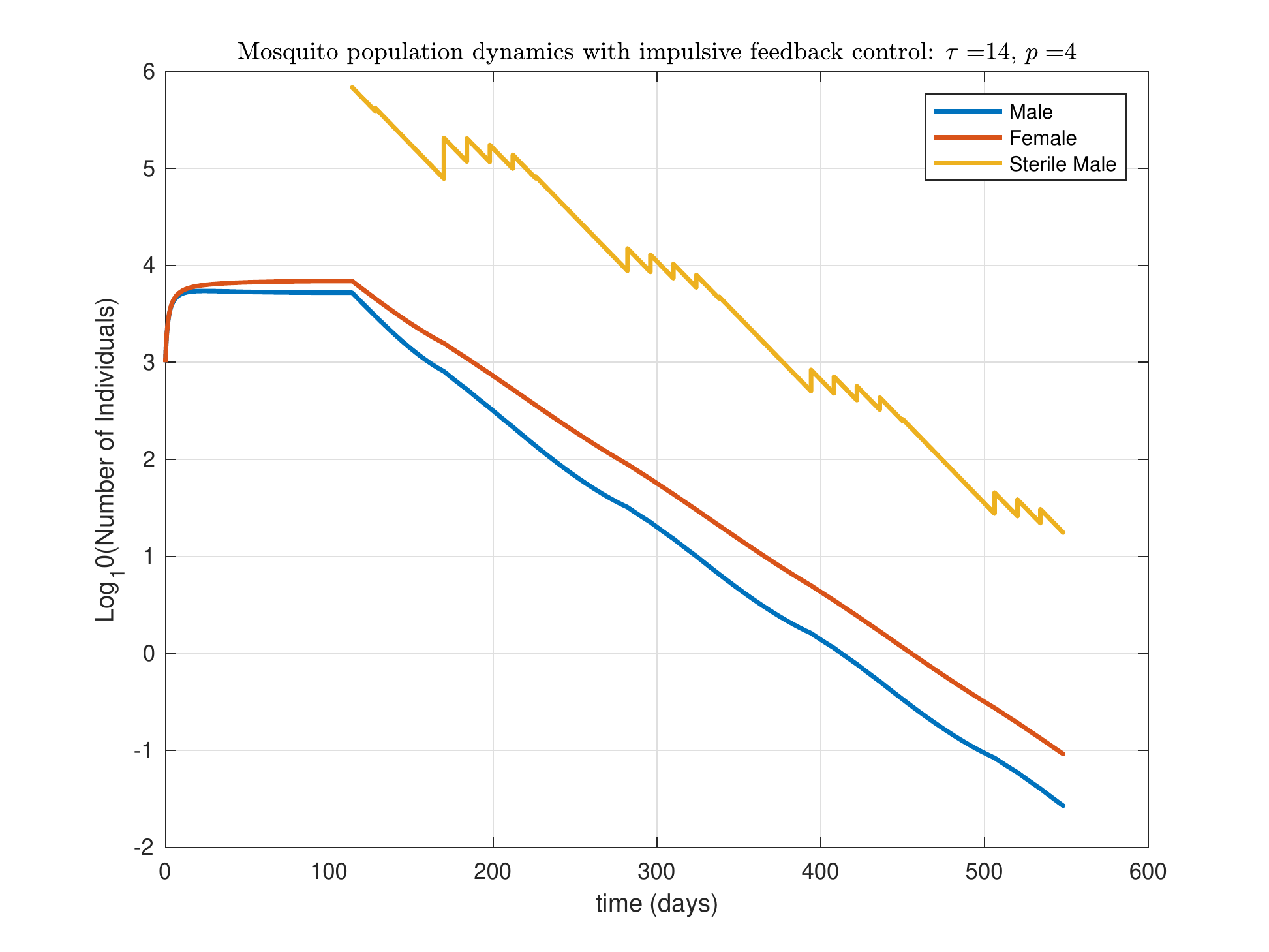} \\
  (c) & (d)
  \end{tabular}
 \caption{Closed-loop periodic impulsive  SIT control of system \eqref{system_SIT_imp} with $k=\dfrac{0.99}{\cN_F}$: (a) $7$ days, $p=1$; (b) $7$ days, $p=4$; (c) $14$ days, $p=1$; (d) $14$ days, $p=4$. See Table \ref{tab3}, page \pageref{tab3}}
 \label{fig:2b}
 \end{figure}

 \begin{table}[!ht]
\centering
\begin{tabular}{|c|c|c|>{\centering}p{1.5cm}|>{\centering}p{1.5cm}|c|}   \hline
 & \multicolumn{2}{|c|}{Cumulative Nb of}  & \multicolumn{2}{c|}{Nb of weeks needed} &  \\
  & \multicolumn{2}{|c|}{released sterile males}  & \multicolumn{2}{c|}{to reach elimination} & Nb of \tcm{nonzero} releases \\ \hline
 \backslashbox{Period}{p} & $1$ & $4$ & $1$ & $4$ & $4$ \\
 \hline
$\tau=7$  & $794,807$ & $1,221,593$ & 240  & 58 & 37 \\
\hline
$\tau=14$ & $ 909,344$ & $1,043,107$ & 130  & 62  & 20  \\
\hline
\end{tabular}
\caption{Cumulative number of released sterile males and number of releases for each closed-loop periodic SIT control treatment when $k=\dfrac{0.99}{\cN_F}$. See Figure \ref{fig:2b}, page \pageref{fig:2b}}
\label{tab3}
\end{table}

 \begin{figure}[!t]
 \begin{tabular}{cc}
 \includegraphics[width=.5\textwidth]{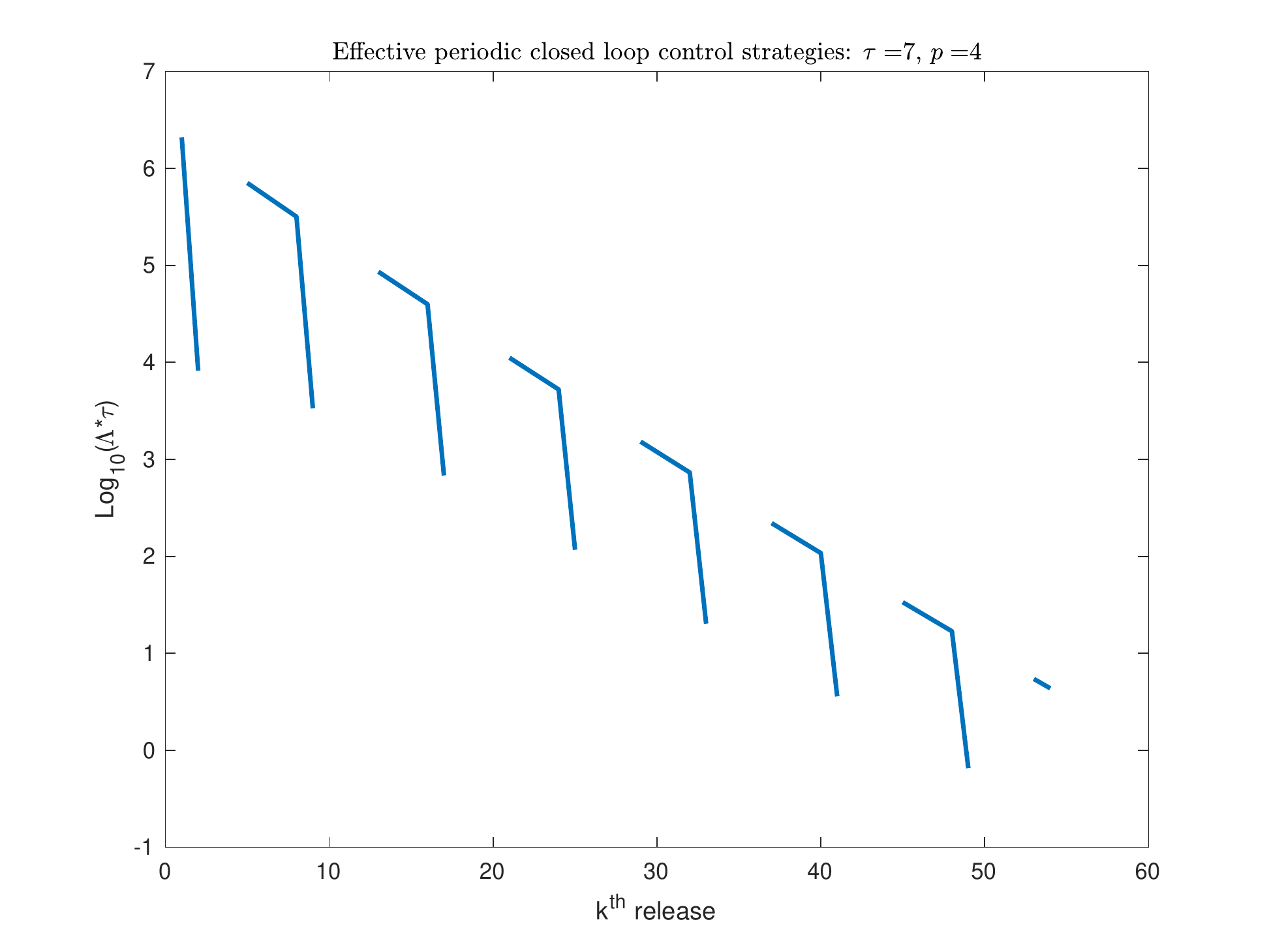} &  \includegraphics[width=.5\textwidth]{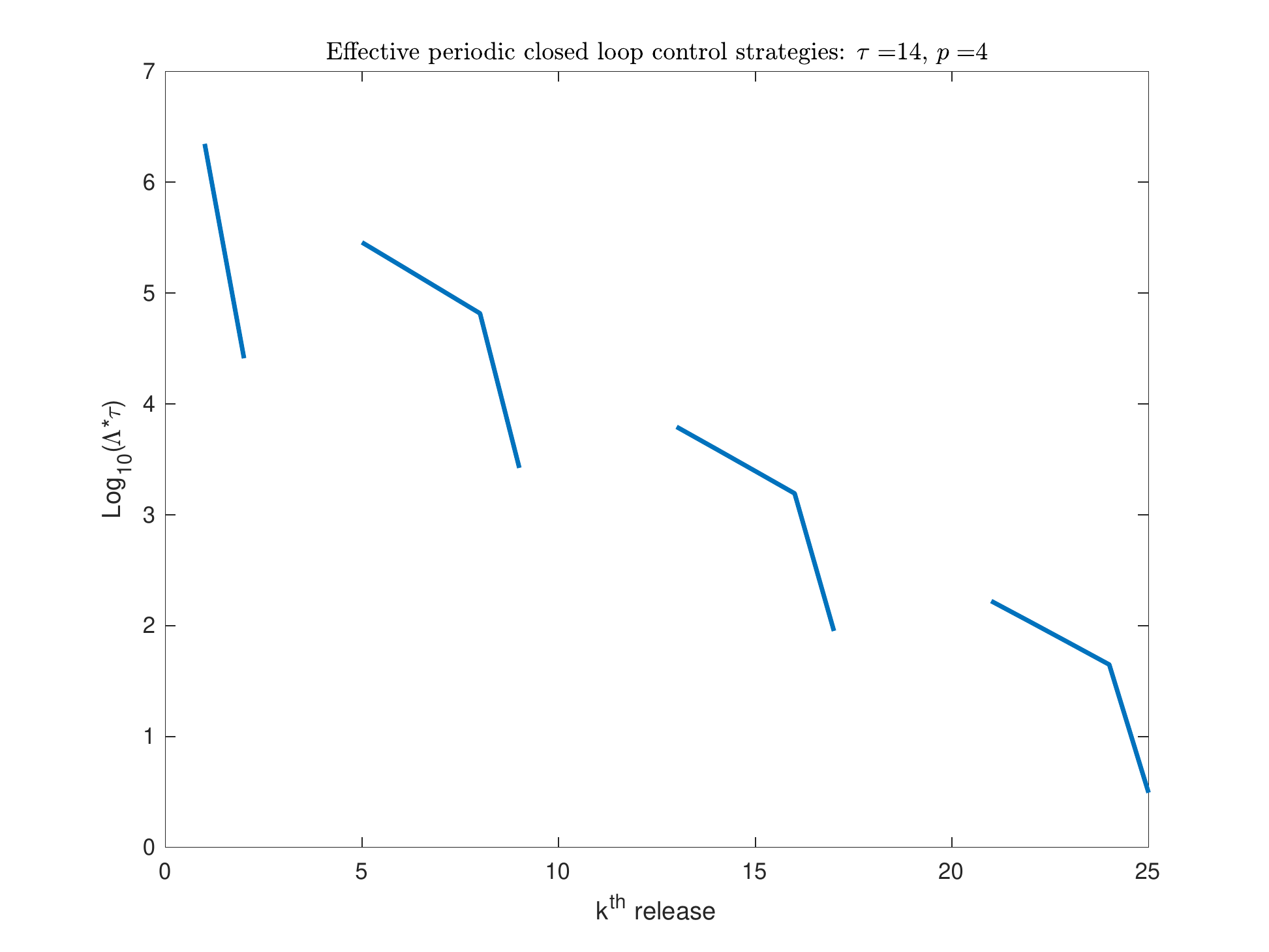}\\
 (a) & (b) \\
  \includegraphics[width=.5\textwidth]{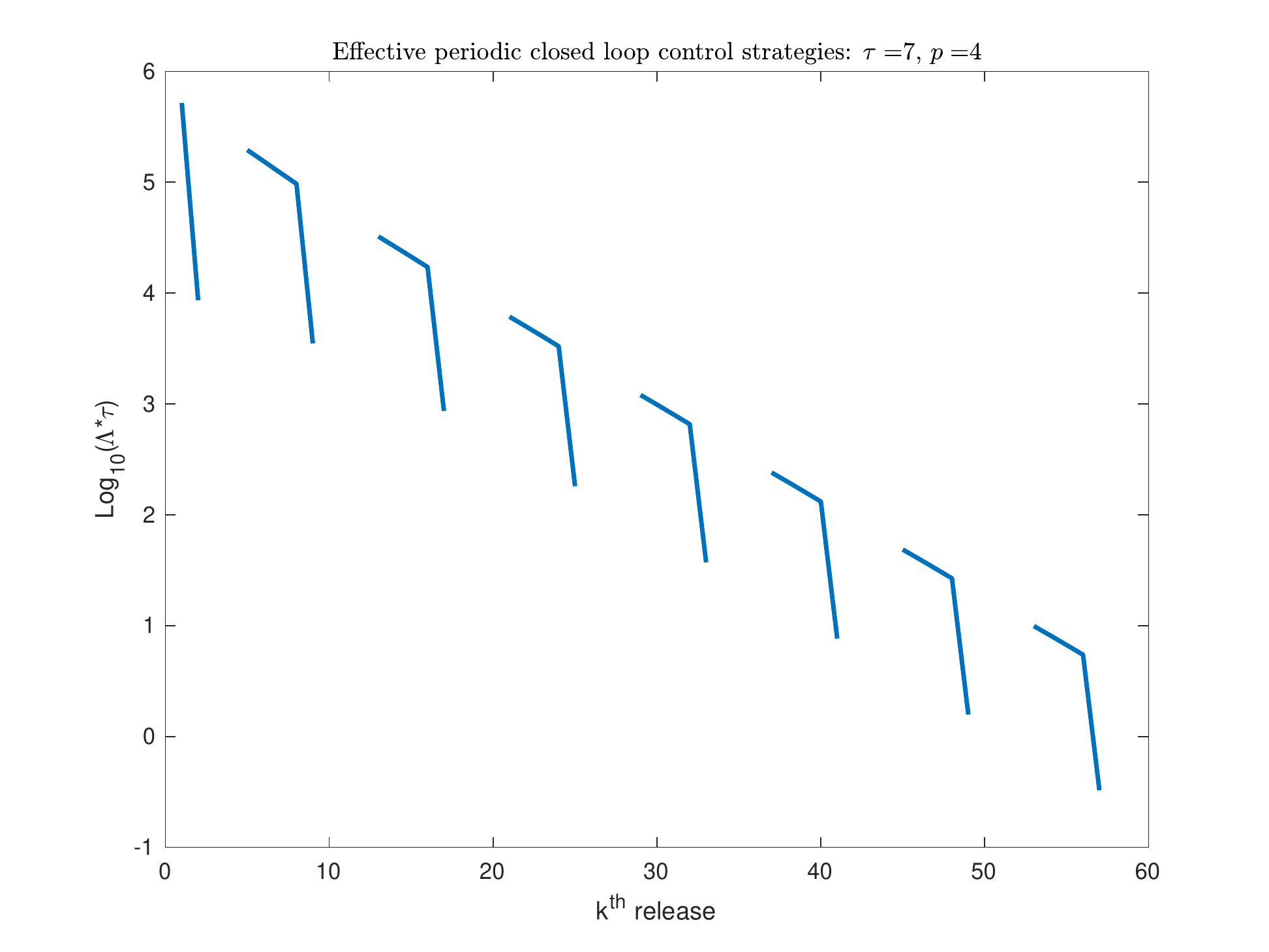} &  \includegraphics[width=.5\textwidth]{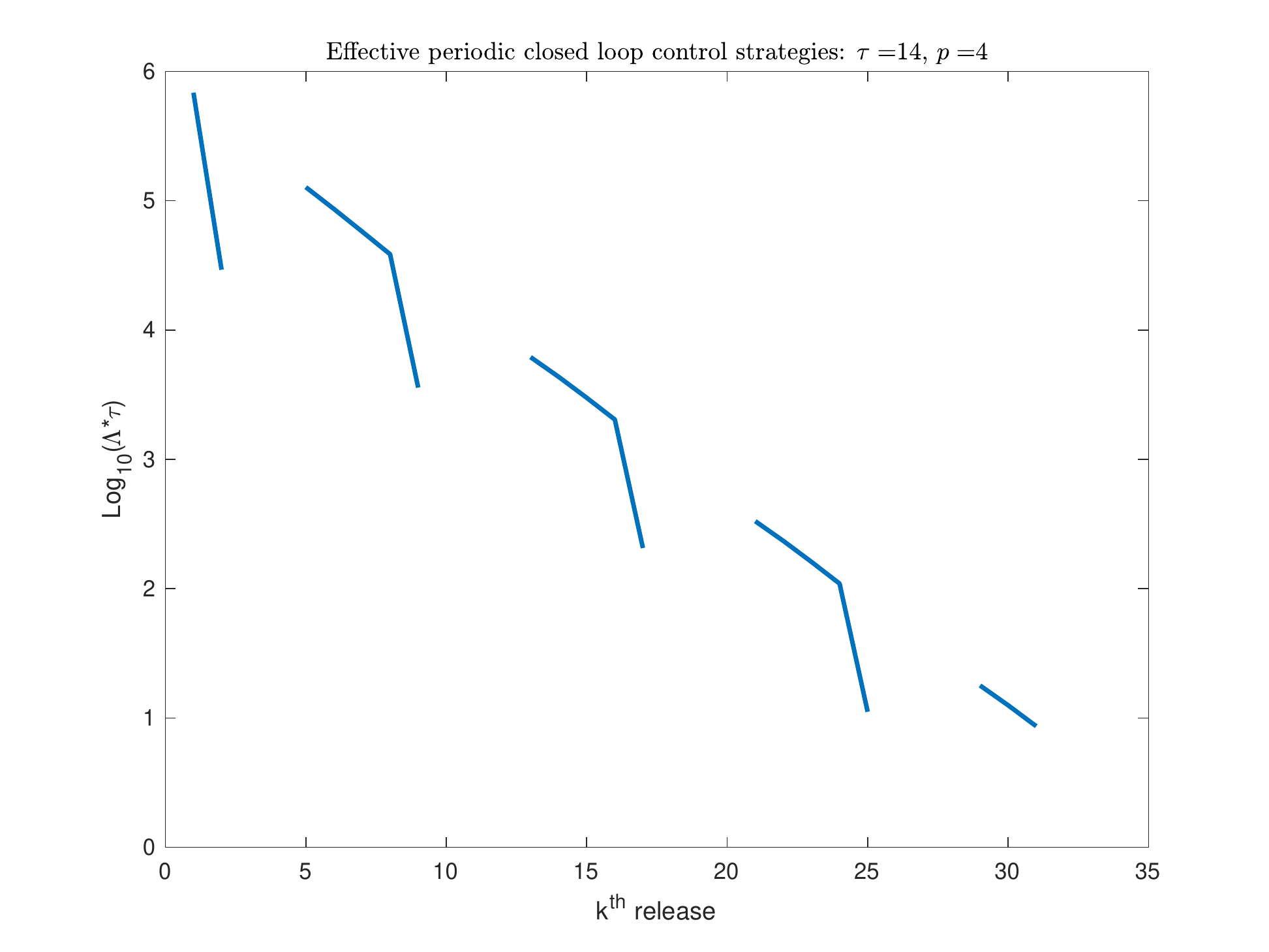} \\
  (c) & (d)
  \end{tabular}
 \caption{ Size of the release, $\Lambda_n$ at time $t=n\tau$ for closed-loop SIT control: (a,b) $k=\dfrac{0.2}{\cN_F}$; (c,d) $k=\dfrac{0.99}{\cN_F}$.  The discontinuities indicate ``no release''.}
 \label{fig:4-5}
 \end{figure}

We now consider mixed control strategies as exposed in Section \ref{se6}. In Figures \ref{fig:3a} and \ref{fig:3b} (pages \pageref{fig:3a} and \pageref{fig:3b}, respectively) we derive the simulations with {\color{black} the two} underlying values of $k$ given in \eqref{k}.
\begin{figure}[!t]
\begin{tabular}{cc}
 \includegraphics[width=.5\textwidth]{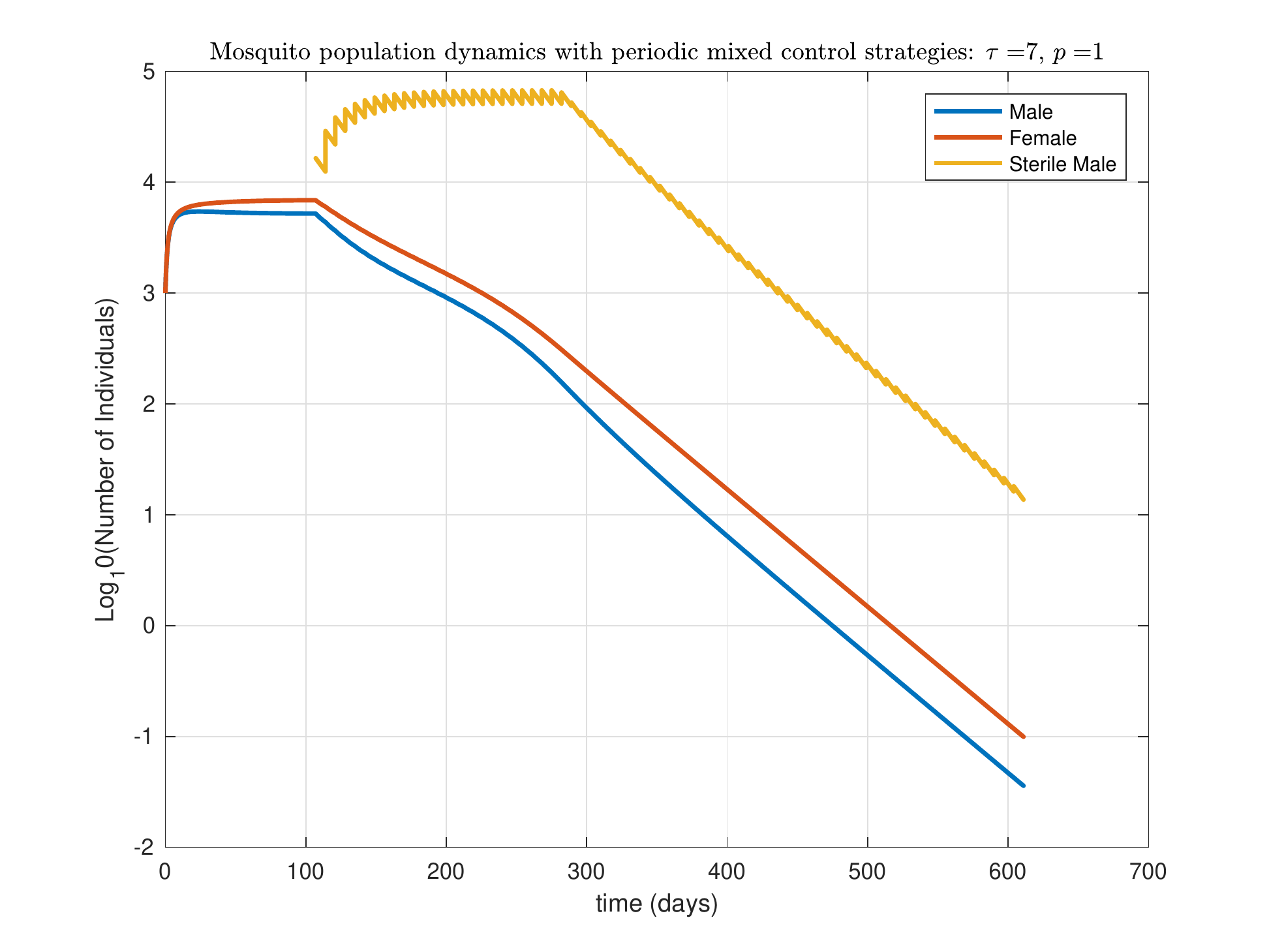} &  \includegraphics[width=.5\textwidth]{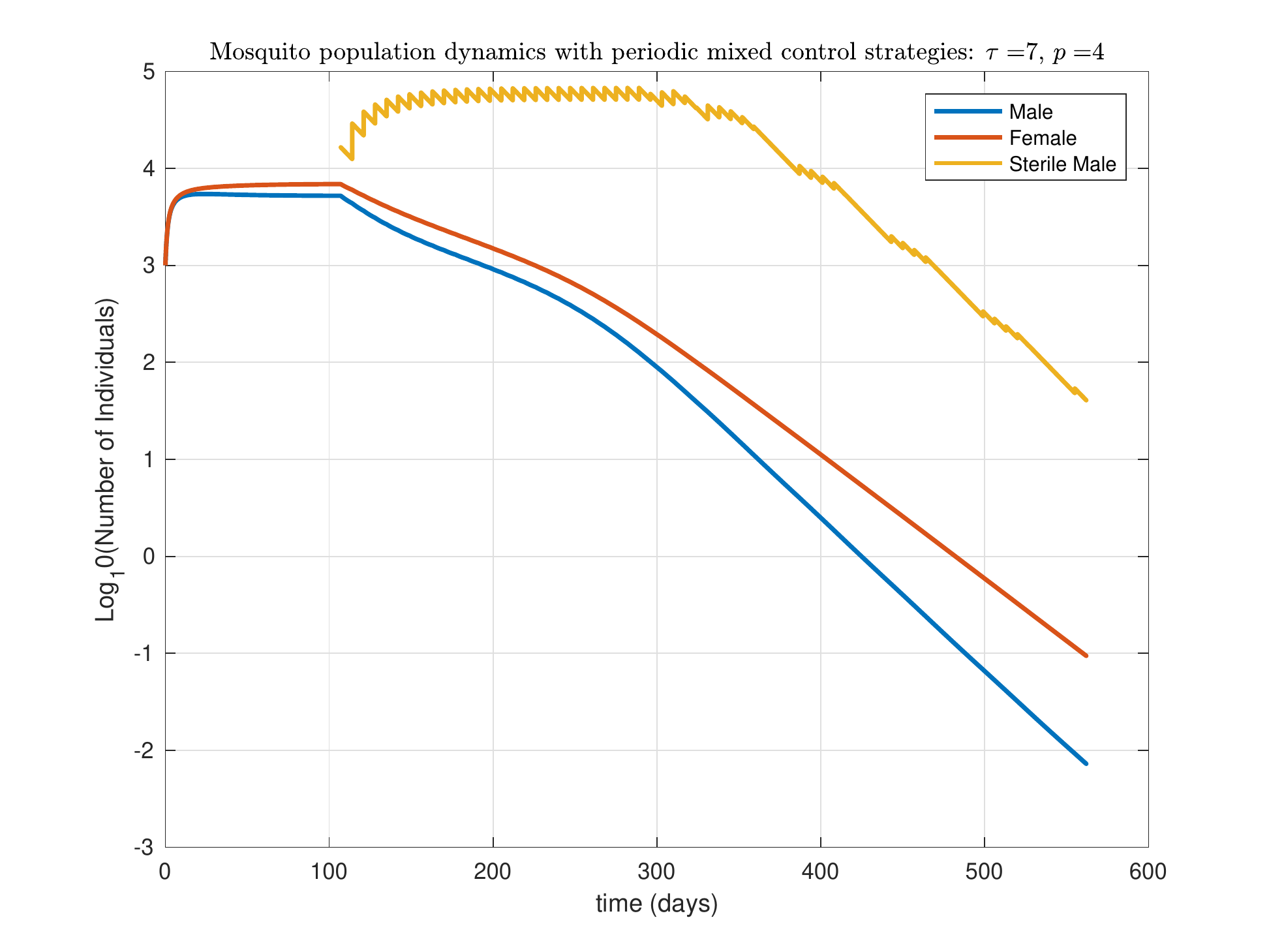} \\
 (a) & (b) \\
  \includegraphics[width=.5\textwidth]{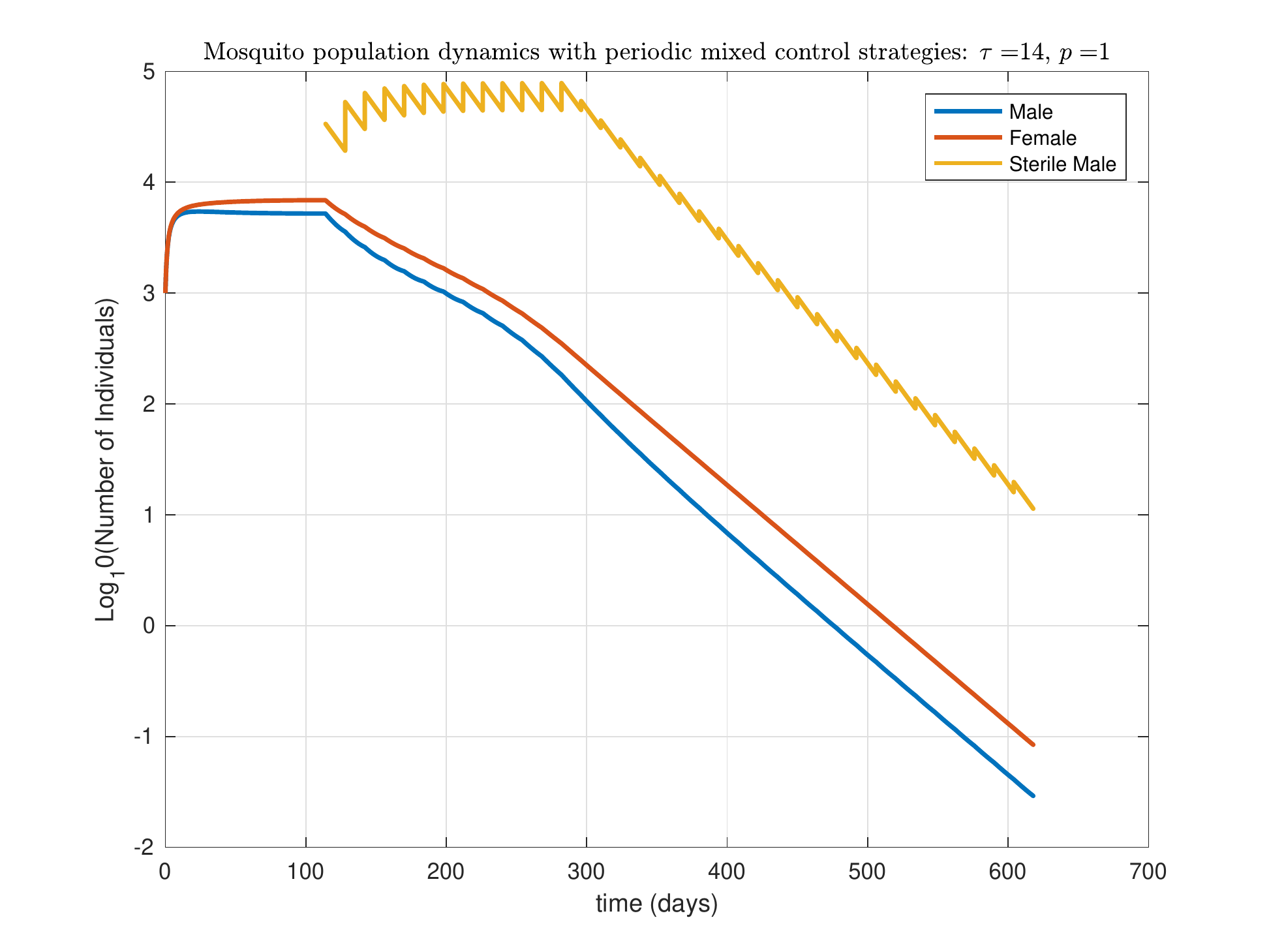} &  \includegraphics[width=.5\textwidth]{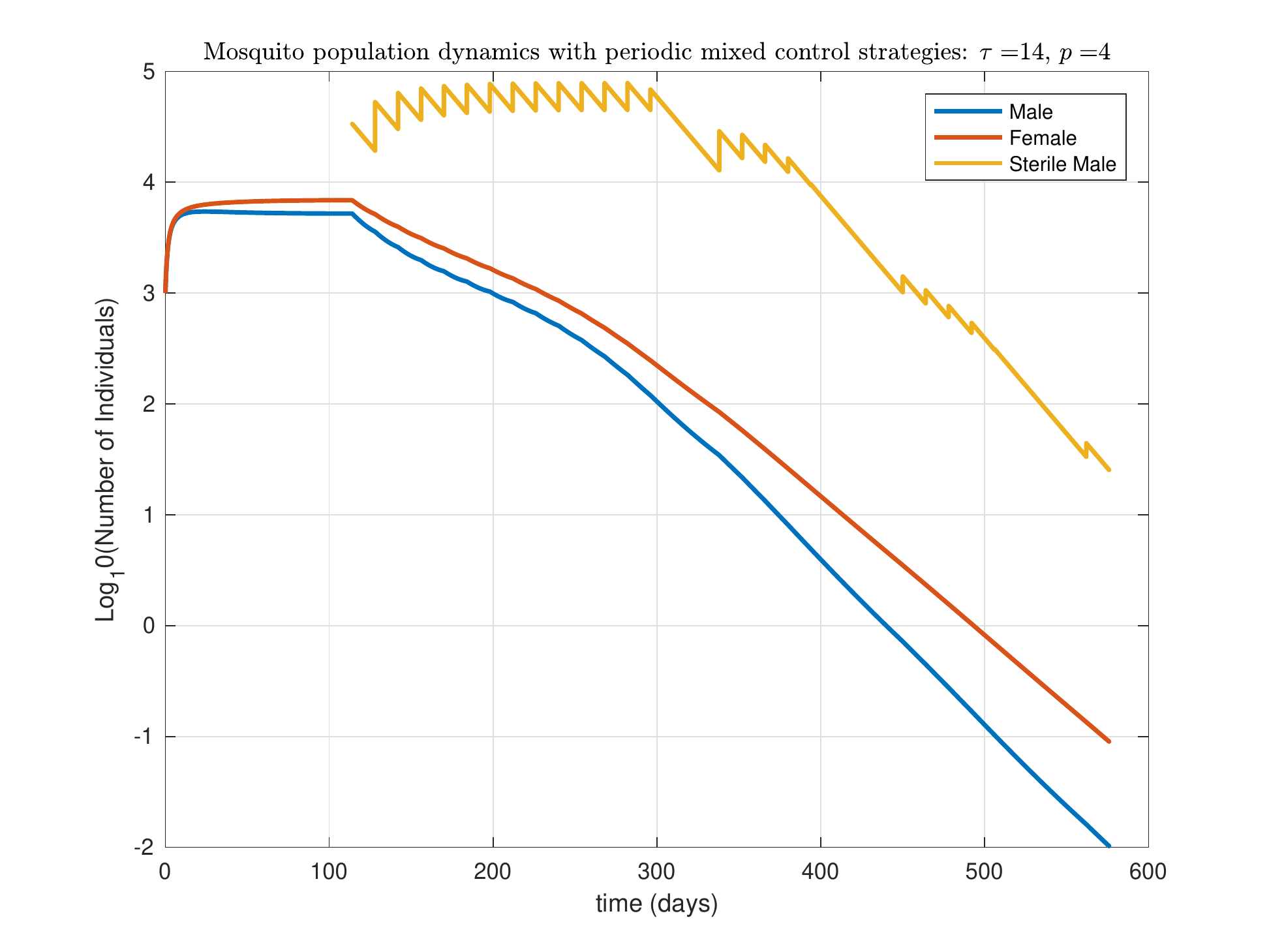} \\
  (c) & (d)
  \end{tabular}
 \caption{Combination of open/closed-loop periodic impulsive SIT control of system \eqref{system_SIT_imp} with $k=\dfrac{0.2}{\cN_F}$: (a) $7$ days, $p=1$; (b) $7$ days, $p=4$; (c) $14$ days, $p=1$; (d) $14$ days $p=4$. See Table \ref{tab4a}, page \pageref{tab4a}.}
 \label{fig:3a}
 \end{figure}

\begin{table}[!ht]
\centering
\begin{tabular}{|c|c|c|>{\centering}p{1.5cm}|>{\centering}p{1.5cm}|c|}   \hline
 & \multicolumn{2}{|c|}{Cumulative Nb of}  & \multicolumn{2}{c|}{Nb of weeks needed} &  \\
  & \multicolumn{2}{|c|}{released sterile males}  & \multicolumn{2}{c|}{to reach elimination} & Nb of \tcm{nonzero} releases \\
 \hline
 \backslashbox{Period}{p} & $1$ & $4$ & $1$ & $4$ & $4$ \\
 \hline
$\tau=7$  & $ 450,668$ & $534,849$ & 72 & 65 & 53 \\
\hline
$\tau=14$  & $465,187$ & $499,497$ &  72 & 66 & 25 \\
\hline
\end{tabular}
\caption{Cumulative number of released sterile males and number of releases for each mixed open/closed-loop periodic SIT control treatment when $k=\dfrac{0.2}{\cN_F}$. See Figure \ref{fig:3a}, page \pageref{fig:3a}.}
\label{tab4a}
\end{table}

\begin{figure}[!t]
\begin{tabular}{cc}
 \includegraphics[width=.5\textwidth]{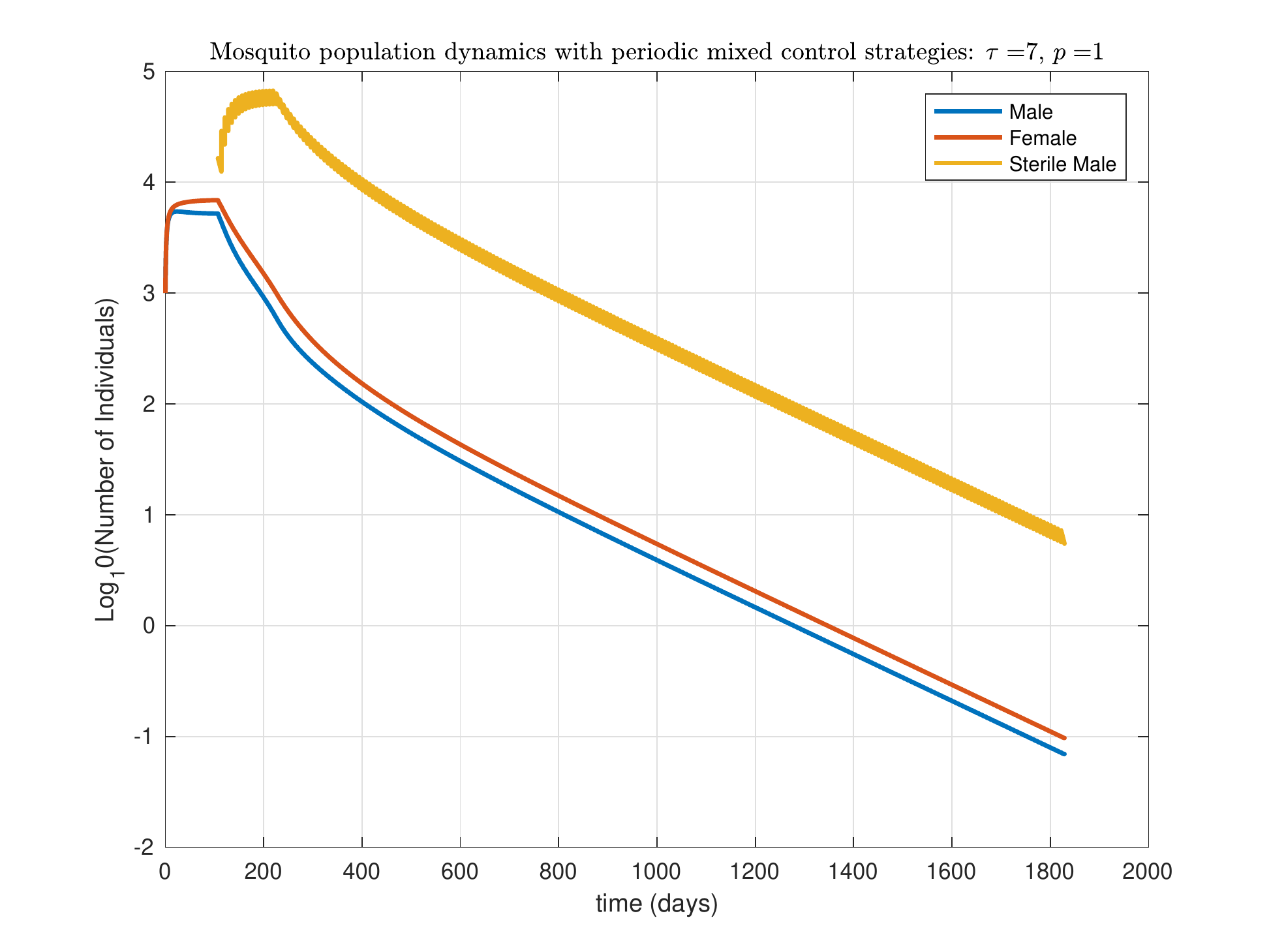} &
 \includegraphics[width=.5\textwidth]{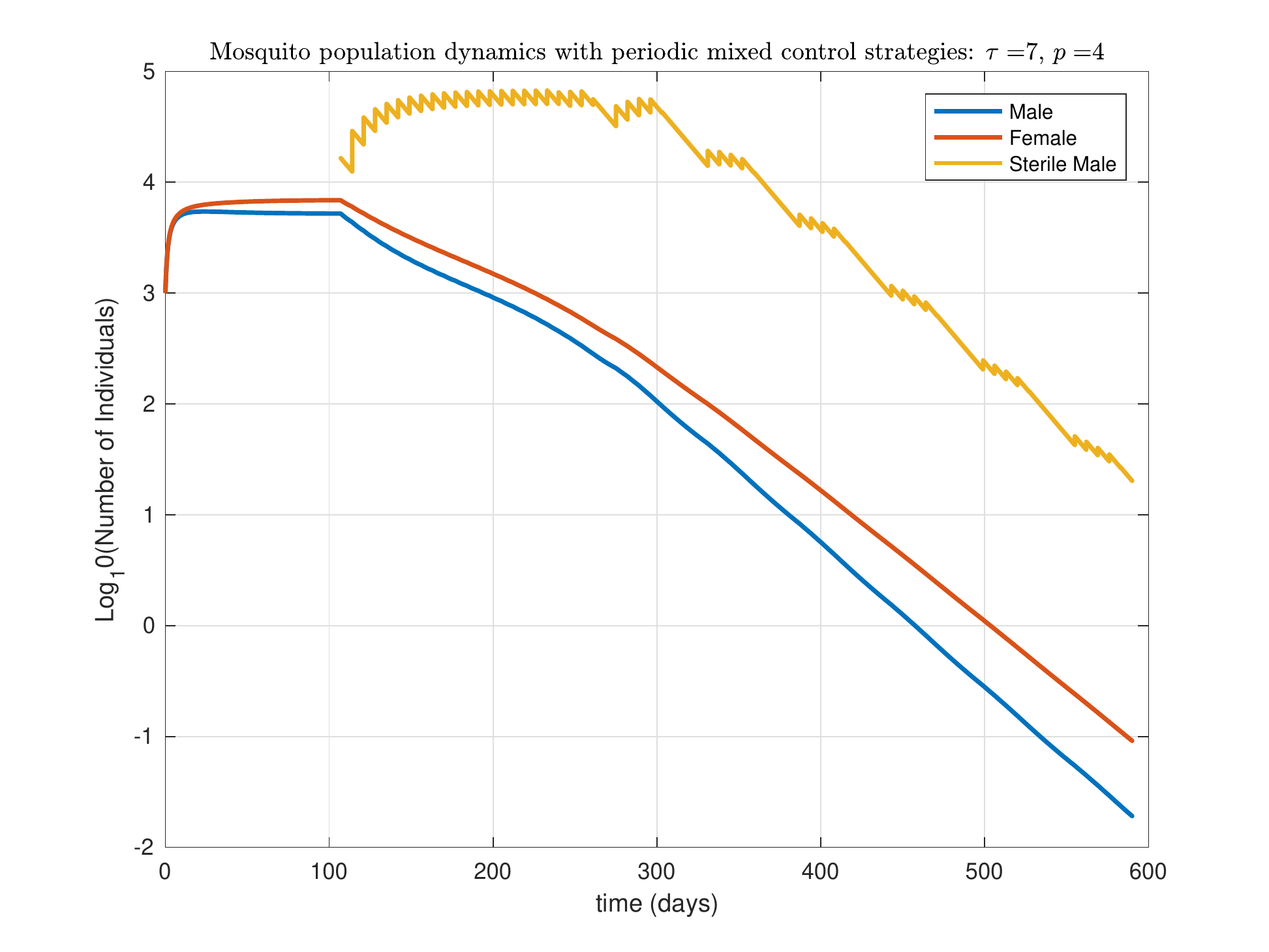} \\
  (a) & (b) \\
  \includegraphics[width=.5\textwidth]{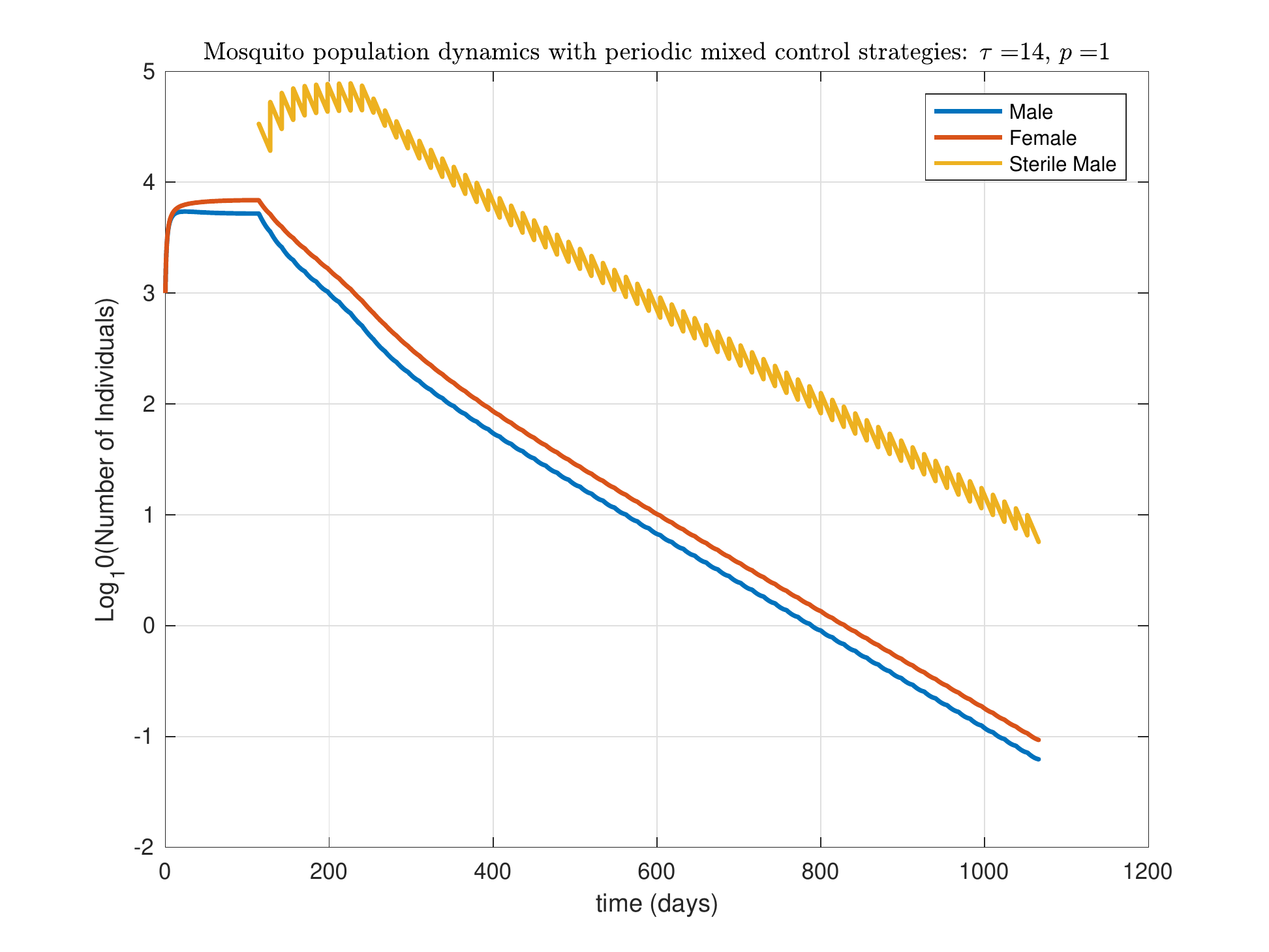}&
 \includegraphics[width=.5\textwidth]{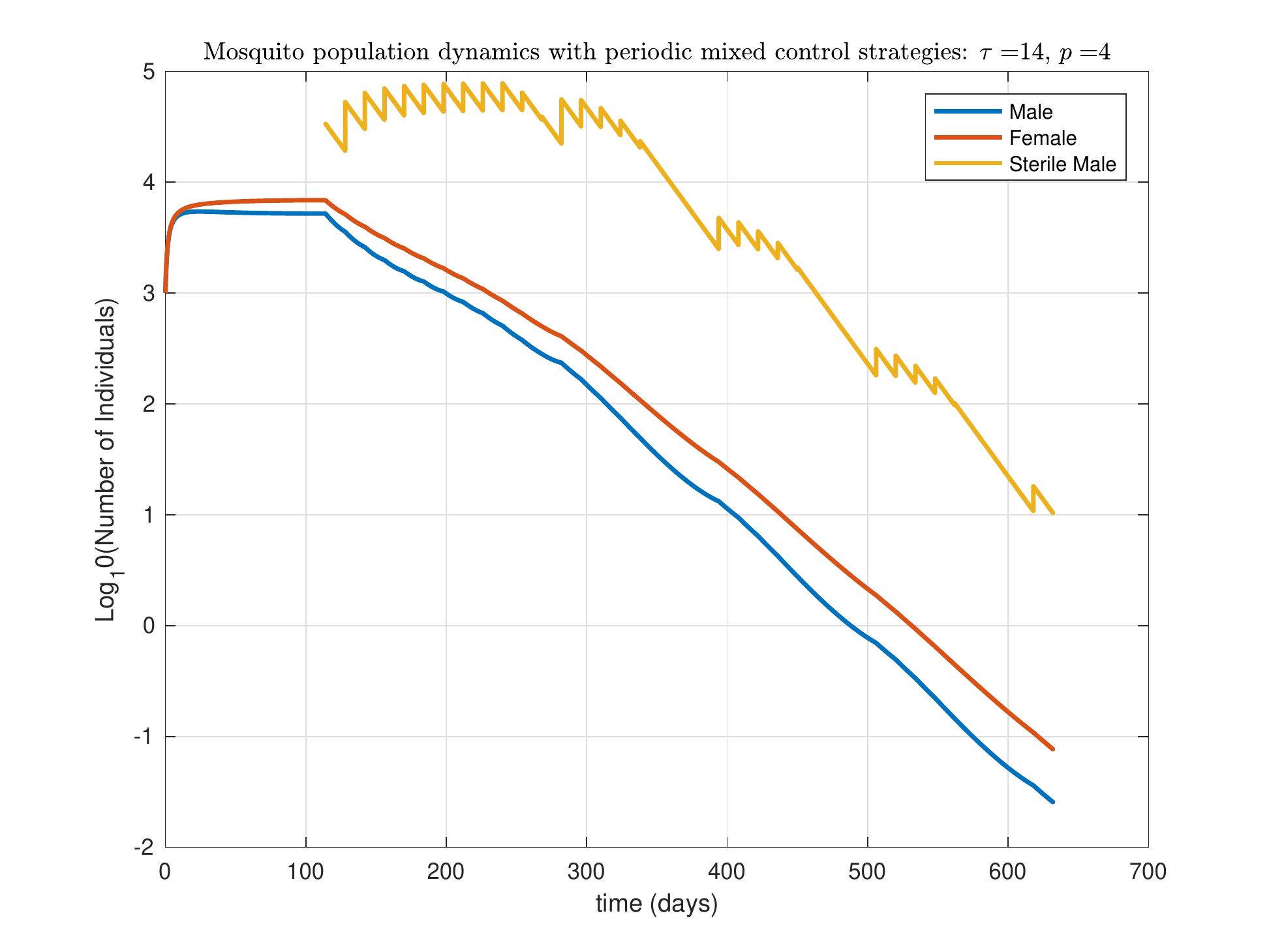} \\
  (c) & (d) \\
  \end{tabular}
 \caption{Combination of open/closed-loop periodic impulsive SIT control of system \eqref{system_SIT_imp}, with $k=\dfrac{0.99}{\cN_F}$: (a) $7$ days, $p=1$, (b) $7$ days, $p=4$, (c) $14$ days, $p=1$, (d) $14$ days $p=4$. See Table \ref{tab5}, page \pageref{tab5}.}
 \label{fig:3b}
 \end{figure}

 \begin{table}[!ht]
\centering
\begin{tabular}{|c|c|c|>{\centering}p{1.5cm}|>{\centering}p{1.5cm}|c|}   \hline
 & \multicolumn{2}{|c|}{Cumulative Nb of}  & \multicolumn{2}{c|}{Nb of weeks needed} &  \\
  & \multicolumn{2}{|c|}{released sterile males}  & \multicolumn{2}{c|}{to reach elimination} & Nb of {\color{black} nonzero} releases \\
 \hline
  \backslashbox{Period}{p} & $1$ & $4$ & $1$ & $4$ & $4$ \\
 \hline
 $\tau=7$  & $ 457,489$ & $450,077$ & 246 & 69 & 53 \\
\hline
$\tau=14$  & $427,701$ & $449,059$ & 136 & 74 &  28 \\
\hline
\end{tabular}
\caption{Cumulative number of released sterile males and number of releases for each mixed open/closed-loop periodic SIT control treatment when $k=\dfrac{0.99}{\cN_F}$. See Figure \ref{fig:3b}, page \pageref{fig:3b}.}
\label{tab5}
\end{table}

\begin{figure}[!t]
\begin{tabular}{cc}
\includegraphics[width=.5\textwidth]{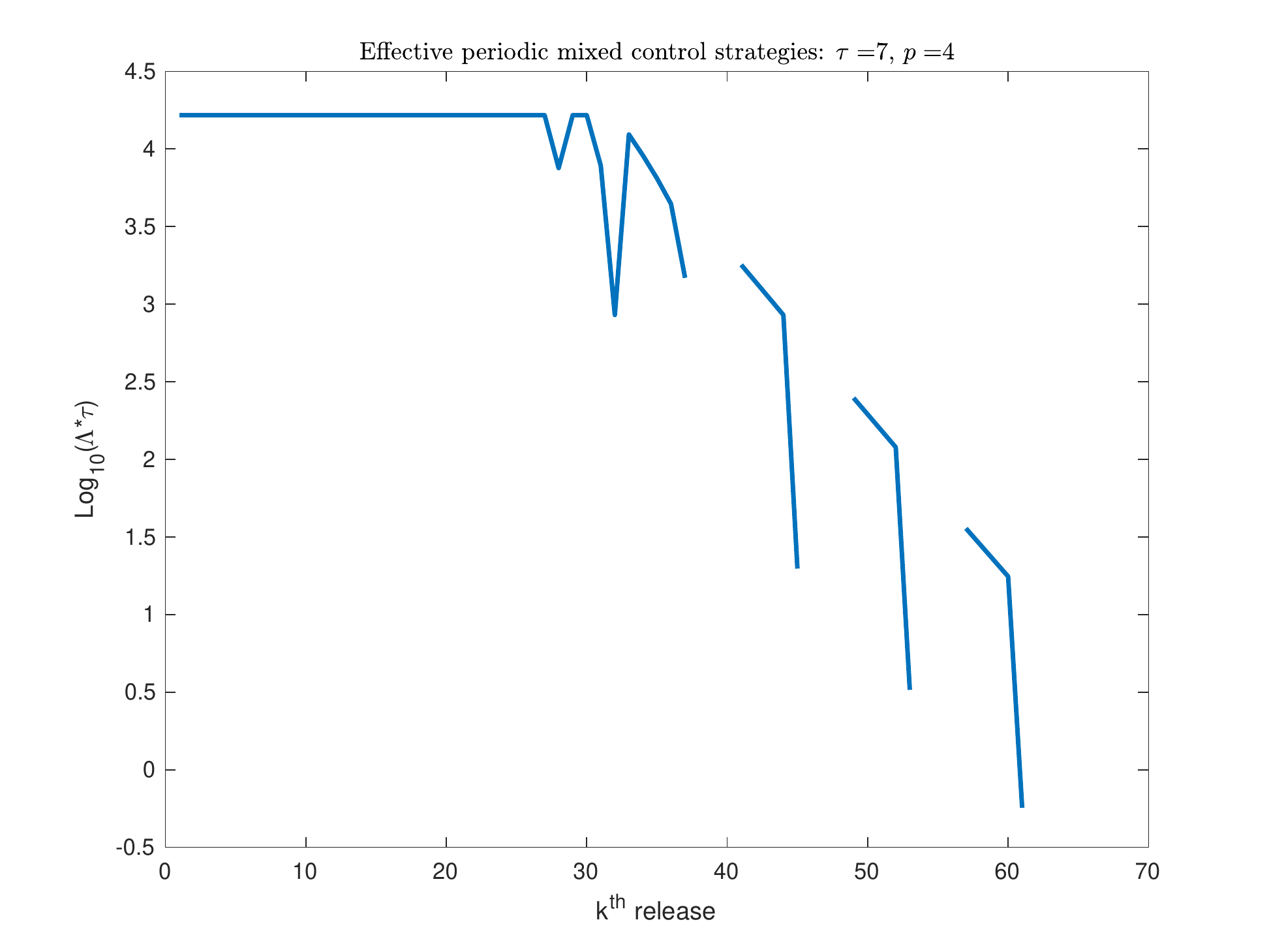} &  \includegraphics[width=.5\textwidth]{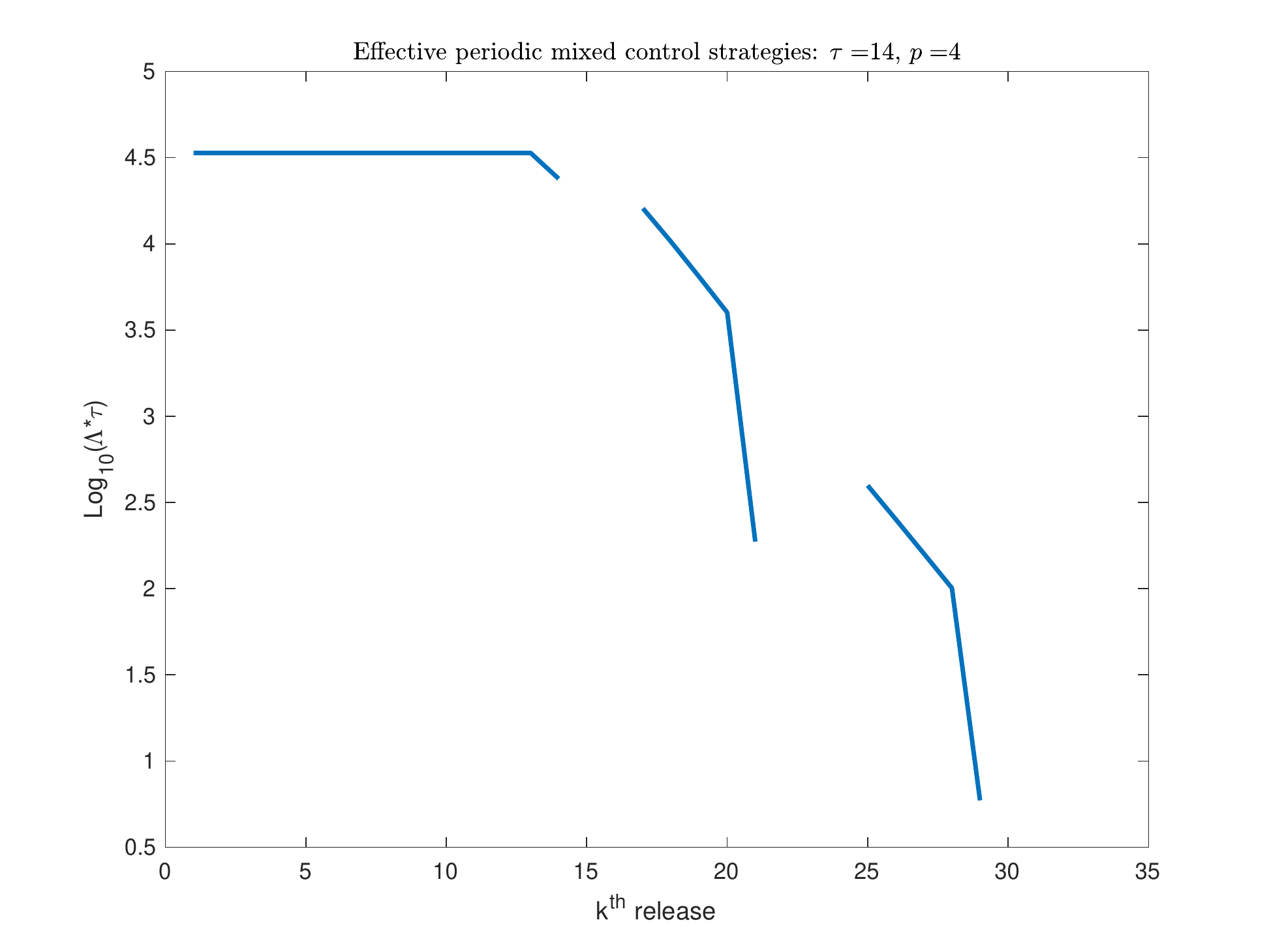} \\
 (a) & (b) \\
 \includegraphics[width=.5\textwidth]{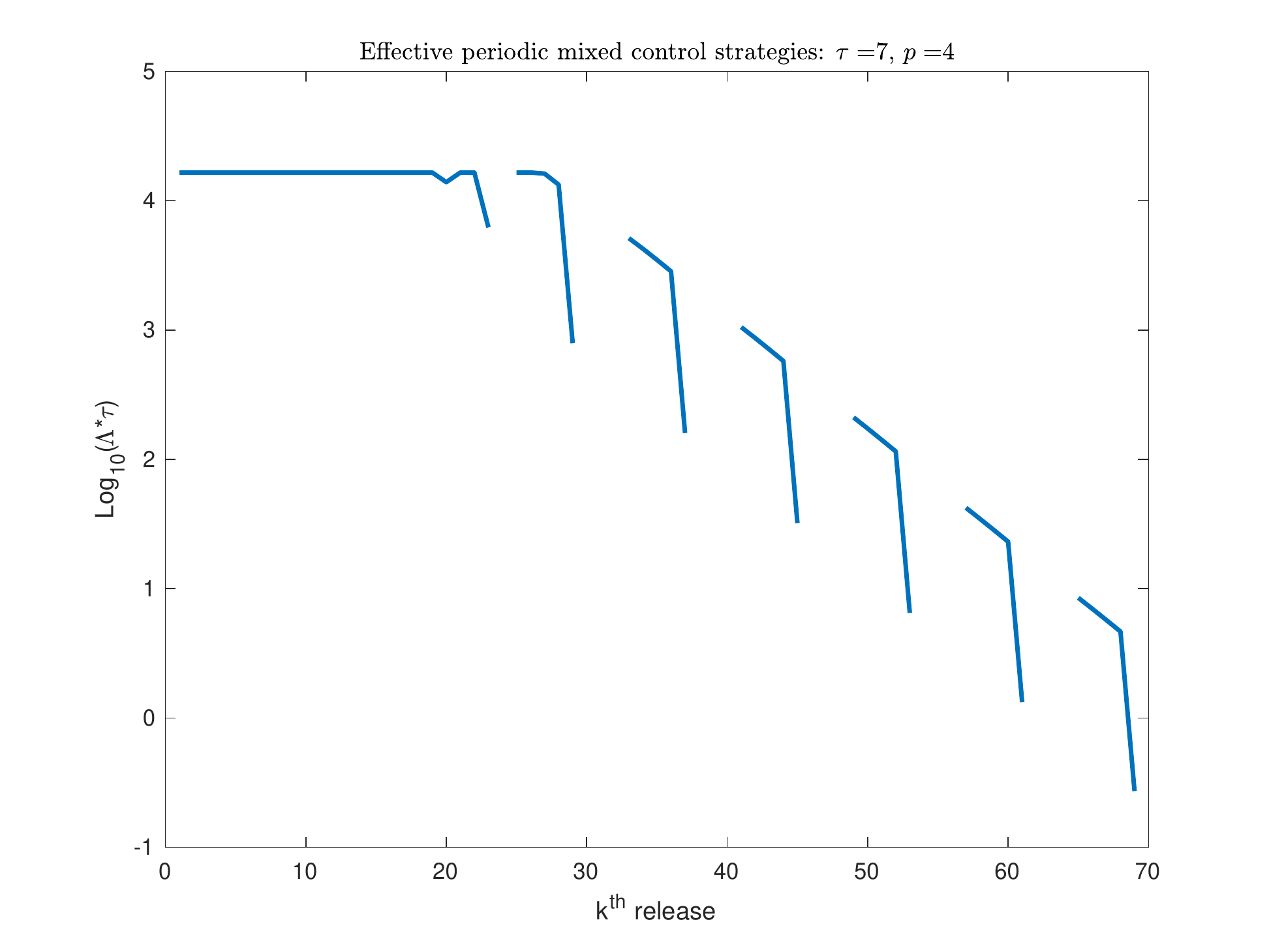} &   \includegraphics[width=.5\textwidth]{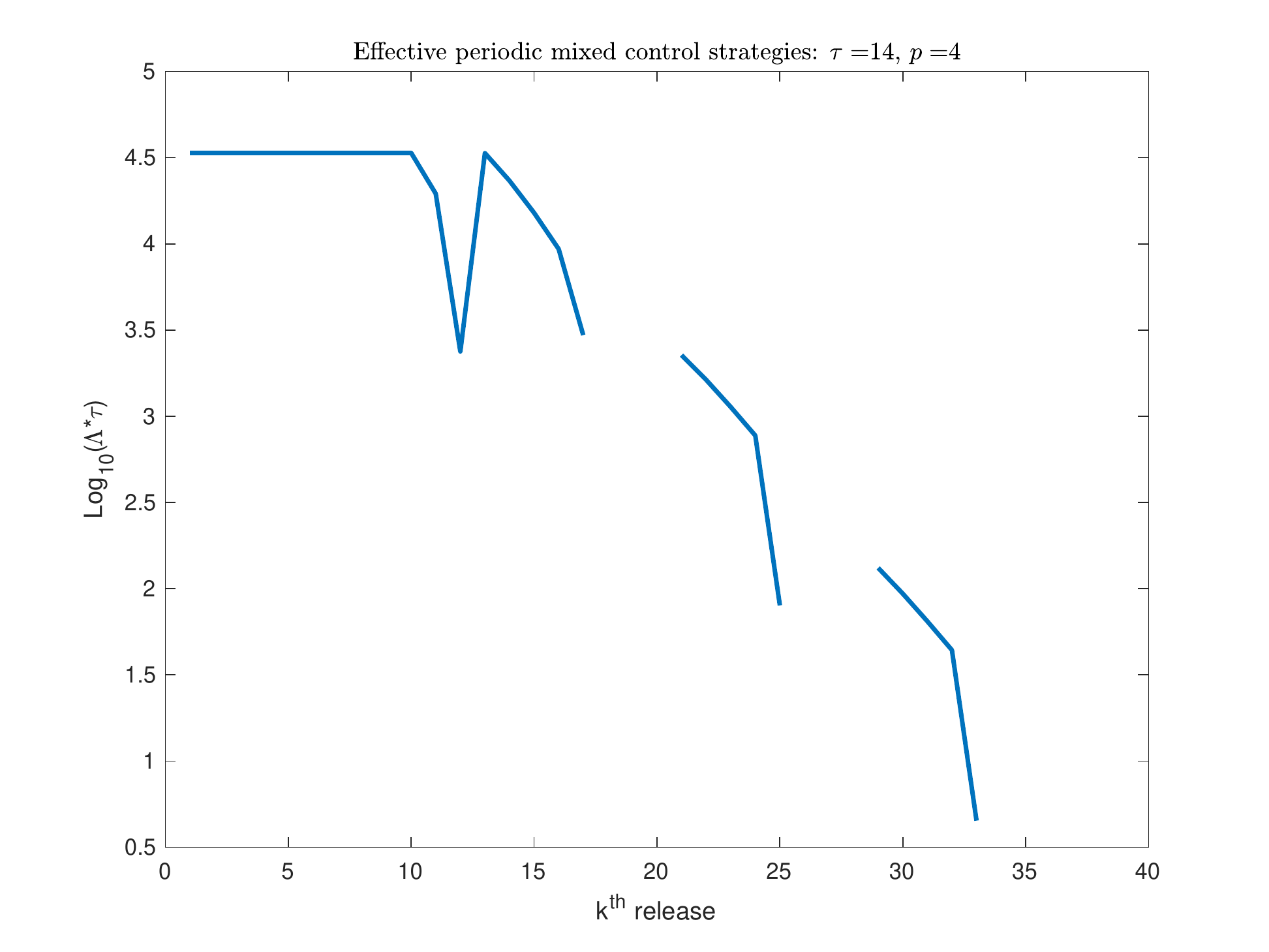} \\
 (c) & (d)
 \end{tabular}
 \caption{Size of the release, $\Lambda_n$, at time $t=n\tau$ for mixed open/closed-loop SIT control: (a,b) $k=\dfrac{0.2}{\cN_F}$; (c,d) $k=\dfrac{0.99}{\cN_F}$.  The discontinuities indicate ``no release''.}
 \label{fig:8}
 \end{figure}

Except for the case with $(\tau,p)=(7,1)$ and  $k=\dfrac{0.99}{\cN_F}$ (see  Table \ref{tab4a}, page \pageref{tab4a}), where the convergence to $E^*_0$ is slow, it turns out that the mixed open/closed-loop control strategies derive the best results, not only in terms of releases number but also in terms of overall cumulative number of sterile males to be released.

According to Tables {\color{black} \ref{tab4a} and \ref{tab5}} (pages \pageref{tab4a} and \pageref{tab5}, respectively), for both values of $k$, the optimal solution would be to release sterile insects every 2 weeks with  population assessments carried out by MRR experiments  every 4 weeks ($p=4$). In addition, and thanks to \eqref{eq450practical}, Figure \ref{fig:8} displays the release sizes $\Lambda_n$ for each mixed strategy. It clearly shows that during the first releases, $\Lambda_n=\Lambda^{crit}_{per}$. Further, when the wild population drops below a certain threshold, the feedback control occurs or not, depending on the (estimated) size of the sterile male population. That is why in Tables \ref{tab4a} and \ref{tab5}, we derive the number of effective releases (only for the case $p=4$) which confirms that the best combination is $(\tau,p)=(14,4)$, regardless of the value of $k$.

For the mixed open/closed-loop periodic impulsive SIT control, the choice of $k$ does not matter compared to the closed-loop control only. Our preliminary results thereby indicate that a mixed SIT control option with $(\tau=14,p=4)$ leads to the best strategy in terms of the total number of released sterile males and also in terms of effective releases number.


\section{Conclusion}
In this work, we studied various strategies to control mosquito population using SIT: open-loop and closed-loop periodic impulsive control strategies, as well as their combination (mixed open/closed-loop strategy). For the open-loop strategy (that is usually considered during field experiments) we found the minimal {\color{black} number 
of sterile} males to be released every $\tau$ days in order to reach elimination of wild mosquitoes.
{\color{black} This number is constant and relatively low.
The question of determining a stopping time for the release campaigns is not simple, but clearly of primordial importance, as premature ending ruins the preceding efforts.}

{\color{black} On the contrary,}
the feedback SIT-control commences with relatively abundant releases and their amplitude steadily declines with the wild population size until fading away and vanishing when the system converges towards the desired mosquito-free state.
{\color{black} This closed-loop} control strategy requires to assess the current size of the wild population (using MRR experiments, for instance).

Finally, we proposed a mixed control strategy, combining  open-loop and closed-loop strategies. This control input mode renders the best result, and turns out rather meaningful from the experimental standpoint: the control input is launched at the open-loop mode during first weeks (initial phase) and then is shifted to the closed-loop mode (final phase), once the size of wild population exhibits steady decline. With this approach, the gain in terms of {\color{black} release pick-value, number of nonzero releases, and overall cumulative volume is clearly visible.}
This is due to the fact that initial phase of control action is done at the open-loop mode, {\color{black} i.e.\ by performing releases of sterile} males regardless of the current size of wild population, what induces an essential decline of the wild population before switching to the closed-loop control mode.
{\color{black} Even considering the shown simulations} in terms of cost, the mixed control seems to be definitively the best choice with a release every two weeks and a population estimate every four weeks.

Knowledge of the cost of each stage of the SIT control (mass rearing, sterilization either by irradiation or using {\em Wolbachia}, transportation to the target locality, wild population measurements with MRR techniques, and other necessary supplies) will allow to estimate more precisely and optimize the treatment cost, and thus to make the most appropriate choices from an economical point of view.

As a last remark, we notice that, from a mathematical point of view, the use of closed-loop methods, as well as the fact that the proof of their effectiveness is based on argument of monotonicity, are certainly able to guarantee robustness of the proposed closed-loop algorithms with respect to several uncertainties present in the problem under study.
In particular, it is believed that the framework developed here could most certainly be extended to consider the effects of modeling and measurement errors, as well as imprecision and delay in the control-loop.



\vspace{1cm}
\section*{Acknowledgments}
Support from the Colciencias -- ECOS-Nord Program (Colombia: Project CI-71089; France: Project \textcolor{black}{C17M01}) is kindly acknowledged. DC and OV were supported by the inter-institutional cooperation program MathAmsud (18-MATH-05). This study was part of the Phase 2A {\color{black} `SIT feasibility project against \textit{Aedes albopictus} in Reunion Island'}, \tcm{jointly funded by the French Ministry of Health (Convention 3800/TIS) and the European Regional Development Fund (ERDF) (Convention No. 2012-32122 and Convention No. GURDTI 2017-0583-0001899). YD was (partially) supported by the DST/NRF SARChI Chair  in Mathematical Models and Methods in Biosciences and Bioengineering at the University of Pretoria (grant 82770). YD also acknowledges the support of the Visiting Professor Granting Scheme from the Office of the Deputy Vice-Chancellor for Research Office of the University of Pretoria.}

\bibliographystyle{siam}
\bibliography{biblio}

\begin{thebibliography}{10}

\bibitem{anguelov2012}
{\sc R.~Anguelov, Y.~Dumont, and J.~Lubuma}, {\em Mathematical modeling of
  sterile insect technology for control of \textit{Anopheles} mosquito},
  Comput. Math. Appl., 64 (2012), pp.~374--389.

\bibitem{Bliman2017}
{\sc P.-A. Bliman, M.~S. Aronna, F.~C. Coelho, and M.~A. H.~B. da~Silva}, {\em
  Ensuring successful introduction of \textit{Wolbachia} in natural populations
  of \textit{Aedes aegypti} by means of feedback control}, Journal of
  Mathematical Biology, 76 (2018), pp.~1269--1300.

\bibitem{Bourtzis2008}
{\sc K.~Bourtzis}, {\em \textit{Wolbachia}-based technologies for insect pest
  population control}, in Advances in Experimental Medicine and Biology,
  vol.~627, Springer, New York, NY, 02 2008.

\bibitem{Campo2017}
{\sc D.~E. Campo-Duarte, D.~Cardona-Salgado, and O.~Vasilieva}, {\em
  Establishing \textit{wMelPop Wolbachia} infection among wild \textit{Aedes
  aegypti} females by optimal control approach}, Applied Mathematics and
  Information Sciences, 11 (2017), pp.~1011--1027.

\bibitem{Campo2018}
{\sc D.~E. Campo-Duarte, O.~Vasilieva, D.~Cardona-Salgado, and M.~Svinin}, {\em
  Optimal control approach for establishing \textit{wMelPop Wolbachia}
  infection among wild \textit{Aedes aegypti} populations}, Journal of
  mathematical biology, 76 (2018), pp.~1907--1950.

\bibitem{Cooke99}
{\sc K.~Cooke, P.~van~den Driessche, and X.~Zou}, {\em Interaction of
  maturation delay and nonlinear birth in population and epidemic models},
  Journal of Mathematical Biology, 39 (1999), pp.~332--352.

\bibitem{dufourd2011}
{\sc C.~Dufourd and Y.~Dumont}, {\em Modeling and simulations of mosquito
  dispersal. the case of \textit{Aedes albopictus}}, Biomath, 1209262 (2012),
  pp.~1--7.

\bibitem{dufourd2013}
\leavevmode\vrule height 2pt depth -1.6pt width 23pt, {\em Impact of
  environmental factors on mosquito dispersal in the prospect of sterile insect
  technique control}, Comput. Math. Appl., 66 (2013), pp.~1695--1715.

\bibitem{dumont2012}
{\sc Y.~Dumont and J.~M. Tchuenche}, {\em Mathematical studies on the sterile
  insect technique for the {C}hikungunya disease and \textit{Aedes
  albopictus}}, Journal of Mathematical Biology, 65 (2012), pp.~809--855.

\bibitem{SIT}
{\sc V.~A. Dyck, J.~Hendrichs, and A.~S. Robinson}, {\em The Sterile Insect
  Technique, Principles and Practice in Area-Wide Integrated Pest Management},
  Springer, Dordrecht, 2006.

\bibitem{Farkas2017}
{\sc J.~Z. Farkas, S.~A. Gourley, R.~Liu, and A.-A. Yakubu}, {\em Modelling
  \textit{Wolbachia} infection in a sex-structured mosquito population carrying
  {W}est {N}ile virus}, Journal of Mathematical Biology, 75 (2017),
  pp.~621--647.

\bibitem{Farkas2010}
{\sc J.~Z. Farkas and P.~Hinow}, {\em Structured and unstructured continuous
  models for \textit{Wolbachia} infections}, Bulletin of Mathematical Biology,
  72 (2010), pp.~2067--2088.

\bibitem{Fen.Solving}
{\sc A.~Fenton, K.~N. Johnson, J.~C. Brownlie, and G.~D.~D. Hurst}, {\em
  {Solving the \textit{Wolbachia} paradox: modeling the tripartite interaction
  between host, \textit{Wolbachia}, and a natural enemy}}, The American
  Naturalist, 178 (2011), pp.~333--342.

\bibitem{Gouagna2015}
{\sc L.~Gouagna, J.~Dehecq, D.~Fontenille, Y.~Dumont, and S.~Boyer}, {\em
  Seasonal variation in size estimates of \textit{Aedes albopictus} population
  based on standard mark-release-recapture experiments in an urban area on
  {R}eunion {I}sland}, Acta Tropica, 143 (2015), pp.~89--96.

\bibitem{W}
{\sc M.~Hertig and S.~B. Wolbach}, {\em Studies on rickettsia-like
  micro-organisms in insects}, The Journal of medical research, 44 (1924),
  p.~329.

\bibitem{Li2017}
{\sc M.~Huang, X.~Song, and J.~Li}, {\em Modelling and analysis of impulsive
  releases of sterile mosquitoes}, Journal of Biological Dynamics, 11 (2017),
  pp.~147--171.

\bibitem{HugBri.Modeling}
{\sc H.~Hughes and N.~F. Britton}, {\em {Modeling the Use of \textit{Wolbachia}
  to Control Dengue Fever Transmission}}, Bull. Math. Biol., 75 (2013),
  pp.~796--818.

\bibitem{koiller2014}
{\sc J.~Koiller, M.~Da~Silva, M.~Souza, C.~Code{\c c}o, A.~Iggidr, and
  G.~Sallet}, {\em {Aedes, Wolbachia and Dengue}}, Research Report RR-8462,
  {Inria Nancy - Grand Est (Villers-l{\`e}s-Nancy, France)}, Jan. 2014.

\bibitem{Li2015}
{\sc J.~Li and Z.~Yuan}, {\em Modelling releases of sterile mosquitoes with
  different strategies}, Journal of Biological Dynamics, 9 (2015), pp.~1--14.

\bibitem{Moreira2009}
{\sc L.~A. Moreira, I.~Iturbe-Ormaetxe, J.~A. Jeffery, G.~Lu, A.~T. Pyke, L.~M.
  Hedges, B.~C. Rocha, S.~Hall-Mendelin, A.~Day, M.~Riegler, L.~E. Hugo, K.~N.
  Johnson, B.~H. Kay, E.~A. McGraw, A.~F. van~den Hurk, P.~A. Ryan, and S.~L.
  O'Neill}, {\em A \textit{Wolbachia} symbiont in \textit{Aedes aegypti} limits
  infection with dengue, chikungunya, and plasmodium}, Cell, 139 (2009),
  pp.~1268 -- 1278.

\bibitem{nadin2017}
{\sc G.~Nadin, M.~Strugarek, and N.~Vauchelet}, {\em Hindrances to bistable
  front propagation: application to \textit{Wolbachia} invasion}, Journal of
  Mathematical Biology, 76 (2018), pp.~1489--1533.

\bibitem{Perko}
{\sc L.~Perko}, {\em Differential Equations and Dynamical Systems},
  Springer-Verlag, 2006.

\bibitem{rasgon2003wolbachia}
{\sc J.~L. Rasgon and T.~W. Scott}, {\em \textit{Wolbachia} and cytoplasmic
  incompatibility in the {C}alifornia \textit{Culex pipiens} mosquito species
  complex: parameter estimates and infection dynamics in natural populations},
  Genetics, 165 (2003), pp.~2029--2038.

\bibitem{Schraiber2012}
{\sc J.~G. Schraiber, A.~N. Kaczmarczyk, R.~Kwok, M.~Park, R.~Silverstein,
  F.~U. Rutaganira, T.~Aggarwal, M.~A. Schwemmer, C.~L. Hom, R.~K. Grosberg,
  and S.~J. Schreiber}, {\em Constraints on the use of lifespan-shortening
  \textit{Wolbachia} to control dengue fever}, Journal of Theoretical Biology,
  297 (2012), pp.~26 -- 32.

\bibitem{Sinkin2004}
{\sc S.~P. Sinkins}, {\em \textit{Wolbachia} and cytoplasmic incompatibility in
  mosquitoes}, Insect Biochemistry and Molecular Biology, 34 (2004), pp.~723 --
  729.
\newblock Molecular and population biology of mosquitoes.

\bibitem{Smith1995}
{\sc H.~L. Smith}, {\em Monotone Dynamical Systems: An Introduction to the
  Theory of Competitive and Cooperative Systems}, Providence, R.I.: American
  Mathematical Society, 1995.

\bibitem{Strugarek2018b}
{\sc M.~{Strugarek}, H.~{Bossin}, and Y.~{Dumont}}, {\em On the use of the
  sterile insect release technique to reduce or eliminate mosquito
  populations}, Applied Mathematical Modelling,  (2018).

\bibitem{Strugarek2018}
{\sc M.~Strugarek, N.~Vauchelet, and J.~Zubelli}, {\em {Quantifying the
  survival uncertainty of \textit{Wolbachia}-infected mosquitoes in a spatial
  model}}, Mathematical Biosciences and Engineering, 15(4) (2018),
  pp.~961--991.

\end{thebibliography}
\newpage

\begin{appendix}
\appendix
\renewcommand{\theequation}{A-\arabic{equation}}
  \setcounter{equation}{0}  
\renewcommand{\thefigure}{A.\arabic{figure}}
  \setcounter{figure}{0}  
  \renewcommand{\thetable}{A.\arabic{table}}
   \setcounter{table}{0}  

\section*{Appendix: Proof of Lemma {\rm\ref{le0}}, page \pageref{le0}}
\label{seA}
First, it is easy to check that function $f$, defined in (\ref{eeq2}), page \pageref{eeq2}, is first decreasing and then increasing, and, thus, may solely have no root, one root or two zeros.

On the other hand, the number of roots of $f(x)=0$ is clearly {\em non-increasing} with respect to $a>0$: it has two roots for `small values' of $a$, no root for `large values' of $a$, and exactly one root for a certain critical value $a^{crit}$ separating the two previous regions.
This critical value is characterized by the fact that it possesses a double root $x^{crit}>0$, such that $f(x^{crit}) = f'(x^{crit}) =0$, that is:
\begin{equation}
\label{eeq3}
1+\frac{a^{crit}}{x^{crit}} = be^{-cx^{crit}},\qquad \frac{a^{crit}}{(x^{crit})^2} = bce^{-cx^{crit}}.
\end{equation}
Eliminating the exponential term in the previous formulas yields the second-order polynomial equation in $\dfrac{1}{x^{crit}}$
\[
\left(
\frac{1}{x^{crit}}
\right)^2-\frac{c}{x^{crit}} -\frac{c}{a^{crit}}=0.
\]
Its unique positive root is
\[
\frac{1}{x^{crit}} = \frac{c+\sqrt{c^2+4\dfrac{c}{a^{crit}}}}{2}
= \frac{c}{2}\left(
1+\sqrt{1+\frac{4}{a^{crit}c}}
\right),
\quad \text{ that is: }
\qquad
x^{crit} := \frac{2}{c}\frac{1}{1+\sqrt{1+\dfrac{4}{a^{crit}c}}}.
\]
Introducing this expression back in \eqref{eeq3}, leads to
$$
1+\frac{a^{crit}c}{2}\left(
1+\sqrt{1+\frac{4}{a^{crit}c}}
\right) = be^{-\frac{2}{1+\sqrt{1+\frac{4}{a^{crit}c}}}}.
$$
Thus $\phi^{crit}:= \dfrac{a^{crit}c}{2}$ is solution of \eqref{eeq1}, page \pageref{eeq1}, with $\phi^{crit}:= \dfrac{a^{crit}c}{2}$, such that, at the critical point, the parameters $a^{crit},b,c$ are interrelated.

For positive values of $a$ smaller than $a^{crit}$, the equation $f(x)=0$ has two roots, and no root whenever $a>a^{crit}$.
This achieves the proof of Lemma \ref{le0}.
\end{appendix}
\end{document}